\documentclass[a4paper]{article}

\usepackage[dvips]{graphics}
\usepackage[dvips]{graphicx}
\usepackage{amsfonts}
\usepackage{amssymb}
\usepackage{amsmath}
\usepackage{amstext}
\usepackage{amsbsy}
\usepackage{amsopn}
\usepackage{amsthm}
\usepackage{color}
\usepackage{amscd}
\usepackage{enumerate}
\usepackage{amsxtra}
\usepackage{mathrsfs}
\usepackage{upref}
\usepackage[colorlinks,
linkcolor=red,
anchorcolor=red,
citecolor=red
]{hyperref}
\usepackage{lineno}
\usepackage{geometry}
\newcommand{\A}{\mathcal{A}}
\newcommand{\EE}{\mathbb{E}}
\newcommand{\RR}{\mathbb{R}}
\newcommand{\PP}{\mathbb{P}}
\newcommand{\B}{\mathcal{B}}
\newcommand{\mC}{\mathcal{C}}

\newcommand{\nel}{\left\|}
\newcommand{\ner}{\right\|}

\newcommand{\norm}[1]{\lVert#1\rVert}
\newcommand{\Norm}[1]{\left\lVert#1\right\rVert}

\newtheorem{thm}{Theorem}[section]
\newtheorem{assu}{Assumption}[section]
\newtheorem{proposition}{Proposition}[section]
\newtheorem{lem}{Lemma}[section]
\newtheorem{cor}{Corollary}[section]
\theoremstyle{definition}
\newtheorem{defn}{Definition}[section]
\theoremstyle{remark}
\newtheorem{rmk}{Remark}[section]

\begin{document}
\title{On the mean-field limit for the \\ Vlasov-Poisson-Fokker-Planck system}

\author{Hui Huang\thanks{Department of Mathematics, Simon Fraser University, Burnaby, BC, Canada. Email: hha101@sfu.ca}, 
	Jian-Guo Liu\thanks{Departments of Physics and Mathematics,
		Duke University, Durham, NC, USA. Email: jliu@phy.duke.edu}, 
	Peter Pickl\thanks{Mathematisches Institut,
		Universit\"{a}t M\"{u}nchen, M\"{u}nchen, Germany. Email: pickl@math.lmu.de}
		\thanks{
		Duke Kunshan University, Kunshan, Jiangsu, China. Email: peter.pickl@duke.edu}
		}
\maketitle
\begin{abstract}
We rigorously justify the mean-field limit of an $N$-particle system subject to   Brownian motions and  interacting through the Newtonian potential  in $\RR^3$. Our result leads to a derivation of the Vlasov-Poisson-Fokker-Planck (VPFP) equations  from the regularized microscopic $N$-particle system. More precisely, we show that the maximal distance between the exact microscopic trajectories and the mean-field trajectories   is bounded by $N^{-\frac{1}{3}+\varepsilon}$ ($\frac{1}{63}\leq\varepsilon<\frac{1}{36}$) with a blob size of $N^{-\delta}$  ($\frac{1}{3}\leq\delta<\frac{19}{54}-\frac{2\varepsilon}{3}$) up to a probability of $1-N^{-\alpha}$ for any $\alpha>0$.  Moreover, we prove the convergence rate between the empirical measure associated to the regularized particle system and the solution of the VPFP equations. The technical novelty of this paper is that our estimates rely on the randomness coming from the initial data  and from the Brownian motions.
\end{abstract}

\section{Introduction}
Systems of interacting particles  are quite  common in physics and biosciences, and they are usually formulated according to  first principles (such as Newton's second law).  For instance, particles can represent galaxies in cosmological models \cite{aarseth2003gravitational}, molecules in a fluid \cite{jabin2004identification}, or ions and electrons in plasmas \cite{vlasov1968vibrational}. Such particle systems are also relevant as models for the collective behavior of certain animals like birds, fish, insects, and even micro-organisms (such as cells or bacteria) \cite{bolley2011stochastic,carrillo2010particle,motsch2011new}.  In this paper, we are interested in the classical Newtonian dynamics of $N$ indistinguishable particles interacting through pair  interaction forces and subject to  Brownian noise. Denote by $x_i\in \RR^3$ and $v_i\in\RR^3$ the position and velocity of particle $i$. The evolution of the system is given by the following stochastic differential equations (SDEs),
	\begin{equation}\label{IBM}
	dx_i=v_idt,\quad dv_i=\frac{1}{N-1}\sum\limits_{j\neq i}^Nk(x_i-x_j)dt+\sqrt{2\sigma}dB_i,\quad i=1,\cdots,N,
	\end{equation}
	where $k(x)$ models the pairwise interaction between the individuals, and
	$\{B_i\}_{i=1}^N$ are independent realizations of Brownian motions which count for extrinsic random perturbations such as random collisions against the background. In the presence of friction, model \eqref{IBM} is known as the interacting Ornstein-Uhlenbeck model in the probability or statistical mechanics community. In particular, we refer readers to \cite{olla1991scaling,tremoulet2002hydrodynamic} by Olla, Varadhan  and Tremoulet for the scaling limit of the Ornstein-Uhlenbeck system. In this manuscript, we take the interaction kernel to be  the Coulombian kernel
	\begin{equation}\label{couke}
	k(x) =   a \frac{x}{|x|^{3} }, 
	\end{equation}
	for some real number $a$. The case $a>0$ corresponds, for example,  to the electrostatic (repulsive) interaction of charged particles in a plasma, while the case $a<0$ describes the attraction between massive particles subject to gravitation.  We refer readers to \cite{jeans1915theory,vlasov1968vibrational} for the original modelings.

    Since the number $N$ of particles is large, it is extremely complicated to investigate the microscopic particle system \eqref{IBM} directly. Fortunately, it can be studied through macroscopic descriptions of the system based on the probability density for the particles on phase space.  These macroscopic descriptions are usually expressed as continuous partial differential equations (PDEs). The analysis of 
    the scaling limit of the interacting particle system to the macroscopic continuum model  is usually called the \textit{mean-field limit}.  For the second order particle system \eqref{IBM}, it is expected to be approximated by the following Vlasov-Poisson-Fokker-Planck (VPFP)  equations 
    \begin{equation}\label{vlasovoriginal}
    \left\{
    \begin{aligned}
    &\partial_t f(x,v,t)+v\cdot\nabla_x f(x,v,t)+k\ast\rho(x,t)\cdot\nabla_v f(x,v,t)=\sigma\Delta_v f(x,v,t),\\
    &f(x,v,0)=f_0(x,v),
    \end{aligned}
    \right.
    \end{equation}
    where $f(x,v,t):~(x,v,t) \in\RR^3\times\RR^3\times[0,\infty)\rightarrow \RR^+$ is the probability density function in the phase space $(x,v)$ at time $t$,  and 
    \begin{equation}\label{chargeden}
    \rho(x,t)=\int_{\RR^3} f(x,v,t)dv,
    \end{equation}
    is the charge density introduced by $f(x,v,t)$. We denote by $E(x,t):=k\ast\rho(x,t)$ the Coulombian or gravitational force field.

The intent of this research is to show the mean-field limit of the particle system \eqref{IBM} towards  the Vlasov-Poisson-Fokker-Planck equations \eqref{vlasovoriginal}. In particular, we quantify how close these descriptions are for a given $N$. 
Where $\sigma=0$ (there is no randomness coming from the noise), mean-field limit results for interacting  particle systems with globally Lipschitz forces have been obtained by Braun and Hepp \cite{braun1977vlasov} and Dobrushin \cite{dobrushin1979vlasov}.  Bolley, Ca\~{n}izo and Carrilo \cite{bolley2011stochastic} presented an extension of the classical theory to the particle system with only locally Lipschitz interacting force. Such case concerning kernels  $k\in W_{loc}^{1,\infty}$ are also used in the context of neuroscience \cite{bossy2015clarification,touboul2014propagation}. 
 The last few years have seen great progress in  mean-field limits for singular forces by treating them with an $N$-dependent cut-off. In particular, Hauray and Jabin \cite{jabin2015particles} discussed mildly singular force kernels satisfying $|k(x)|\leq \frac{C}{|x|^\alpha}$ with $\alpha<d-1$ in  dimensions $d\geq3$. For $1<\alpha<d-1$, they performed the mean-field limit for typical initial data, where they chose the  cut-off  to be $N^{-\frac{1}{2d}}$.  For $\alpha< 1$, they prove molecular chaos without cut-off. Unfortunately, their method fails precisely at the Coulomb threshold when $\alpha= d-1$. More recently, Boers and Pickl \cite{boers2016mean} proposed a novel method for deriving mean-field equations 
with interaction forces scaling like $\frac{1}{|x|^{3\lambda-1}}$ $(5/6<\lambda<1)$, and they were able to obtain a cut-off  as small as $ N^{-\frac{1}{d}}$. Furthermore, Lazarovici and Pickl \cite{lazarovici2015mean} extended the method in \cite{boers2016mean} to include the Coulomb singularity and they obtained a microscopic derivation of the Vlasov-Poisson equations with a cut-off of $N^{-\delta}$  $(0<\delta<\frac{1}{d})$.  More recently, the cut-off parameter was reduced to as small as $N^{-\frac{7}{18}}$ in \cite{grass2019microscopic} by using the second order nature of the dynamics.
Where $\sigma>0$, the random particle method for approximating the VPFP system with the Coulombian kernel was studied in \cite{havlak1996numerical}, where the initial data was chose on a mesh and the cut-off parameter can be $N^{-\delta}$ $(0<\delta<\frac{1}{d})$.  Most recently, Carrilo $et.al.$ \cite{carrillo2018propagation}  also investigated the singular VPFP system but with  the i.i.d. initial data, and obtained the propagation of chaos through a cut-off of $N^{-\delta}$  $(0<\delta<\frac{1}{d})$, which was a generalization of \cite{lazarovici2015mean}. We also note that
Jabin and Wang \cite{jabin2016mean} rigorously justified the mean-field limit and propagation of chaos for the Vlasov systems with $L^\infty$ forces and vanishing viscosity ($\sigma_N\rightarrow 0$ as $N\rightarrow \infty$) by using a relative entropy method.  Lastly, for a general overview of this topic we refer readers to \cite{carrillo2010particle,jabin2014review,jabin2017mean,spohn2004dynamics}.

When the interacting kernel $k$ is singular, it poses problems for both theory and  numerical simulations. An easy remedy is to regularize the force with an $N$-dependent cut-off parameter and get $k^N$. The delicate question is how to choose this cut-off. On  the one hand,  the larger the cut-off is,  the smoother $k^N$ will be  and the easier it will be to show the convergence. However, the regularized system is not a good approximation of the actual system. On the other hand, the smaller the cut-off is, the closer $k^N$ is to the real $k$, thus the less information will be lost through the cut-off. Consequently, the necessary balance between accuracy (small cut-off) and regularity (large cut-off) is crucial. The analyses we reviewed above tried to justify that. In this manuscript, we set $\sigma>0$.  Compared with the recent work \cite{carrillo2018propagation}, the main technical innovation of this paper is that we fully use the randomness coming from the initial conditions {\it and} the Brownian motions to significantly improve the cut-off. Note that in \cite{carrillo2018propagation} the size of cut-off can be very close to but larger than $N^{-\frac{1}{d}}$. However we manage to reduce the cut-off size  to be smaller than $N^{-\frac{1}{d}}$ (see Remark \ref{remark}), which is a sort of average minimal distance between $N$ particles in dimension $d$.  This manuscript significantly improves the ideas presented in \cite{garcia2017}. There the potential is split up into a more singular and less singular part.  The less singular part is controlled in the usual manner while the mixing coming from the Brownian motion is used to estimate the more singular part. The technical innovation in the present paper is that the possible number of particles subject to the singular part of the interaction can be bounded due to the fact that the support of the singular part is small using a Law of Large Numbers argument. Again using the Law of Large Numbers based on the randomness coming from the Brownian motion, we show that the leading order of the singular part of the interaction can be replaced by its expectation value.  This step is a key point of the present manuscript. The replacement by the expectation value, i.e. the integration of the  force against the probability density, gives the regularization of the singular part and gives a significant improvement of our estimates. This is carried out in  Lemma \ref{tildaXficedtime}, the proof of which can be found in section \ref{prolem}.
 \cite{garcia2017} and the present paper are, to our knowledge, so far the only results where the mixing from the Brownian motion has been used in the derivation of a mean-field limit for an interacting many-body system.

As a companion of \eqref{IBM}, some also consider the first order stochastic system
\begin{equation}\label{IBM1}
dx_i=\frac{1}{N-1}\sum\limits_{j\neq i}^Nk(x_i-x_j)dt+\sqrt{2\sigma}dB_i,\quad i=1,\cdots,N.
\end{equation}
As before, one can expect that as the number of the particles $N$ goes to infinity we can get the continuous description of the dynamics as the following nonlinear PDE
\begin{equation}
\partial_t f(x,t)+\nabla\cdot (f(k\ast f))=\sigma\Delta_xf\,,
\end{equation}
where $f(x,t)$ is now the spatial density.

The particle system \eqref{IBM1} has many important applications. One of the  best known classical applications is in  fluid dynamics with the Biot-Savart kernel
\begin{equation}
k(x)=\frac{1}{2\pi}(\frac{-x_2}{|x|^2},\frac{x_1}{|x|^2})\,.
\end{equation}
It can be treated by the well-known vortex method introduced by Chorin in 1973 \cite{chorin1973numerical}. The convergence of the vortex method for two and three dimensional inviscid ($\sigma=0$) incompressible fluid flows was first proved by Hald $et\,al.$ \cite{hald,HOH}, Beale and Majda \cite{BM,BM1}. When the effect of viscosity  is involved ($\sigma>0$), the vortex method is replaced by the so called random vortex method by adding a Brownian motion to every vortex. The convergence analysis of the random vortex method for the Navier-Stokes equation was given by \cite{GJ,LD,MCPM,OH} in the 1980s. For more recent results we refer to \cite{duerinckx2016mean,fournier2014propagation,jabin2018quantitative,serfaty2018mean}.
Another well-known application of the system \eqref{IBM1} is to choose the interaction to be the Poisson kernel 
\begin{equation}
k(x)=-C_d\frac{x}{|x|^d},\quad d\geq 2\,,
\end{equation}
where $C_d>0$ and $k$ is set to be attractive.
Now the system \eqref{IBM1} coincides with the particle models to approximate the classical Keller-Segel (KS) equation for chemotaxis \cite{keller1970initiation,PCS}.  We mainly refer  to \cite{garcia2017,fetecau2018propagation,HH2,HH1,huilearning,liu2016propagation,YZ} for the mean-field limit of the KS system. Concerning the size of the cut-off, more specifically, \cite{liu2016propagation} chose the cut-off to be $(\ln N)^{-\frac{1}{d}}$, which was significantly improved in \cite{HH2}, where the cut-off size can be as small as $N^{-\frac{1}{d(d+1)}}\log(N)$.  In \cite{garcia2017,HH1}, the cut-off size was almost optimal, coming fairly close to $N^{-\frac{1}{d}}$.  Many techniques  used in this manuscript are adapted from these papers. For the Poisson-Nernst-Planck equation ($k$ is set to be repulsive),  \cite{liu2016propagation} proved the mean-field limit  without a cut-off.

The rest of the introduction  will be split into three parts: We start with introducing the microscopic random particle system in Section \ref{intro1}. Then we present some results on the existence of  the macroscopic mean-field VPFP equations in Section \ref{intro2}. Lastly, our main theorem will be stated in Section \ref{intro3}, where we prove the closeness of the approximation of solutions to VPFP equations by the microscopic system.
\subsection{Microscopic random particle system}\label{intro1}
We are interested in the time evolution of a system of $N$-interacting Newtonian particles with noise in the $N\to\infty$ limit. The  motion of the system studied in this paper is described by trajectories on phase space, i.e. a time dependent $\Phi_t:\mathbb{R}\to\mathbb{R}^{6N}$. We use the notation
\begin{equation}\label{Phit}
\Phi_t:=\left(X_t,V_t\right):=\left(x^t_1,\ldots, x^t_N,v^t_1,\ldots v^t_N\right),
\end{equation}
where $x^t_j$ stands for the position of the $j^{\text{th}}$ particle at time $t$
and  $v^t_j$ stands for the velocity of the $j^{\text{th}}$ particle at time $t$. 
 The system is a Newtonian system with a noise term coupled to the velocity, whose evolution is governed by a system of SDEs of the type
 \begin{align}\label{eq:motion}
 \left\{
 \begin{aligned}
 &d x_i^t=v_i^tdt,\quad i=1,\cdots,N\,,\\
 &d v_i^t=\frac{1}{N-1}\sum_{j\neq i}^Nk(x_i^t-x_j^t)dt + \sqrt{2\sigma}dB_i^t\;,
\end{aligned}
\right.
 \end{align}
 where $k$ is the Coulomb kernel \eqref{couke} modeling interaction between particles and $B_i^t$ are independent realizations of Brownian motions. 
 
 We regularize the kernel $k$ by a blob function $0\leq \psi(x)\in C^2(\RR^3)$, $\mbox{supp }\psi(x)\subseteq B(0,1)$ and $\int_{\RR^3}\psi(x)dx=1$. Let $\psi_\delta^N=N^{3\delta}\psi(N^{\delta}x)$,
 then the Coulomb kernel with regularization has the form
\begin{equation}\label{eq:force}
k^N(x) =  k \ast \psi_\delta^N.
\end{equation}
Thus one has the regularized microscopic $N$-particle system for $i=1,2\cdots,N$
 \begin{align}\label{eq:regpar}
 \left\{
 \begin{aligned}
& d x_i^t=v_i^t dt,\\
 &d v_i^t=\frac{1}{N-1}\sum_{i\neq j}^Nk^N(x_i^t-x_j^t)dt + \sqrt{2\sigma}dB_i^t\;.
 \end{aligned}
 \right.
 \end{align}
 Here the initial condition $\Phi_0$ of the system is independently, identically  distributed  (i.i.d.) with the common probability density given by $f_0$. 
 And the corresponding  regularized VPFP equations are
 \begin{equation}\label{vlasov}
 \left\{
 \begin{aligned}
 &\partial_t f^N(x,v,t)+v\cdot\nabla_x f^N(x,v,t)+k^N\ast\rho^N(x,t)\cdot\nabla_v f^N(x,v,t)=\sigma\Delta_v f^N(x,v,t),\\
 &\rho^N(x,t)=\int_{\RR^3}f^N(x,v,t)dv,\\
 &f^N(x,v,0)=f_0(x,v).
 \end{aligned}
 \right.
 \end{equation}
 
\subsection{Existence of classical solutions to the Vlasov-Poisson-Fokker-Planck system}\label{intro2}


The existence of weak and classical solutions to VPFP equations \eqref{vlasovoriginal} and related systems has been very well studied. Degond \cite{degond} first showed the existence of a global-in-time smooth solution for the Vlasov-Fokker-Planck equations in one and two space dimensions in the electrostatic
case. Later on, Bouchut \cite{BO,BO1} extended the result to three dimensions when the electric field was coupled through a Poisson equation, and the results were given in both the electrostatic and gravitational case. Also, Victory and O'Dwyer \cite{V-O-B} showed existence of classical solutions for VPFP equations when the spacial dimension is less than or equal to two, and local existence  for all other dimensions.  Then Bouchut in  \cite{BO} proved the global existence of classical solutions for the VPFP system  \eqref{vlasovoriginal} in dimension $d=3$.  His proof relied on the  techniques introduced by Lions and Perthame \cite{lions1991propagation} concerning the existence to the Vlasov-Poisson system in three dimensions.
The long time behavior of the VPFP system was studied by Ono and Strauss \cite{ono2000regular}, Carpio \cite{carpio1998long} and Carrillo $et\,al.$ \cite{soler1997asymptotic}.

The existence results in \cite{V-O-B} and  \cite{BO} are most appropriate for this work. We summarize them in the following theorem, which is also used in \cite[Theorem 2.1]{havlak1996numerical}.
\begin{thm}\label{existence} (Classical solutions of the VPFP equations)
	Let the initial data $0\leq f_0(x,v)$ satisfies the following properties:
	\begin{enumerate}[a)]
		\item $f_0\in W^{1,1}\cap W^{1,\infty}(\RR^6)$;
		\item there exists a $m_0>6$, such that
			\begin{equation}
			(1+|v|^2)^{\frac{m_0}{2}}f_0\in W^{1,\infty}(\RR^6)\,.
			\end{equation}
	\end{enumerate}
Then for any $T>0$, the VPFP equations \eqref{vlasovoriginal} admits a unique  classical solution on $[0,T]$.
\end{thm}
\begin{rmk}
	The proof of the above theorem given in  \cite{V-O-B} and  \cite{BO}  indicates that the map
	\begin{equation}
	t\rightarrow E(\cdot,t):=F\ast\rho(\cdot,t)\,,
	\end{equation}
	is a continuous map from $[0,T]$ to $W^{1,\infty}(\RR^3)$.  This implies that initial smooth data remains smooth for all time intervals $[0,T]$.  So if we assume the initial data satisfies  the following for any $k\geq 1$
	\begin{enumerate}[a)]
		\item $f_0\in W^{k,1}\cap W^{k,\infty}(\RR^6)$;
		\item there exists a $m_0>6$, such that
		\begin{equation}
		(1+|v|^2)^{\frac{m_0}{2}}f_0\in W^{k,1}\cap W^{k,\infty}(\RR^6)\,.
		\end{equation}
	\end{enumerate}
	Then the unique classical solution $f$ maintains the regularity on $[0,T]$ for any $k\geq 1$:
	\begin{equation}
	\max\limits_{0\leq t\leq T}\norm{(1+|v|^2)^{\frac{m_0}{2}}f_t}_{W^{k,1}\cap W^{k,\infty}(\RR^6)}<\infty\,.
	\end{equation}
\end{rmk}

The present paper also needs the uniform-in-time $L^\infty$  bound of the charge density $\rho$:
\begin{equation}\label{Linfrho}
\max\limits_{0\leq t\leq T}\norm{\rho(\cdot,t)}_{W^{1,\infty}(\RR^3)}<\infty\,,
\end{equation}
which was obtained in \cite{pulvirenti2000infty} by means of the stochastic characteristic method under the assumption the $f_0$ is compactly supported in velocity. We also note that  \cite{carrillo2018propagation} provided a proof of the local-in-time $L^\infty$ bound for $\rho$  by employing Feynman-Kac's formula and assuming the initial data has polynomial decay.

In this paper, we assume that the initial data $f_0$ satisfies the following assumption:
\begin{assu}\label{assum}
	The initial data $0\leq f_0(x,v)$ satisfies
	\begin{enumerate}
		\item  $f_0\in W^{1,1}\cap W^{1,\infty}(\RR^6)$;
			\item there exists a $m_0>6$, such that
		\begin{equation}
		(1+|v|^2)^{\frac{m_0}{2}}f_0\in  W^{1,1}\cap W^{1,\infty}(\RR^6);
		\end{equation}
		\item $f_0(x,v)=0$ when $|v|>Q_v$.
	\end{enumerate}
\end{assu}

The above assumption makes sure that we have the regularity needed for this article:
for any $T>0$, 
\begin{align}\label{regularity}
\max\limits_{t\in[0,T]}\norm{\rho(\cdot,t)}_{W^{1,\infty}(\RR^3)}&+\max\limits_{0\leq t\leq T}\norm{(1+|v|^2)^{\frac{m_0}{2}}f_t}_{W^{1,1}\cap W^{1,\infty}(\RR^6)}\leq C_{f_0},
\end{align}
where $C_{f_0}$ depends only on $\|f_0\|_{W^{1,1}\cap W^{1,\infty}(\RR^6)}$, $\norm{(1+|v|^2)^{\frac{m_0}{2}}f_0}_{W^{1,1}\cap W^{1,\infty}(\RR^6)}$ and $Q_v$.   Note that the charge density $\rho$ satisfies
\begin{equation}
\partial_t\rho(x,t)+\int_{\RR^3}v\cdot \nabla_xf(x,v,t)dv=0\,.
\end{equation}
Thus we have
\begin{align}\label{partialtrho}
\max\limits_{t\in[0,T]}\norm{\partial_t\rho(\cdot,t)}_{L^{\infty}(\RR^3)}&\leq \max\limits_{t\in[0,T]}\Norm{\int_{\RR^3}|v|| \nabla_xf(x,v,t)|dv}_{L^{\infty}(\RR^3)}\notag \\
&\leq \max\limits_{t\in[0,T]}\norm{(1+|v|^2)^{\frac{m_0}{2}}f_t}_{W^{1,\infty}(\RR^6)}\int_{\RR^3}\frac{|v|}{(1+|v|^2)^{\frac{m_0}{2}}}dv\leq C_{f_0}\,.
\end{align}
We also note that equivalently one can estimate a bound   for $f^N$ and $\rho^N$ uniformly in $N$.  
\begin{rmk}
	The assumption that $f_0$ is compactly supported in the velocity variable is not required for the existence of the VPFP system. However it is used to get the $L^\infty$  bound of the charge density $\rho$ (see in \cite{pulvirenti2000infty}) and also in the proof of Lemma \ref{lmQ} (see in \eqref{compactv}). 
\end{rmk}

\begin{rmk}
	All our estimates below are also possible in the presence of sufficiently smooth external fields. 
	Due to the fluctuation-dissipation principle it is more natural to add an external, velocity-dependent friction force to the system. 
	\end{rmk}

\subsection{Statement of the main results}\label{intro3}

Our objective is to derive the macroscopic mean-field PDE \eqref{vlasovoriginal} from the regularized microscopic particle system  \eqref{eq:regpar}. We will do this by using probabilistic methods as in \cite{garcia2017,HH1,HH2,lazarovici2015mean}.  More precisely, we shall prove the convergence rate between the solution of VPFP equations \eqref{vlasovoriginal} and the empirical measure associated to the regularized particle system $\Phi_t$ satisfying \eqref{eq:regpar}. We assume that the initial condition $\Phi_0$ of the system is independently, identically distributed (i.i.d.) with the common probability density given by $f_0$. 

Given the solution $f^N$ to the mean-field equation \eqref{vlasov}, we first construct an auxiliary trajectory $\Psi_t$ from  \eqref{vlasov}. Then we  prove the closeness between $\Phi_t$ and $\Psi_t$. 
For the auxiliary trajectory 
\begin{equation}\label{Psi}
\Psi_t:=\left(\overline X_t,\overline V_t\right)=\left(\overline x^t_1,\ldots,\overline x^t_N,\overline v^t_1,\ldots \overline v^t_N\right),
\end{equation}
we shall consider again a Newtonian system with noise, however, this time not subject to the pair interaction but under the influence of the external mean field $k^N\ast \rho^N(x,t)$
\begin{align}\label{eq:mean}
\left\{
\begin{aligned}
&d\overline x_i^t=\overline v_i^t dt,\quad i=1,\cdots,N\,,\\
&d\overline v_i^t=\int_{\RR^3} k^N(\overline x_i^t-x)\rho^N(x,t)dxdt+\sqrt{2\sigma}dB_i^t\,.
\end{aligned}
\right.
\end{align}
Here we let $\Psi_t$ have the same initial condition as $\Phi_t$ (i.i.d. with the common density $f_0$).
Since the particles are just $N$ identical copies of evolution, the independence is conserved. Therefore the $\Psi_t$ are distributed i.i.d.  according to the common probability density $f^N$. We remark  that the VPFP equation \eqref{vlasov} is Kolmogorov's forward equation for any solution of \eqref{eq:mean}, and in particular their probability density $f^N$ solves \eqref{vlasov}. This i.i.d. property will play a crucial role below, where we shall use the concentration inequality (see in Lemma \ref{central}) on some functions depending on $\Psi_t$.

Our main result states that the $N$-particle trajectory $\Phi_t$ starting from $\Phi_0$ (i.i.d. with the common density $f_0$) remains close to the mean-field trajectory $\Psi_t$ with the same initial configuration $\Phi_0=\Psi_0$ during any finite time $[0,T]$. More precisely, we prove that the measure of the set where the maximal distance $\max\limits_{t\in[0,T]}\norm{\Phi_t-\Psi_t}_\infty$ on $[0,T]$ exceeds $N^{-\lambda_2}$ decreases exponentially as the number $N$ of particles grows to infinity. Here the distance $\norm{\Phi_t-\Psi_t}_\infty$ is measured by
\begin{equation}\label{distance}
\norm{\Phi_t-\Psi_t}_\infty:=\sqrt{\log(N)}\|X_t-\overline X_t\|_\infty+ \|V_t-\overline V_t\|_\infty.
\end{equation}

\begin{thm}\label{mainthm}
For any $T>0$, assume that trajectories $\Phi_t=(X_t,V_t)$, $\Psi_t=(\overline X_t,\overline V_t)$ satisfy \eqref{eq:regpar} and \eqref{eq:mean} respectively with the initial data $\Phi_0=\Psi_0$, which is i.i.d. sharing the common density $f_0$ that satisfies Assumption \ref{assum}.  Then for any $\alpha>0$ and $0<\lambda_2<\frac{1}{3}$,  there exists some $0<\lambda_1<\frac{\lambda_2}{3}$ and a $N_0\in \mathbb{N}$ which both depend only on $\alpha$, $T$ and $C_{f_0}$, such that for $N\geq N_0$, the following estimate holds with the cut-off index $\delta\in\left[\frac{1}{3},\min\left\{\frac{\lambda_1+3\lambda_2+1}{6},\frac{1-\lambda_2}{2}\right\}\right)$
$$\mathbb{P}\left(\max\limits_{t\in[0,T]}\norm{\Phi_t-\Psi_t}_\infty \leq N^{-\lambda_2} \right)\geq 1-N^{-\alpha},$$
where $\norm{\Phi_t-\Psi_t}_\infty$ is defined in \eqref{distance}.
\end{thm}
\begin{rmk}\label{remark}
	In particular, for any $\frac{1}{63}\leq\varepsilon<\frac{1}{36}$, choosing $\lambda_2=\frac{1}{3}-\varepsilon$ and $\lambda_1=\frac{1}{9}-\varepsilon$, we have a convergence rate $N^{-\frac{1}{3}+\varepsilon}$ with a cut-off size of $N^{-\delta}$ $(\frac{1}{3}\leq\delta<\frac{19}{54}-\frac{2\varepsilon}{3})$. In other words, the cut-off parameter $\delta$ can be chosen very close to $\frac{19}{54}$ and in particular larger than $\frac{1}{3}$, which is a significant improvement over  previous results in the literature.
\end{rmk}

 \textbf{Strategy of the proof.}  The strategy is to obtain a Gronwall-type inequality for $\max\limits_{t\in[0,T]}\norm{\Phi_t-\Psi_t}_\infty $. Notice that
\begin{equation*}
\frac{d (V_t-\overline V_t)}{d t}=K^N(X_t)-\overline K^N(\overline X_t),
\end{equation*}
where $K^N(X_t)$  and $\overline K^N(\overline X_t)$ are defined as
\begin{equation}
(K^N(X_t))_i:=\frac{1}{N-1}\sum_{j\neq i}k^N(x_i^t-x_j^t);\quad (\overline K^N(\overline X_t))_i:=\int_{\RR^3} k^N(\overline x_i^t-x)\rho^N(x,t)dx.
\end{equation}

One can compute
	\begin{align} \label{strategy}
	&\frac{d \norm{\Phi_t-\Psi_t}_\infty}{dt} \leq\sqrt{\log(N)}\nel V_t-\overline V_t\ner_\infty+\nel K^N(X_t)-\overline K^N(\overline X_t)\ner_\infty  \notag \\
	\leq &\sqrt{\log(N)}\nel V_t-\overline V_t\ner_\infty+\nel K^N(X_t)- K^N(\overline X_t)\ner_\infty +\nel K^N(\overline X_t)-\overline K^N(\overline X_t)\ner_\infty \,.  
	\end{align}
	If the force $k^N$ is Lipschitz continuous with a Lipschitz constant independent of $N$, the desired convergence follows easily \cite{braun1977vlasov,dobrushin1979vlasov}. However the force considered here becomes singular as $N\rightarrow \infty$, hence it does not satisfy a uniform Lipschitz bound.
	
	The first term in \eqref{strategy} is already a sufficient bound in view of Gronwall's Lemma.
	
	$\bullet$ By the Law of Large Numbers, carried out in detail for our purpose here in Proposition \ref{propconsis}, we show for any $T>0$
	\begin{equation}\label{strategy1}
	\max\limits_{t\in[0,T]}\nel K^N(\overline X_t)-\overline K^N(\overline X_t)\ner_\infty \preceq CN^{2\delta-1}\log(N)\,,
	\end{equation}
	where for convenience we  abused the notation $a\preceq b$ to denote $a\leq b$ except for an event with probability approaching zero.
	
	This direct error estimate can be seen as a consistency of the two evolutions in high probability.
	
	$\bullet$ In Proposition \ref{propstab}, we show that the propagation of errors, coming from the second term in \eqref{strategy}, is stable. This  stability is important to be able to close the Gronwall argument. We show that  for any $T>0$
	\begin{equation}\label{strategy2}
	\| K^N(X_t)-K^N(\overline X_t) \|_\infty \preceq C\log(N)\nel X_t-\overline X_t\ner_\infty+C\log^2(N)N^{-\lambda_3}\,, \mbox{ for all }t\in[0,T]\,,
	\end{equation}
	holds under the condition that
	\begin{equation}\label{stabcond}
	\max\limits_{t\in[0,T]}\norm{\Phi_t-\Psi_t}_\infty\preceq N^{-\lambda_2}\,.
	\end{equation}
	Here it is crucial to ensure the constant $\lambda_3$ satisfies $2\delta-1\leq -\lambda_3 <-\lambda_2<0$. The function of this additional condition will be clear later (see Remark \ref{ATexplain}).
	
	To get this improvement of the cutoff parameter compared to previous results in the literature, we make use of the mixing caused by the Brownian motion. Therefore we  split potential $K^N:=K_1^N+K_2^N$, where $K_2^N$ is chosen to have a wider cut-off of order $N^{-\lambda_2}>N^{-\delta}$. The less singular part $K_2^N$ is controlled in the usual manner \cite{boers2016mean,HH1,HH2,lazarovici2015mean} (see in estimate \eqref{I1'}).
	\begin{equation}
	\| K_2^N(X_t)-K_2^N(\overline X_t) \|_\infty \preceq C\log(N)\nel X_t-\overline X_t\ner_\infty.
	\end{equation}  
	
	Thus we are left with the force $K_1^N$. We shall first estimate the number of particles that will be present in the support of $K_1^N$. Since the latter is small, this number will always be very small compared to $N$.
	
	Under the condition \eqref{stabcond} we can not track the particles of the Newtonian time evolution with an accuracy larger than $N^{-\lambda_2}$. Thus -- without using the Brownian motion -- we have to assume the worst case scenario, which is all particles giving the maximal possible solution to the force, i.e. sitting close to the edge of the cutoff region, and all forces summing up, i.e. all particles sitting on top of each other.
	
	But the Brownian motion in our system will lead to mixing. For a short time interval the effect of mixing will be much larger than the effect of the correlations coming from the pair interaction, and we can make use of the independence of the Brownian motions. This mixing, which happens on a larger spacial scale than the range of the potential, causes the particles to be distributed  roughly  equally over the support of the interaction resulting in a cancellation of  the leading order of  $K_1^N$.

	It follows for the more singular part $K_1^N$ that
		\begin{equation}
	\| K_1^N(X_t)-K_1^N(\overline X_t) \|_\infty\preceq C\log^2(N)N^{-\lambda_3}\,,
	\end{equation}

  This is mainly carried out  in Section \ref{prolem} (the proof of Lemma \ref{tildaXficedtime}).
	
	$\bullet$ Combining consistency \eqref{strategy1} and stability \eqref{strategy2}, we conclude that for any $0< T_1\leq T$
	\begin{equation}
	\frac{d \norm{\Phi_t-\Psi_t}_\infty}{dt} \preceq C\log(N)\nel X_t-\overline X_t\ner_\infty+C\log^2(N)N^{-\lambda_3}\,, \mbox{ for all }t\in(0,T_1]\,,
	\end{equation}
 holds	provided that
		\begin{equation}
	\max\limits_{t\in[0,T_1]}\norm{\Phi_t-\Psi_t}_\infty\preceq N^{-\lambda_2}\,,
	\end{equation}
	where $-\lambda_3<-\lambda_2<0$.
	This implies a generalized Gronwall's inequality (see Lemma \ref{lmprior}), which leads to
	\begin{equation}
    \max\limits_{t\in[0,T]}\norm{\Phi_t-\Psi_t}_\infty\preceq N^{-\lambda_2}\,.
	\end{equation}
	Hence it completes our proof.

To quantify the convergence of probability measures, we give a brief introduction on the topology of the $p$-Wasserstein space.  In the context of kinetic equations, it was first introduced by Dobrushin \cite{dobrushin1979vlasov}.
Consider the following probability space with finite $p$-th moment:
\begin{equation}
\mathcal{P}_p(\mathbb{R}^{d})=\big\{\mu|~\mu \mbox{ is a probability measure on } \mathbb{R}^{d} \mbox{ and } \int_{\mathbb{R}^{d}}|x|^p d\mu(x)<+\infty\big\}.
\end{equation}
We denote the Monge-Kantorovich-Wasserstein distance in $\mathcal{P}_p(\mathbb{R}^{d})$ as follows
\begin{equation}
W_p^p(\mu,\nu)=\inf_{\pi\in\Lambda(\mu,~\nu)}\Big\{\int_{\mathbb{R}^{d}\times\mathbb{R}^{d}}|x-y|^pd\pi(x,y)\Big\}=\inf_{X\sim \mu,Y\sim \nu}\Big\{\EE[|X-Y|^p]\Big\},
\end{equation}
where $\Lambda(\mu,~\nu)$ is the set of joint probability measures on ${\mathbb{R}^{d}}\times{\mathbb{R}^{d}}$ with marginals $\mu$ and $\nu$ respectively and $(X,Y)$ are all possible couples of random variables with $\mu$ and $\nu$ as respective laws.  For notational simplicity, the notation for a probability measure and its probability density is often abused. So if $\mu,\nu$ have densities $\rho_1,\rho_2$ respectively, we also denote the distance as $W_p^p(\rho_1,\rho_2)$.  For further details, we refer the reader to the book of Villani \cite{villani2008optimal}.

Following the same argument as \cite[Corollary 4.3]{lazarovici2015mean}, Theorem \ref{mainthm} implies molecular chaos in the following sense:
\begin{cor}
	For any $T>0$, let $F_0^N:=\otimes^N f_0$ and $F_t^N$ be the $N$-particle distribution evolving with the microscopic flow \eqref{eq:regpar} starting from $F_0^N$. Then the $k$-particle marginal
\begin{equation*}
^{(k)}F_t^N(z_1,\cdots,z_k):=\int F_t^N(Z)dz_{k+1}\cdots dz_{N}
\end{equation*}
converges weakly to $\otimes^kf_t$ as $N\rightarrow\infty$ for all $k\in N$, where $f_t$ is the unique solution of the VPFP equations \eqref{vlasovoriginal} with $f_t|_{t=0}= f_0$. More precisely, under the assumptions of Theorem \ref{mainthm}, for any $\alpha>0$, there exists some constants $C>0$ and $N_0>0$ depending only on $\alpha$, $T$ and $C_{f_0}$, such that for $N\geq N_0$, the following estimate holds
\begin{equation*}
W_1\left(^{(k)}F_t^N,\otimes^kf_t\right)\leq k\exp\left(TC\sqrt{\log (N)}\right)N^{-\lambda_2},\quad\forall\,~ 0\leq t\leq T,
\end{equation*}
where $\lambda_2$ is used in Theorem \ref{mainthm}.
\end{cor}

Another result from Theorem \ref{mainthm} is the derivation of the macroscopic mean-field VPFP equations \eqref{vlasovoriginal} from the microscopic random particle system \eqref{eq:regpar}. We define the empirical measure associated to the microscopic $N$-particle systems \eqref{eq:regpar} and \eqref{eq:mean} respectively as 
\begin{equation}\label{empirical}
\mu_\Phi(t):=\frac{1}{N}\sum_{i=1}^{N}\delta(x-x_i^t)\delta(v-v_i^t),\quad \mu_\Psi(t):=\frac{1}{N}\sum_{i=1}^{N}\delta(x-\overline x_i^t)\delta(v-\overline v_i^t).
\end{equation}

The following  theorem shows that under additional moment control assumptions on $f_0$, the empirical measure $\mu_\Phi(t)$ converges to the solution of VPFP equations \eqref{vlasovoriginal} in $W_p$ distance with high probability.
\begin{thm}\label{cor} Under the same assumption as in Theorem \ref{mainthm}, let $f_t$ be the unique solution to the VPFP equations \eqref{vlasovoriginal} with the initial data satisfying Assumption \ref{assum} and $\mu_\Phi(t)$ be the empirical measure defined in \eqref{empirical} with $\Phi_t$ being the particle flow solving \eqref{eq:regpar}. Let $p\in[1,\infty)$ and assume that there exists $m>2p$ such that $\iint_{\RR^6}|x|^mf_0(x,v)dxdv<+\infty$. Then for any $T>0$ and $\kappa<\min\{\frac{1}{6},\frac{1}{2p},\delta\}$, there exists a constant $C_1$ depending only on  $T$ and $C_{f_0}$ and constants $C_2$, $C_3$ depending only on $m$, $p$, $\kappa$,  such that for all $N\geq e^{\left(\frac{C_1}{1-3\lambda_2}\right)^2}$ it holds that
\begin{align}
\PP\bigg(\max\limits_{t\in[0,T]}W_p(f_t,\mu_\Phi(t))&\leq N^{-\kappa+1-3\lambda_2}+N^{-\lambda_2}\bigg)\notag\\
&\geq 1-C_2\left(e^{-C_3N^{1-\max\{6,2p\}\kappa}}+N^{1-\frac{m}{2p}}\right).
\end{align}
 where $\delta$ and $\lambda_2$ are used in Theorem \ref{mainthm}.
\end{thm}
This theorem provides a derivation of the  VPFP equations from an interacting $N$-particle system, bridging the gap between the microscopic descriptions in terms of agent based models and macroscopic or hydrodynamic descriptions for the particle probability density. 

\section{Preliminaries}
In this section we collect the technical lemmas that are used in the proofs of the main theorems. Throughout this manuscript, generic constants will be denoted generically by $C$ (independent of $N$), even if they are different from line to line. We use $\norm{\cdot}_p$ for the $L^p$ ($1\leq p\leq \infty$) norm of a function. Moreover if $v=(v_1,\cdots,v_N)$ is a vector, then $\norm{v}_\infty:=\max\limits_{i=1,\cdots,N}|v_i|$.
\subsection{Local Lipschitz bound}
First let us recall some estimates of the regularized kernel $k^N$ defined in \eqref{eq:force}:
\begin{lem}\label{lmkenerl}(Regularity of $k^N$)
	\begin{enumerate}[(i)]
		\item $k^N(0)=0$, $k^N(x)=k(x)$, for any $|x|\geq N^{-\delta}$ and  
		$|k^N(x)|\leq |k(x)|$, \mbox{for  
			any } $x\in\mathbb{R}^3$;
		\item $|\partial^\beta k^N(x)|\leq CN^{(2+|\beta|)\delta},\mbox{for  
			any } x\in\mathbb{R}^3$;
		\item $\|k^N\|_{2}\leq CN^{\frac{\delta}{2}}.$
	\end{enumerate}
\end{lem}
The estimate $(i)$ has been proved in \cite[Lemma 2.1]{RY} and $(ii)$ follows from \cite[Lemma 5.1]{BM}. As for $(iii)$, it is a direct result of Young's inequality.

Next we define a cut-off function $\ell^N$, which will provide the local Lipschitz bound for $k^N$.
\begin{defn}\label{defLN}
	Let
	\begin{equation}\label{lN}
	\ell^N(x)=\left\{
	\begin{aligned}
	&\frac{6^3}{|x|^{3}},  && \text{ if } |x|\geq 6N^{-\delta},\\
	&N^{3\delta},  && \text{ else },
	\end{aligned}
	\right.
	\end{equation}
	and $L^N: \RR^{3N}\rightarrow \RR^N$ be defined by  
	$(L^N(X_t))_i:=\frac{1}{N-1}\sum\limits_{i\neq j}\ell^N(x_i^t-x_j^t)$.  
	Furthermore, we define $\overline L^N(\overline X_t)$ by $( \overline  
	L^N(\overline X_t))_i:=\int_{\RR^3} \ell^N(\overline  
	x_i^t-x)\rho^N(x,t)dx$.
\end{defn}

We summarize our first observation of $k^N$ and $\ell^N$ in the following lemma:
\begin{lem}\label{lmmid}
	There is a constant $C>0$ independent of $N$ such that for all $x,y\in\mathbb{R}^3$ with
	$|x-y|\leq N^{-\lambda_2}\gg N^{-\delta}$ $(\lambda_2<\delta)$ the following holds:
	\begin{align*}
	\frac{\left|\nabla k^N(x)\right|}{\ell^N(y)}\leq CN^{3(\delta-\lambda_2)},
	\end{align*}
	where $k^N$ is the regularization of the Coulomb kernel \eqref{couke} and $\ell^N$ satisfies Definition \ref{defLN}.
\end{lem}

\begin{proof}
	Let us first consider the case $|y|<2 N^ {-\lambda_2}$. It follows  
	from the bound from Lemma \ref{lmkenerl} and the decrease of $\ell^ N$ that
	\begin{align}\label{case1}
	\frac{\left|\nabla k^N(x)\right|}{\ell^N(y)}\leq  \frac{N^{3\delta}}{\ell^  
		N(2 N^ {-\lambda_2})}=CN^{3(\delta-\lambda_2)},
	\end{align}
	where we used  $2N^{-\lambda_2}>6N^ {-\delta}$, thus $\ell^ N( 2N^  
	{-\lambda_2})=27N^{3\lambda_2}$.
	
	Next we consider the case $|y|\geq 2 N^ {-\lambda_2}$. It follows that  
	$|x|\geq N^ {-\lambda_2}$ and thus by  Lemma \ref{lmkenerl} (i)
	
	\begin{align}\label{case2}
	\frac{\left|\nabla k^N(x)\right|}{\ell^N(y)}\leq \frac{C|x|^{-3}}{|y|^{-3}}\leq C \frac{(|y|-N^  
		{-\lambda_2})^{-3} }{|y|^{-3}}\leq C,
	\end{align}
	where in the last step we used $|x|\geq (|y|-N^  
	{-\lambda_2})\geq\frac{|y|}{2}$ for $|y|\geq 2N^ {-\lambda_2}$.
	Collecting \eqref{case1} and \eqref{case2} finishes the proof.
\end{proof}

Recall the notations
\begin{equation}\label{FN}
(K^N(X_t))_i:=\frac{1}{N-1}\sum_{j\neq i}k^N(x_i^t-x_j^t),\quad (K^N(\overline X_t))_i:=\frac{1}{N-1}\sum_{j\neq i}k^N(\overline x_i^t-\overline x_j^t),
\end{equation}
and we have the local Lipschitz continuity of $K^N$:
\begin{lem}\label{lmlip}
	If $\|X_t-\overline X_t\|_\infty\leq 2N^{-\delta}$, then it holds that
	\begin{equation}\label{lmlipeq}
	\nel K^N(X_t)-K^N(\overline X_t)\ner_\infty\leq C\|L^N(\overline X_t)\|_\infty\|X_t-\overline X_t\|_\infty,
	\end{equation}
	for some $C>0$ independent of $N$.
\end{lem}
\begin{proof}
	For any $\xi\in\RR^3$ with $|\xi|<4N^{-\delta}$, we claim that 
	\begin{equation}\label{lemcla}
	|k^N(x+\xi)-k^N(x)|\leq C\ell^N(x)|\xi|,
	\end{equation}
	where $\ell^N(x)$ is defined in \eqref{lN}.  Indeed, for $|x|<6N^{-\delta}$, estimate \eqref{lemcla} holds due to the fact that
	$\norm{\nabla k^N}_\infty\leq  N^{3\delta}$. For $|x|\geq 6N^{-\delta}$, there exists $s\in[0,1]$ such that
	 $$	|k^N(x+\xi)-k^N(x)|\leq |\nabla k^N(x+s\xi)||\xi|,$$
	where 
	$$|\nabla k^N(x+s\xi)|\leq C|x+s\xi|^{-3}.$$
	The right hand side of the above expression takes its largest value when $s=1$ and
	$$|x+s\xi|^{-3}\leq |x(1-\frac{|\xi|}{|x|})|^{-3}.$$
	Since $|\xi|<4N^{-\delta}$ and $|x|\geq 6N^{-\delta}$, it follows that $\frac{|\xi|}{|x|}<\frac{2}{3}$. Therefore, we get
	$$|k^N(x+\xi)-k^N(x)|\leq C\left(\frac{3}{|x|}\right)^2|\xi|\leq C\frac{|\xi|}{|x|^3}.$$
	
	Applying claim \eqref{lemcla} one has
	\begin{align}
	|(K^N(X_t))_i-(K^N(\overline X_t))_i|&\leq \frac{1}{N-1}\sum\limits_{j\neq i}^{N}| k^N(x_i^t-x_j^t)-k^N(\overline x_i^t-\overline x_j^t)|\notag\\
	&\leq \frac{1}{N-1}\sum\limits_{j\neq i}^{N}C\ell^N(\overline x_i^t-\overline x_j^t)|x_i^t-x_j^t-\overline x_i^t+\overline x_j^t|\notag \\
	&\leq C( L^N(\overline X_t))_i\|X_t-\overline X_t\|_\infty\leq C\|L^N(\overline X_t)\|_\infty\|X_t-\overline X_t\|_\infty,
	\end{align}
	which leads to \eqref{lmlipeq}.
\end{proof}

The following observations of $k^N$ and $\ell^N$ turn out to be very helpful in the sequel:
\begin{lem}\label{converse}
	Let $\ell^N(x)$ be defined in Definition \ref{defLN} and $\rho \in W^{1,1}\cap W^{1,\infty} (\RR^3)$.  Then there exists a constant $C>0$ independent of $N$ such that
	\begin{equation}\label{con1}
	\norm{\ell^N\ast \rho}_\infty\leq C\log(N)(\norm{\rho}_1+\norm{\rho}_\infty),\quad\norm{(\ell^N)^2\ast \rho}_\infty\leq CN^{(3\delta)}(\norm{\rho}_1+\norm{\rho}_\infty);
	\end{equation}
		and
		\begin{equation}\label{con2}
	\norm{ k^N\ast \rho}_\infty\leq C(\norm{\rho}_1+\norm{\rho}_\infty),\quad	\norm{\nabla k^N\ast \rho}_\infty\leq C(\norm{\nabla\rho}_1+\norm{\nabla\rho}_\infty).
		\end{equation}
\end{lem}
\begin{proof}
	We only prove one of the estimates above, since all the estimates can be obtained through the same procedure.
	One can estimate
		\begin{align}
		&\norm{\ell^N\ast\rho}_\infty=\Norm{\int_{\RR^3}\ell^N(x-y)\rho(y)dy}_\infty\notag \\
		\leq &\Norm{\int_{|x-y|<6N^{-\delta}}\ell^N(x-y)\rho(y)dy}_\infty+\Norm{\int_{6N^{-\delta}
			\leq|x-y|\leq 1}\ell^N(x-y)\rho(y)dy}_\infty\notag\\
	&+\Norm{\int_{1
			\leq|x-y|}\ell^N(x-y)\rho(y)dy}_\infty.
		\end{align}		
We estimate the first term
\begin{align}\label{ter1}
\Norm{\int_{|x-y|<6N^{-\delta}}\ell^N(x-y)\rho(y)dy}_\infty\leq \norm{\rho}_\infty\norm{\ell^N}_\infty|B(6N^{-\delta})|\leq \frac{4\pi}{3}(6N^{-\delta})^3N^{3\delta}\norm{\rho}_\infty\leq C\norm{\rho}_\infty,
\end{align}
where $B(r)$ denotes the ball with radius $r$ in $\RR^3$.
The second term is bounded by
\begin{align}\label{ter2}
\Norm{\int_{6N^{-\delta}\leq|x-y|\leq 1}\ell^N(x-y)\rho(y)dy}_\infty\leq \norm{\rho}_\infty\int_{6N^{-\delta}\leq|y|\leq 1}\frac{C}{|y|^3}dy\leq C\log(N)\norm{\rho}_\infty.
\end{align}
It is easy to compute the last term 
\begin{equation}\label{ter3}
\Norm{\int_{1
		\leq|x-y|}\ell^N(x-y)\rho(y)dy}_\infty\leq C\norm{\rho}_1.
\end{equation}
Collecting estimates \eqref{ter1}, \eqref{ter2} and \eqref{ter3}, one has
\begin{equation}
\norm{\ell^N\ast\rho}_\infty\leq C\norm{\rho}_{\infty}+C\log(N)\norm{\rho}_\infty+C\norm{\rho}_1\leq C\log(N)(\norm{\rho}_\infty+\norm{\rho}_1).
\end{equation}
\end{proof}

\subsection{Law of Large Numbers}
Also, we need the following concentration inequality to  provide us the probability bounds of random variables:
\begin{lem}\label{central} 
	Let $Z_1,\cdots,Z_N$ be $i.i.d.$ random variables with $\mathbb{E}[Z_i]=0,$ $\mathbb{E}[Z_i^2]\leq g(N)$
	and $|Z_i|\leq C\sqrt{Ng(N)}$. Then for any $\alpha>0$, the sample mean $\bar{Z}=\frac{1}{N}\sum_{i=1}^{N}Z_i$ satisfies
	\begin{equation}
	\PP\left(|\bar{Z}|\geq\frac{C_\alpha \sqrt{g(N)}\log(N)}{\sqrt{N}}\right)\leq N^{-\alpha},
	\end{equation}
	where $C_\alpha$ depends only on $C$ and $\alpha$.
\end{lem}
The proof can be seen in \cite[Lemma 1]{GJ}, which is a direct result of  Taylor's expansion and Markov's inequality.

Recall the notation
\begin{equation}\label{barFN}
(\overline K^N(\overline X_t))_i:=\int_{\RR^3} k^N(\overline x_i^t-x)\rho^N(x,t)dx.
\end{equation}
We can introduce the following version of the Law of Large Numbers:
\begin{lem}\label{lmlarge} At any fixed time $t\in[0,T]$, suppose that $\overline X_t$ satisfies the mean-field dynamics \eqref{eq:mean}, $K^N$ and $\overline K^N$ are defined in \eqref{FN} and \eqref{barFN} respectively,  $L^N$ and $\overline L^N$ are introduced in Definition \ref{defLN}. For any $\alpha>0$ and $\frac{1}{3}\leq\delta<1$, there exist a constant $C_{1,\alpha}>0$ depending only on $\alpha$, $T$ and $C_{f_0}$ such that
	\begin{equation}\label{largef}
	\PP\left(\nel K^N(\overline X_t)-\overline K^N(\overline X_t)\ner_\infty\geq C_{1,\alpha} N^{2\delta-1}\log (N)\right)\leq N^{-\alpha},
	\end{equation}
	and
	\begin{equation}\label{largel}
	\PP\left(\nel L^N(\overline X_{t})-\overline{L}^N(\overline X_{t})\ner_\infty\geq C_{1,\alpha} N^{3\delta-1}\log (N)\right)\leq N^{-\alpha}.
	\end{equation}
\end{lem}
\begin{proof}
	We can prove this lemma by using Lemma \ref{central}.
	Due to the exchangeability of the particles, we are ready to bound
	\begin{equation}
	(K^N(\overline X_t))_1-(\overline K^N(\overline X_t))_1=\frac{1}{N-1}\sum_{j=2}^Nk^N(\overline x_1^t-\overline x_j^t)-\int_{\RR^3} k^N(\overline x_1^t-x)\rho^N(x,t)dx=\frac{1}{N-1}\sum_{j=2}^{N}Z_j,
	\end{equation}
	where $$Z_j:=k^N(\overline x_1^t-\overline x_j^t)-\int_{\RR^3} k^N(\overline x_1^t-x)\rho^N(x,t)dx.$$ 
	Since $\overline x_1^t$ and $\overline x_j^t$ are independent when $j\neq 1$ and $k^N(0)=0$, let us consider $\overline x_1^t$ as given and denote $\mathbb{E'}[\cdot]=\mathbb{E}[\cdot|\overline x_1^t]$. It is easy to show that
	$\mathbb{E}'[Z_j]=0$ since
	\begin{align}
	\mathbb{E}'\left[k^N(\overline x_1^t-\overline x_j^t)\right]&=\iint_{\RR^6} k^N(\overline x_1^t-x)f^N(x,v,t)dxdv\notag \\
	&=\int_{\RR^3} k^N(\overline x_1^t-x)\rho^N(x,t)dx.
	\end{align}
	
	To use Lemma \ref{central}, we need a bound for the variance
	\begin{equation}
	\mathbb{E}'\big[|Z_j|^2\big]=\mathbb{E}'\left[\left|k^N(\overline x_1^t-\overline x_j^t)-\int_{\RR^3} k^N(\overline x_1^t-x)\rho^N(x,t)dx\right|^2\right].
	\end{equation}
	Since it follows from Lemma \ref{converse} that
	\begin{equation}
\int_{\RR^3} k^N(\overline x_1^t-x)\rho^N(x,t)dx\leq C(\norm{\rho^N}_1+\norm{\rho^N}_\infty),
	\end{equation}
	it suffices to bound
	\begin{equation}
	\mathbb{E'}\big[k^N(\overline x_1^t-\overline x_j^t)\big]=\int_{\RR^3} k^N(\overline x_1^t-x)\rho^N(x,t)dx\leq C(\norm{\rho^N}_1+\norm{\rho^N}_\infty)\leq C(T,C_{f_0}),
	\end{equation}
	and
	\begin{equation}
	\mathbb{E'}\big[k^N(\overline x_1^t-\overline x_j^t)^2\big]=\int_{\RR^3} k^N(\overline x_1^t-x)^2\rho^N(x,t)dx\leq \norm{\rho^N}_\infty\norm{k^N}_2^2\leq C(T,C_{f_0})N^{\delta},
	\end{equation}
	where we have used $\norm{k^N}_2\leq CN^{\frac{\delta}{2}}$   in Lemma \ref{lmkenerl} $(iii)$.
	Hence one has
	\begin{equation}\label{Esqure}
	\mathbb{E}'\big[|Z_j|^2\big]\leq CN^{\delta}.
	\end{equation}
	
So the hypotheses of Lemma \ref{central} are satisfied with $g(N)=CN^{4\delta-1}$. In addition, it follows from $(ii)$ in Lemma \ref{lmkenerl} that $|Z_j|\leq CN^{2\delta}\leq C\sqrt{Ng(N)}$. Hence, using Lemma \ref{central}, we have the probability bound 
	\begin{equation}
	\PP\left(\left|(K^N(\overline X_t))_1-(\overline K^N(\overline X_t))_1\right|\geq C(\alpha,T,C_{f_0}) N^{2\delta-1}\log (N)\right)\leq N^{-\alpha}.
	\end{equation}
	Similarly, the same bound  also  holds for all other indexes  $i=2,\cdots,N$, which leads to
	\begin{equation}\label{residual1'}
	\PP\left(\nel K^N(\overline X_t)-\overline K^N(\overline X_t)\ner_\infty\geq C(\alpha,T,C_{f_0}) N^{2\delta-1}\log (N)\right)\leq N^{1-\alpha}.
	\end{equation}
	Let $C_{1,\alpha}$ be the constant $C(\alpha,T,C_{f_0}) $ in  \eqref{residual1'}, then we conclude \eqref{largef}.
	
	To prove \eqref{largel}, we follow the same procedure as above
	\begin{equation}
	(L^N(\overline X_t))_1-(\overline L^N(\overline X_t))_1=\frac{1}{N-1}\sum_{j=2}^N\ell^N(\overline x_1^t-\overline x_j^t)-\int_{\RR^3} \ell^N(\overline x_1^t-x)\rho^N(x,t)dx=\frac{1}{N-1}\sum_{j=2}^{N}Z_j,
	\end{equation}
	where $$Z_j=\ell^N(\overline x_1^t-\overline x_j^t)-\int_{\RR^3} \ell^N(\overline x_1^t-x)\rho^N(x,t)dx.$$ 
It is easy to show that $\mathbb{E}'[Z_j]=0$.
	To use Lemma \ref{central}, we need a bound for the variance. One computes that
	\begin{equation}
	\mathbb{E'}\big[\ell^N(\overline x_1^t-\overline x_j^t)\big]=\int_{\RR^3} \ell^N(\overline x_1^t-x)\rho^N(x,t)dx\leq C\log(N)(\norm{\rho}_1+\norm{\rho}_\infty)\leq C(T,C_{f_0})\log(N),
	\end{equation}
	and
	\begin{equation}
	\mathbb{E'}\big[\ell^N(\overline x_1^t-\overline x_j^t)^2\big]=\int_{\RR^3} \ell^N(\overline x_1^t-x)^2\rho^N(x,t)dx\leq CN^{3\delta}(\norm{\rho}_1+\norm{\rho}_\infty)\leq C(T,C_{f_0})N^{3\delta},
	\end{equation}
	where we have used the estimates of $\ell^N$ in Lemma \ref{converse}.
	Hence one has
	\begin{equation}
	\mathbb{E}'\big[|Z_j|^2\big]\leq CN^{3\delta}.
	\end{equation}
	
	So the hypotheses of Lemma \ref{central} are satisfied with $g(N)=CN^{6\delta-1}$. In addition, it follows from Definition \ref{defLN}  that $|Z_j|\leq CN^{3\delta}\leq C\sqrt{Ng(N)}$. Hence, we have the probability bound 
	\begin{equation}
	\PP\left(\left|(L^N(\overline X_t))_1-(\overline L^N(\overline X_t))_1\right|\geq C(\alpha,T,C_{f_0}) N^{3\delta-1}\log (N)\right)\leq N^{-\alpha},
	\end{equation}
	by Lemma \ref{central}, which leads to
	\begin{equation}\label{residual1''}
	\PP\left(\nel L^N(\overline X_t)-\overline L^N(\overline X_t)\ner_\infty\geq C(\alpha,T,C_{f_0}) N^{3\delta-1}\log (N)\right)\leq N^{1-\alpha}.
	\end{equation}
Thus, \eqref{largel} follows from \eqref{residual1''}.
	
\end{proof}

\section{Proof of Theorem \ref{mainthm}}
We do the proof by following the idea in \cite{HH2,HH1}, which is that consistency and stability imply convergence. This at least in principle corresponds to the Lax's equivalence theorem of proving the convergence of a numerical algorithm, which is
that stability and consistency of an algorithm imply its convergence. 

\subsection{Consistency}

In order to obtain the  consistency error for the entire time interval, we divide $[0,T]$ into $M+1$ subintervals with length $\Delta \tau=N^{-\frac{\gamma}{3}}$ for some $\gamma>4$ and $\tau_k=n\Delta \tau$, $k=0,\cdots,M+1$. The choice of $\gamma$ will be clear from the discussion below. Here the choice of $\Delta \tau$ is only for the purpose of proving consistency and it can be sufficiently small. Note that it is different from $\Delta t$ in the proof of stability in the next subsection.

First, we establish the following lemma on the traveling distance of $\overline X_t$ in a short time interval $[\tau_k,\tau_{k+1}]$:
\begin{lem}\label{lmQ}
	Assume that $(\overline X_t,\overline V_t)$  satisfies the mean-field dynamics \eqref{eq:mean}. For $\gamma>4$ it holds
	\begin{equation}\label{lmQeq1}
	\PP\left(\max\limits_k\max\limits_{t\in[\tau_k,\tau_{k+1}]}\nel\overline X_t-\overline X_{\tau_k}\ner_\infty\geq C_BN^{-\frac{\gamma-1}{3}}\right)\leq C_BN^{\frac{\gamma-1}{3}}\exp(-C_BN^{\frac{2}{3}}),
	\end{equation}
	where $C_B$ depends only on $T$ and $C_{f_0}$.
\end{lem}
\begin{proof}
	Notice that for $t\in [\tau_k,\tau_{k+1}]$
	\begin{align}\label{Qdiff}
	\overline X_t-\overline X_{\tau_k}&=\int_{\tau_k}^{t}\overline V_sds=\int_{\tau_k}^{t}\int_{\tau_k}^{s}\overline K^N(\overline X_\tau)d\tau ds+\sqrt{2\sigma}\int_{\tau_k}^{t}(B(s)-B(\tau_k))ds+\int_{\tau_k}^{t}\overline V_{\tau_k}ds,\notag \\
	&=:I_1^k(t)+I_2^k(t)+I_3^k(t),
	\end{align}
	where
	\begin{equation}
	\overline V_{\tau_k}=V_0+\int_{0}^{\tau_k}\overline K^N(\overline X_s)ds+\sqrt{2\sigma}B(\tau_k).
	\end{equation}

The estimate of $I_1^k(t)$ follows from Lemma \ref{converse}
\begin{equation}
\int_{\tau_k}^{t}\int_{\tau_k}^{s}\overline K^N(\overline X_\tau)d\tau ds\leq (\Delta t)^2\|\overline K^N\|_\infty\leq CN^{-\frac{2\gamma}{3}}.
\end{equation}
So we have
\begin{equation}\label{I1}
\max\limits_k\max\limits_{t\in[\tau_k,\tau_{k+1}]}\norm{I_1^k(t)}_\infty\leq CN^{-\frac{2\gamma}{3}}.
\end{equation}
	
To estimate $I_2^k(t)$, recall a basic property of Brownian motion \cite[Chap. 1.2]{freedman1983brownian}:
	\begin{equation}\label{Bproperty}
	\PP\left(\max\limits_{t\leq s\leq t+\Delta t}\|B(s)-B(t)\|_\infty\geq b\right)\leq C_1(\sqrt{\Delta t}/b)\exp(-C_2b^2/\Delta t),
	\end{equation}
	which leads to 
	\begin{equation}\label{57}
		\PP\left(\max\limits_{t\in[\tau_k,\tau_{k+1}]}\|B(t)-B(\tau_k)\|_\infty\geq N^{-\frac{1}{3}}\right)\leq C_1N^{-\frac{\gamma-2}{6}}\exp(-C_2N^{\frac{\gamma-2}{3}}),
	\end{equation}
	where we choose $b=N^{-\frac{1}{3}}$.
	
	Since $\max\limits_{t\in[\tau_k,\tau_{k+1}]}\norm{I_2^k(t)}_\infty\leq \Delta t \sqrt{2\sigma} \max\limits_{t\in[\tau_k,\tau_{k+1}]}\|B(t)-B(\tau_k)\|_\infty$, it follows from \eqref{57} that
	\begin{equation}
	\PP\left(\max\limits_{t\in[\tau_k,\tau_{k+1}]}\norm{I_2^k(t)}_\infty\geq C N^{-\frac{\gamma+1}{3}}\right)\leq C_1N^{-\frac{\gamma-2}{6}}\exp(-C_2N^{\frac{\gamma-2}{3}}),
	\end{equation}
	which leads to 
		\begin{equation}\label{I2}
		\PP\left(\max\limits_k\max\limits_{t\in[\tau_k,\tau_{k+1}]}\norm{I_2^k(t)}_\infty\geq C N^{-\frac{\gamma+1}{3}}\right)\leq C_1N^{\frac{\gamma+2}{6}}\exp(-C_2N^{\frac{\gamma-2}{3}}),
		\end{equation}
	where we used the fact that $n\leq\frac{T}{\Delta t}=TN^{\frac{\gamma}{3}}$.
	
	Lastly, we prove the estimate of $I_3^k(t)$. It is obvious that
	\begin{equation}
	\int_{0}^{\tau_k}\overline K^N(\overline X_s)ds\leq n \Delta t \norm{\overline K^N}_\infty\leq CT,
	\end{equation}
	and it follows from \eqref{Bproperty} that
	\begin{equation}
	\PP(\|B(\tau_k)\|_\infty\geq N^{\frac{1}{3}})\leq C_1N^{-\frac{1}{3}}\sqrt{T}\exp(-C_2N^{\frac{2}{3}}/T).
	\end{equation}
	Moreover,  it follows from the assumption in Theorem \ref{existence} $b)$  the distribution $f_0^v(v)$ of $V_0$ has a compact support:
	\begin{equation}
	f_0^v(v)=\int_{\RR^3}f_0(x,v)dx=0,\mbox{ when }|v|> Q_v .
	\end{equation}
	Then one has
	\begin{equation}\label{compactv}
	\PP(\|V_0\|_\infty\geq N^{\frac{1}{3}})=\int_{|v|\geq N^{\frac{1}{3}}}f_0^v(v)dv=0, \mbox{ when }N> Q_v ^3.
	\end{equation}
	It follows from \eqref{Qdiff} that 
	\begin{align}
\max\limits_{t\in[\tau_k,\tau_{k+1}]}\|I_3^k(t)\|_\infty&=\int_{\tau_k}^{t}\|\overline V_{\tau_k}\|_\infty ds\leq  N^{-\frac{\gamma}{3}}\left(\|V_0\|_\infty+\sqrt{2\sigma}\|B(\tau_k)\|_\infty+\int_{0}^{\tau_k}\overline K^N(\overline X_s)ds\right)\notag\\
&\leq  N^{-\frac{\gamma}{3}}(\|V_0\|_\infty+\sqrt{2\sigma}\|B(\tau_k)\|_\infty)+CN^{-\frac{\gamma}{3}},
	\end{align}
	then it yields
	\begin{align}
	&\PP\left(\max\limits_{t\in[\tau_k,\tau_{k+1}]}\nel I_3^k(t)\ner_\infty \geq3N^{-\frac{\gamma-1}{3}}\right)\notag \\
	\leq &\PP\left(N^{-\frac{\gamma}{3}}\|V_0\|_\infty\geq N^{-\frac{\gamma-1}{3}}\right)+\PP\left(\sqrt{2\sigma}N^{-\frac{\gamma}{3}}\|B(\tau_k)\|_\infty\geq N^{-\frac{\gamma-1}{3}}\right)+\PP\left(CN^{-\frac{\gamma}{3}}\geq N^{-\frac{\gamma-1}{3}}\right)\notag \\
	\leq &0+CN^{-\frac{1}{3}}\exp(-CN^{\frac{2}{3}}) +0\leq CN^{-\frac{1}{3}}\exp(-CN^{\frac{2}{3}}) ,
	\end{align}
	which leads to
	
		\begin{align}\label{I3}
		\PP\left(\max\limits_k\max\limits_{t\in[\tau_k,\tau_{k+1}]}\nel I_3^k(t)\ner_\infty\geq3N^{-\frac{\gamma-1}{3}}\right)\leq CN^{\frac{\gamma-1}{3}}\exp(-CN^{\frac{2}{3}}).
		\end{align}
	Then it follows from \eqref{I1}, \eqref{I2} and \eqref{I3} that
		\begin{align*}
		&\PP\left(\max\limits_k\max\limits_{t\in[\tau_k,\tau_{k+1}]}\nel\overline X_t-\overline X_{\tau_k}\ner_\infty\geq CN^{-\frac{\gamma-1}{3}}\right)\notag\\
		\leq &C_1N^{\frac{\gamma+2}{6}}\exp(-C_2N^{\frac{\gamma-2}{3}})+CN^{\frac{\gamma-1}{3}}\exp(-CN^{\frac{2}{3}}) \leq CN^{\frac{\gamma-1}{3}}\exp(-CN^{\frac{2}{3}}) , 
		\end{align*}
	   for $\gamma>4$,	which completes the proof of \eqref{lmQeq1}.
\end{proof}

Now we can prove the consistency error for the entire time interval $[0,T]$. 
\begin{proposition}\label{propconsis}(Consistency) For any $T>0$, let $(\overline X_t,\overline V_t)$ satisfy the mean-field dynamics \eqref{eq:mean} with initial density $f_0(x,v)$, $K^N$ and $\overline K^N$ be defined in \eqref{FN} and \eqref{barFN} respectively. For any $\alpha>0$ and $\frac{1}{3}\leq\delta<1$, there exist a constant $C_{2,\alpha}>0$ depending only on $\alpha$, $T$ and $C_{f_0}$ such that
	\begin{equation}\label{consistency}
	\PP\left(\max\limits_{t\in[0,T]}\nel K^N(\overline X_t)-\overline K^N(\overline X_t)\ner_\infty\geq C_{2,\alpha} N^{2\delta-1}\log (N)\right)\leq  N^{-\alpha},
	\end{equation}
	and
		\begin{equation}\label{consistency1}
		\PP\left(\max\limits_{t\in[0,T]}\nel L^N(\overline X_t)-\overline L^N(\overline X_t)\ner_\infty\geq C_{2,\alpha} N^{3\delta-1}\log (N)\right)\leq  N^{-\alpha}.
		\end{equation}
\end{proposition}
\begin{proof}
Denote the events:
\begin{equation}\label{eventH}
\mathcal{\overline H}:=\left\{\max\limits_k\max\limits_{t\in[\tau_k,\tau_{k+1}]}\nel\overline X_t-\overline X_{\tau_k}\ner_\infty\leq C_BN^{-\frac{\gamma-1}{3}}\right\},
\end{equation}
and
\begin{equation}
\mC_{\tau_k}:=\left\{\nel K^N(\overline X_{\tau_k})-\overline K^N(\overline X_{\tau_k})\ner_\infty\leq C_{1,\alpha} N^{2\delta-1}\log (N)\right\},
\end{equation}
where $C_B$ and $C_{1,\alpha} $ are used in Lemma \ref{lmlarge} and Lemma \ref{lmQ} respectively.
According to   Lemma \ref{lmlarge} and Lemma \ref{lmQ}, one has
\begin{equation}\label{CHevent}
\PP(\mC_{\tau_k}^c)\leq N^{-\alpha},\quad \PP(\mathcal{\overline H}^c)\leq C_BN^{\frac{\gamma-1}{3}}\exp(-C_BN^{\frac{2}{3}}),
\end{equation}
for any $\alpha>0$ and $\gamma>4$.

	Furthermore,  we denote
	\begin{equation}\label{Btn}
	\B_{\tau_k}:=\left\{\nel L^N(\overline X_{\tau_k})-\overline{L}^N(\overline X_{\tau_k})\ner_\infty\leq C_{1,\alpha} N^{3\delta-1}\log (N)\right\},
	\end{equation}
	then one has 
\begin{equation}\label{Bevent}
	\PP(\B_{\tau_k}^c)\leq N^{-\alpha},
\end{equation}
	by Lemma \ref{lmlarge}.
	Also, under the event $\B_{\tau_k}$, it holds that
	\begin{equation}\label{Btnresult}
	\|L^N(\overline X_{\tau_k})\|_\infty\leq \|\overline L^N(\overline X_{\tau_k})\|_\infty+C_{1,\alpha}N^{3\delta-1}\log (N)\leq C(\alpha, T, C_{f_0})N^{3\delta-1}\log (N),
	\end{equation}
where we have used $\|\overline L^N(\overline X_{\tau_k})\|_\infty\leq C\log (N)$ from Lemma \ref{converse}.

For all $t\in[\tau_k,\tau_{k+1}]$, under the event $\B_{\tau_k}\cap\mC_{\tau_k}\cap \mathcal{\overline H}$, we obtain
\begin{align}
&\nel K^N(\overline X_t)-\overline K^N(\overline X_t)\ner_\infty\notag\\
\leq &\nel K^N(\overline X_t)-K^N(\overline X_{\tau_k})\ner_\infty+\nel K^N(\overline X_{\tau_k})-\overline K^N(\overline X_{\tau_k})\ner_\infty+\nel \overline K^N(\overline X_{\tau_k})-\overline K^N(\overline X_t)\ner_\infty\notag\\
\leq& C\|L^N(\overline X_{\tau_k})\|_\infty\|\overline X_t-\overline X_{\tau_k}\|_\infty+C_{1,\alpha} N^{2\delta-1}\log (N)+C\nel\overline X_{t}-\overline X_{\tau_k}\ner_\infty+CN^{-\frac{\gamma}{3}}\notag\\
\leq&C(\alpha, T, C_{f_0}) N^{3\delta-1}\log (N)N^{-\frac{\gamma-1}{3}}+C_{1,\alpha} N^{2\delta-1}\log (N)\notag\\
\leq&C(\alpha, T, C_{f_0}) N^{2\delta-1}\log (N),
\end{align}
due to the fact that $3\delta+1<4<\gamma$. In the second inequality we have used the local Lipschitz bound of $K^N$
\begin{equation}
\nel K^N(X_t)-K^N(\overline X_{\tau_k})\ner_\infty\leq C\|L^N(\overline X_{\tau_k})\|_\infty\|X_t-\overline X_{\tau_k}\|_\infty,
\end{equation}
 under the event $\mathcal{\overline H}$ (see in Lemma \ref{lmlip}).   To bound the third term $\nel \overline K^N(\overline X_{\tau_k})-\overline K^N(\overline X_t)\ner_\infty$, we used the uniform control of $\max\limits_{\tau_k\leq t \leq \tau_{k+1}}\norm{\partial_t\rho^N}_{L^\infty(\RR^3)}$ in \eqref{partialtrho}. Indeed,  
 \begin{align}
&\| k^N\ast \rho_{t}(X_{t})- k^N\ast \rho_{\tau_k}(X_{\tau_k})\|_\infty\notag\\
\leq& \| k^N\ast \rho_{t}(X_{t})- k^N\ast \rho_{t}(X_{\tau_k})\|_\infty+\| k^N\ast \rho_{t}(X_{\tau_k})- k^N\ast \rho_{\tau_k}(X_{\tau_k})\|_\infty\notag\\
\leq &C\nel\overline X_{t}-\overline X_{\tau_k}\ner_\infty+C\Delta \tau\leq C\nel\overline X_{t}-\overline X_{\tau_k}\ner_\infty+CN^{-\frac{\gamma}{3}}.
 \end{align}
 In the third inequality we have used \eqref{eventH} and \eqref{Btnresult}.
This yields that
\begin{equation}
\max\limits_{t\in[0,T]}\nel K^N(\overline X_t)-\overline K^N(\overline X_t)\ner_\infty\leq C(\alpha, T, C_{f_0})  N^{2\delta-1}\log (N),
\end{equation}
holds under the event $\bigcap\limits_{k=0}^{M}\B_{\tau_k}\cap\mC_{\tau_k}\cap\mathcal{\overline H}$. Therefore it follows from \eqref{CHevent} and \eqref{Bevent} that
\begin{align}\label{68}
		&\PP\left(\max\limits_{t\in[0,T]}\nel K^N(\overline X_t)-\overline K^N(\overline X_t)\ner_\infty\geq C(\alpha, T, C_{f_0})N^{2\delta-1}\log (N)\right)\notag \\
		\leq &\sum\limits_{k=0}^MP(\B_{\tau_k}^c)+\sum\limits_{k=0}^MP(\mC_{\tau_k}^c)+P(\mathcal{\overline H}^c)\notag \\
		\leq &TN^{-\frac{3\alpha-\gamma}{3}}+TN^{-\frac{3\alpha-\gamma}{3}}+C_BN^{\frac{\gamma-1}{3}}\exp(-C_BN^{\frac{2}{3}})
		\leq N^{-\alpha'}.
	\end{align}
Denote $C_{2,\alpha'}$ to be the constant $C(\alpha, T, C_{f_0})$ in \eqref{68}. Since $\alpha>0$ is arbitrary and so is $\alpha'$, \eqref{consistency} holds true. The proof of \eqref{consistency1} can be done similarly.
\end{proof}

\subsection{Stability}
In this subsection we obtain  the stability result. 
\begin{defn}\label{defA}
	Let $\A_T$ be the event given by
	\begin{equation}\label{eventA}
	\A_T:=\left\{\max\limits_{t\in[0,T]} \sqrt{\log(N)}\nel X_t-\overline X_t\ner_\infty+\nel V_t-\overline V_t\ner_\infty\leq N^{-\lambda_2}\right\}.
	\end{equation}
\end{defn}	
\begin{proposition}\label{propstab}
	(Stability)   For any $T>0$, assume that the trajectories $\Phi_t=(X_t,V_t)$, $\Psi_t=(\overline X_t,\overline V_t)$ satisfy \eqref{eq:regpar} and \eqref{eq:mean} respectively with the initial data $\Phi_0=\Psi_0$ which is i.i.d. sharing the common density $f_0$ satisfying Assumption \ref{assum}. Let $K^N$ be introduced in \eqref{FN}. For any $0<\lambda_2<\frac{1}{3}$, $0<\lambda_1<\frac{\lambda_2}{3}$ and $\frac{1}{3}\leq \delta <1$, we denote the event:
	\begin{align}
	\mathcal{S}_T(\Lambda):=\bigg\{&\| K^N(X_t)-K^N(\overline X_t) \|_\infty \leq \Lambda\log(N)\nel X_t-\overline X_t\ner_\infty\notag\\
	&+\Lambda\log^2(N)(N^{6\delta-1-\lambda_1-4\lambda_2}+N^{3\lambda_1-2\lambda_2}+N^{2\delta-1}),~\forall~t\in[0,T]\bigg\}.
	\end{align}
 	Then for any $\alpha>0$, there exists some $C_{3,\alpha}>0$ and a $N_0\in\mathbb{N}$  depending only on  $\alpha$, $T$ and $C_{f_0}$  such that
		\begin{align}
	&\mathbb{P}\left(\A_T \cap \mathcal{S}_T^c(C_{3,\alpha})\right)\leq N^{-\alpha},
	\end{align}
	for all $N\geq N_0$.
	
\end{proposition}
\begin{rmk}
	This proposition is one of the crucial statements in our paper. Proving propagation of chaos for systems like the one we consider under the assumptions of Lipschitz-continuous forces is standard, as explained in the introduction. The forces we consider are more singular. However our techniques allow us to show that the Lipschitz condition encoded in the definition of $\mathcal{S}$ holds typically, i.e. with probability close to one. In this sense, Proposition \ref{propstab} is  only helpful if we find an argument that $\A_T$ holds. 
	But as long as  we have good estimates on the difference of the forces and thus the growth of $\max\limits_{t\in[0,T]} \sqrt{\log(N)}\nel X_t-\overline X_t\ner_\infty+\nel V_t-\overline V_t\ner_\infty$, we are in fact able to control $\A_T$. This control is done by a generalization of  Gronwalls Lemma, which will be introduced in our next step (Lemma \ref{lmprior}).
\end{rmk}

\begin{proof} Let $\alpha>0$.
	First, we write $\mathcal{S}_T(\Lambda)$ as the intersection of non-overlapping sets $\{\mathcal{S}_n(\Lambda)\}_{n=0}^{M'}$, where
	\begin{align}
	\mathcal{S}_n(\Lambda):=\bigg\{&\| K^N(X_t)-K^N(\overline X_t) \|_\infty \leq \Lambda\log(N)\nel X_t-\overline X_t\ner_\infty\notag\\
	&+\Lambda\log^2(N)(N^{6\delta-1-\lambda_1-4\lambda_2}+N^{3\lambda_1-2\lambda_2}+N^{2\delta-1}),~\forall~t\in[t_n,t_{n+1}]\,0\leq n\leq M'\bigg\},
	\end{align}
	with $\Delta t:=t_{n+1}-t_n=N^{-\lambda_1}$, then $\mathcal{S}_T(\Lambda)=\bigcap\limits_{n=0}^{M'}\mathcal{S}_n(\Lambda)$. Note that here the choice of $\Delta t$ is for the purpose of proving stability and it is different from $\Delta \tau$ in the proof of consistency.
	
	To prove this proposition, we split the interaction force $k^N$ into $k^N=k_1^N+k_2^N$, where $k_2^N$ is the result of choosing a wider cut-off of order $N^{-\lambda_2}>N^{-\delta}$ in the force kernel $k$ and 
	\begin{equation}\label{k1}
	k_1^N:=k^N-k_2^N,\quad k_2^N=k\ast\psi_{\lambda_2}^N,
	\end{equation}
	which means that for $k_2^N$ and $\ell_2^N$ we choose $\delta=\lambda_2$ in \eqref{eq:force}  and   
	\eqref{lN}   respectively.  
	
	 Following the approach in \cite{garcia2017}, we introduce the  following auxiliary trajectory
	\begin{equation}\label{tilde}
	\left\{
	\begin{aligned}
	&d\widetilde x_i^t= \widetilde v_i^tdt,\\
	&d \widetilde v_i^t=\int_{\RR^3} k^N(\widetilde x_i^t-x)\rho^N(x,t)dxdt + \sqrt{2\sigma}dB_i^t\;.
	\end{aligned}
	\right.
	\end{equation}
	We consider the above auxiliary trajectory with two different initial phases. For any $1\leq n\leq M'$ and $t\in[t_n,t_{n+1}]$, we consider the auxiliary trajectory
	starting from the initial phase
	\begin{equation}\label{tilde1}
	(\widetilde x_i^{t_{n-1}}, \widetilde v_i^{t_{n-1}})=(x_i^{t_{n-1}}, v_i^{t_{n-1}}),
	\end{equation}
	where $(x_i^{t_{n-1}}, v_i^{t_{n-1}})$ satisfies \eqref{eq:regpar} at time $t_{n-1}$. However when $n=1$, i.e.  $t\in[0,t_1]$, the initial phase  of the auxiliary trajectory is chosen to be $(\widetilde x_i^{0}, \widetilde v_i^{0})=(x_i^{0}, v_i^{0})$, which has the distribution $f_0$. Moreover in the latter case the distribution of $(\widetilde x_i^{t}, \widetilde v_i^{t})$ is exactly $f_t^N$, which solves the regularized VPFP  equations \eqref{vlasov} with the initial data $f_0$.
	
	For later reference let us estimate the difference $\|\overline X_t -\widetilde X_t\|_\infty$ and $\|\overline V_t -\widetilde V_t\|_\infty$. Using the equations of these trajectories, we have for $t\in[t_n,t_{n+1}],$
	\begin{equation}
	\frac{d}{dt}\|\overline X_t -\widetilde X_t\|_\infty=\|\overline V_t -\widetilde V_t\|_\infty,
	\end{equation}
	and
	\begin{align}
	\frac{d}{dt}\|\overline V_t -\widetilde V_t\|_\infty=&\|\overline K^N(\overline X_t) -\overline K^N(\widetilde X_t)\|_\infty\notag
	\\\leq& \max_{1\leq j\leq N}|k^N \ast \rho^N(\cdot,t) (\overline x_j)-k^N\ast \rho^N(\cdot,t) (\widetilde x_j)|\notag
	\\\leq& \max_{1\leq j\leq N}|\overline x_j-\widetilde x_j|\|\nabla k^N\ast\rho^N(\cdot,t)\|_\infty\notag
	\\\leq&C(\norm{\nabla \rho^N}_1+\norm{\nabla \rho^N}_\infty)\|\overline X_t -\widetilde X_t\|_\infty \notag
	\\\leq& C \|\overline X_t -\widetilde X_t\|_\infty\;,
	\end{align}
	where $C$ depends only on $T$ and $C_{f_0}$.
	Summarizing, we get
	$$\frac{d}{dt}\left(\|\overline X_t -\widetilde X_t\|_\infty+\|\overline V_t -\widetilde V_t\|_\infty\right)\leq  C\left(\|\overline X_t -\widetilde X_t\|_\infty+\|\overline V_t -\widetilde V_t\|_\infty\right)$$
	Using Gronwall's inequality  it follows that 
	\begin{align}\label{overlineminustilde}
	\max_{t_n \leq t \leq t_{n+1}}\left(\|\overline X_t -\widetilde X_t\|_\infty+\|\overline V_t -\widetilde V_t\|_\infty\right)&\leq \exp(C\Delta t)(\|\overline X_{t_n} - X_{t_n}\|_\infty+\|\overline V_{t_n} - V_{t_n}\|_\infty)\notag\\
	&\leq \exp(CN^{-\lambda_1})N^{-\lambda_2}\leq CN^{-\lambda_2},
	\end{align}
	under the event $\mathcal{A}_T$ defined in \eqref{eventA}.

	Then for any $t\in[t_n,t_{n+1}]$, one splits the error
	\begin{align}
	&\| K^N(X_t)-K^N(\overline X_t) \|_\infty \notag \\
	\leq &\norm{K_2^N(X_t)-K_2^N(\overline X_t) }_\infty+\norm{K_1^N(X_t)-K_1^N(\widetilde X_t) }_\infty+\norm{K_1^N(\widetilde X_t)-K_1^N(\overline X_t) }_\infty \notag \\
	=:&\mathcal{I}_1+\mathcal{I}_2+\mathcal{I}_3.
	\end{align}

	First, let us compute $\mathcal{I}_1$:
	\begin{align}
	\| K_2^N(X_t)-K_2^N(\overline X_t) \|_\infty \leq C\|L_2^N(\overline X_{t})\|_\infty\nel X_t-\overline X_t\ner_\infty,
	\end{align}
	where we have used the local Lipschitz bound of $K_2^N$ under the event $\A_T$ (see in Lemma \ref{lmlip}).
	Furthermore, we denote
	\begin{equation}\label{eventB2}
	\B_2:=\left\{\max\limits_{t\in[0,T]}\nel L_2^N(\overline X_{t})-\overline L_2^N(\overline X_{t})\ner_\infty\leq C_{2,\alpha}N^{3\lambda_2-1}\log(N)\right\}.
	\end{equation}
	 Since Proposition \ref{propconsis} also holds for the case $\lambda_2<\frac{1}{3}$, one has
	\begin{equation}\label{estB2}
	\PP(\B_2^c)\leq N^{-\alpha}.
	\end{equation} 
	
	Under the event $\B_2$, it holds that
	\begin{equation}\label{useb2}
	\|L_2^N(\overline X_{t})\|_\infty\leq \|\overline L_2^N(\overline X_{t})\|_\infty+C_{2,\alpha}N^{3\lambda_2-1}\log(N)\leq C\log(N),
	\end{equation}
	since $\lambda_2<\frac{1}{3}$, 
	where $\|\overline L_2^N(\overline X_{t})\|_\infty\leq C\log(N)$ follows from Lemma \ref{converse}.
	Hence, one has
	\begin{equation}\label{I1'}
	\mathcal{I}_1\leq\| K_2^N(X_t)-K_2^N(\overline X_t) \|_\infty \leq C\log(N)\nel X_t-\overline X_t\ner_\infty,\quad \forall~t\in[t_n,t_{n+1}],
	\end{equation}
	under event $\A_T\cap\B_2$.

	To estimate $\mathcal{I}_2$, notice that by triangle inequality and \eqref{overlineminustilde} one has
	\begin{align}
	\norm{X_t-\widetilde{X}_t}_\infty\leq&\int_{t_n}^{t}\norm{V_s-\widetilde V_s}_\infty ds\leq
	\int_{t_n}^{t}\norm{V_s-\overline V_s}_\infty +\norm{\overline V_s- \widetilde V_s}_\infty ds
	\\\leq& \Delta t \max\limits_{s\in[t_n,t]}\left( \norm{V_s-\overline V_s}_\infty+\norm{\overline V_s- \widetilde V_s}_\infty \right)
	\\\leq& C N^{-\lambda_1-\lambda_2},
	\end{align}
	 under the event $\A_T$, which leads to
	\begin{align}
	&\norm{K_1^N(X_t)-K_1^N(\widetilde X_t) }_\infty\leq (\norm{\nabla K_1^N(X_t)}_\infty+\norm{\nabla K_1^N(\widetilde X_t)}_\infty)\norm{X_t-\widetilde{X}_t}_\infty \notag \\
	\leq& CN^{3(\delta-\lambda_2)}\norm{L^N(\overline X_t) }_\infty\norm{X_t-\widetilde{X}_t}_\infty
	\leq CN^{3\delta-\lambda_1-4\lambda_2} \norm{L^N(\overline X_t) }_\infty.
	\end{align}
	Here the bound  $\frac{\norm{\nabla K_1^N(X_t)}_\infty}{\norm{ L^N(\overline X_t)}_\infty}\leq CN^{3(\delta-\lambda_2)}$ uses Lemma \ref{lmmid} since  
	\begin{equation}
	\left\| X_t-\overline{X}_t\right\|_\infty\leq N^{-\lambda_2}\gg N^{-\delta}.
	\end{equation} 
	And a similar estimate leads to $\frac{\norm{\nabla K_1^N(\widetilde X_t)}_\infty}{\norm{ L^N(\overline X_t)}_\infty}\leq CN^{3(\delta-\lambda_2)}$.
	
	We denote the event
	\begin{equation}\label{eventB3}
	\B_3:=\left\{\max\limits_{t\in[0,T]}\nel  L^N(\overline X_{t})-\overline L^N(\overline X_{t})\ner_\infty\leq C_{2,\alpha}N^{3\delta-1}\log(N)\right\}.
	\end{equation}
	It has been proved in Proposition \ref{propconsis} that 
	\begin{equation}\label{estB3}
	\PP(\B_3^c)\leq N^{-\alpha}.
	\end{equation}
	Then under the event $\B_3$ it follows that
	\begin{equation}
	\|L^N(\overline X_{t})\|_\infty\leq \|\overline L^N(\overline X_{t})\|_\infty+C_{2,\alpha}N^{3\delta-1}\log(N) \leq CN^{3\delta-1}\log(N),
	\end{equation}
	since $\|\overline L^N(\overline X_{t})\|_\infty\leq C\log(N)$ and $\frac{1}{3}\leq \delta<1$. Thus, we have
	\begin{equation}\label{I2'}
	\mathcal{I}_2=\norm{K_1^N(X_t)-K_1^N(\widetilde X_t) }_\infty\leq CN^{6\delta-1-\lambda_1-4\lambda_2}\log(N),\quad \forall~t\in[t_n,t_{n+1}],
	\end{equation}
	under the event $\A_T\cap\B_3$.

	The estimate of $\mathcal{I}_3$	is a result of Lemma \ref{tildaX}. Indeed, we
	denote the event 
	\begin{align}\label{eventG}
	\mathcal{G}_n:=\bigg\{&\max\limits_{t\in[t_n,t_{n+1}]}\left\|  K_1^{N}(\widetilde X_t)-  K_1^{N}(\overline X_t)\right\|_\infty\leq C_{4,\alpha}N^{2\delta-1} \log(N)\notag \\
	&+C_{4,\alpha}\log^2(N)N^{3\lambda_1-\lambda_2}\|k_1^N\|_1\bigg\},
	\end{align}
	so by Lemma \ref{tildaX} one has  that
	for any $0\leq n \leq M'$
	\begin{equation}
	\label{estG}\mathbb{P}\left(\mathcal{A}_T\cap\mathcal{G}_n^c\right)\leq N^{-\alpha}.
	\end{equation}
	Furthermore, it holds  that 
	\begin{align}\label{I3'}
	\mathcal{I}_3=&\norm{ K_1^{N}(\widetilde X_t)-  K_1^{N}(\overline X_t)}_\infty
	\leq C_{4,\alpha}N^{2\delta-1}\log(N)+C_{4,\alpha}\log^2(N)N^{3\lambda_1-2\lambda_2}\notag \\
	\leq&C(\alpha,T,C_{f_0})\log^2(N)(N^{3\lambda_1-2\lambda_2}+N^{2\delta-1}), \quad \forall~t\in[t_n,t_{n+1}],
	\end{align}	
	under the event $\mathcal{G}_n$, where we have used the fact that $\norm{k_1^N}_1\leq CN^{-\lambda_2}$. 
	Indeed, it is easy to compute that
	\begin{equation}\label{k1est}
	\norm{k_1^N}_1=\norm{k^N-k_2^N}_1\leq C\int_{0\leq |x|\leq N^{-\lambda_2}}\frac{1}{|x|^2}dx\leq CN^{-\lambda_2}.
	\end{equation}
	
	Collecting \eqref{I1'},  \eqref{I2'} and \eqref{I3'} yields that 
	\begin{align*}
	&\| K^N(X_t)-K^N(\overline X_t) \|_\infty \notag\\
	\leq &C\log(N)\nel X_t-\overline X_t\ner_\infty+CN^{6\delta-1-\lambda_1-4\lambda_2}\log(N)+C\log^2(N)(N^{3\lambda_1-2\lambda_2}+N^{2\delta-1})\notag\\
	\leq&C\log(N)\nel X_t-\overline X_t\ner_\infty+C\log^2(N)(N^{6\delta-1-\lambda_1-4\lambda_2}+N^{3\lambda_1-2\lambda_2}+N^{2\delta-1}),~\forall t\in[t_n,t_{n+1}],
	\end{align*}
	under the event $\B_2\cap\B_3\cap\A_T\cap \mathcal{G}_n$, where $C$ depends on $\alpha$, $T$ and $C_{f_0}$.
	To distinguish it from other constants we will denote this $C$ by $C_{3,\alpha}$. This implies
	$\B_2\cap\B_3\cap\A_T\cap \mathcal{G}_n\subseteq \mathcal S_n(C_{3,\alpha})$, which yields that
	\begin{equation}
	\B_2\cap\B_3\cap\A_T\cap(\bigcap_{n=0}^{M'}\mathcal{G}_{n})\subseteq \left( \mathcal \bigcap_{n=0}^{M'}S_n(C_{3,\alpha})\right)=S_T(C_{3,\alpha}).
	\end{equation}
	It follows that $$\mathbb{P}\left(\A_T\cap S_T^c(C_{3,\alpha})\right)\leq \mathbb{P}\left(\B_2^c\right)+\mathbb{P}\left(\B_3^c\right)+\sum_{n=0}^{M'}\mathbb{P}\left(\A_T\cap\mathcal{G}_{n}^c\right)\leq (M'+3) N^{-\alpha}\leq 2TN^{\lambda_1-\alpha}\leq N^{-\alpha'},$$
	where we used the estimates in \eqref{estB2}, \eqref{estB3} and\eqref{estG}. Here $\alpha$ is arbitrary and so is $\alpha'$.

\end{proof}

\begin{lem}\label{tildaX} 
	Assume that the event $\A_T$ holds. Consider two trajectories $(\widetilde X_t,\widetilde V_t)$, $(\overline X_t,\overline V_t)$ on $t\in[t_n,t_{n+1}]$ satisfying \eqref{tilde}-\eqref{tilde1} and \eqref{eq:mean} respectively. When $1\leq n\leq M'$, the two different initial phases are  chosen to be $(X_{t_{n-1}},V_{t_{n-1}})$ and $(\overline X_{t_{n-1}},\overline X_{t_{n-1}})$ at time $t=t_{n-1}$, and when $n=0$ the two different initial phases are  chosen to be $(X_{0},V_{0})$ and $(\overline X_{0},\overline V_{0})$ at time $t=0$.  Then  for any $\alpha>0$, there exists a $C_{4,\alpha}>0$  depending only on $\alpha$, $T$ and $C_{f_0}$ such that for $N$ sufficiently large it holds that
	\begin{align}\label{tildaXeq}
	\PP\bigg(&\max\limits_{t\in[t_n,t_{n+1}]} \left\| K_1^{N}(\widetilde X_t)-  K_1^{N}(\overline X_t)\right\|_\infty\geq C_{4,\alpha} N^{2\delta-1}\log(N) \notag \\
	&+C_{4,\alpha} \log^2(N) N^{3\lambda_1-\lambda_2}\|k_1^N\|_1\bigg)\leq N^{-\alpha},
	\end{align}
	where we require $t_{n+1}-t_n=N^{-\lambda_1}$ with  $0<\lambda_1<\frac{\lambda_2}{3}$ and $0<\lambda_2<\frac{1}{3}$. Here $$(K_1^N(\widetilde X_t))_i=\frac{1}{N-1}\sum\limits_{j\neq i}^Nk_1^N(\widetilde x_i^t-\widetilde x_j^t)\,,\quad (K_1^N(\overline X_t))_i=\frac{1}{N-1}\sum\limits_{j\neq i}^Nk_1^N(\overline x_i^t-\overline x_j^t),\quad t\in[t_n,t_{n+1}]\,,$$  
	where $k_1^N$ is defined in \eqref{k1}.
\end{lem}
Lemma \ref{tildaX} is used in the proof of Proposition \ref{propstab}. It follows from the following estimate of the term in \eqref{tildaXeq}
at any fixed time $t\in[t_n,t_{n+1}]$, a statement  which will later be generalized to hold for the maximum of $\max t\in[t_n,t_{n+1}]$. 
\begin{lem}\label{tildaXficedtime} Under the same assumptions as in Lemma \ref{tildaX}, for any $\alpha>0$,  there exists $C_{5,\alpha}>0$  depending only on $\alpha$, $T$ and $C_{f_0}$ such that for $N$ sufficiently large it holds that  for any fixed time $t\in[t_n,t_{n+1}]$
	\begin{align}\label{tildaXfixedtimeeq}
	\PP\bigg(&\left\|  K_1^{N}(\widetilde X_t)-  K_1^{N}(\overline X_t)\right\|_\infty\geq C_{5,\alpha} N^{2\delta-1}\log(N) \notag \\
	&+C_{5,\alpha} \log^2(N) N^{3\lambda_1-\lambda_2}\|k_1^N\|_1\bigg)\leq N^{-\alpha}.
	\end{align}
\end{lem}
The proof of Lemma \ref{tildaXficedtime}  is carried  out in Section \ref{prolem}. The novel technique in the proof used the fact that  $k_1^N$ has a support with the radius $N^{-\lambda_2}$ (small).
This means that in order to contribute to the interaction, $\widetilde x_j^t$ (or $\overline x_j^t$) has to get close enough (less than $N^{-\lambda_2}$) to $\widetilde x_i^t$ (or $\overline  x_i^t$). 
Due to the effect of Brownian motion we get mixing of the positions of the particles over the whole support of $k_1^N$. Using a Law of Large Numbers argument one can show that the leading order of the interaction can in good approximation be replaced by the respective expectation value. Due to symmetry of $k_1^N$ this expectation value is zero. 
 Significant fluctuations of the interaction $k_1^N$ have very small probability. 

\begin{proof}[The proof of Lemma \ref{tildaX}]
We follow the similar procedure as in Proposition \ref{propconsis}. We divide $[t_n,t_{n+1}]$ into $M+1$ subintervals with length $\Delta \tau=N^{-\frac{\gamma}{3}}$ for some $\gamma>4$ and $\tau_k=k\Delta \tau$, $k=0,\cdots,M+1$.  Recall the event $\mathcal{\overline H}$ as in \eqref{eventH} and denote the event 
\begin{equation}\label{eventH1}
\mathcal{\widetilde H}:=\left\{\max\limits_k\max\limits_{t\in[\tau_k,\tau_{k+1}]}\nel\widetilde X_t-\widetilde X_{\tau_k}\ner_\infty\leq C_BN^{-\frac{\gamma-1}{3}}\right\}.
\end{equation}
It follows from Lemma \ref{lmQ} that
\begin{equation}\label{CHevent1}
 \PP(\mathcal{\overline H}^c),\PP(\mathcal{\widetilde H}^c)\leq C_BN^{\frac{\gamma-1}{3}}\exp(-C_BN^{\frac{2}{3}}),
\end{equation}
for any $\gamma>4$. Furthermore we denote the event
	\begin{align}
\mathcal{G}_{\tau_k}:=\bigg\{&\left\|  K_1^{N}(\widetilde X_{\tau_k})-  K_1^{N}(\overline X_{\tau_k})\right\|_\infty\leq C_{5,\alpha}N^{2\delta-1} \log(N)\notag \\
&+C_{5,\alpha}\log^2(N)N^{3\lambda_1-\lambda_2}\|k_1^N\|_1\bigg\}
\end{align}
in \eqref{tildaXfixedtimeeq},
then it follow from Lemma \ref{tildaXficedtime} that
\begin{equation}\label{Cg}
\PP(\mathcal{G}_{\tau_k}^c)\leq N^{-\alpha}.
\end{equation}

For all $t\in[\tau_k,\tau_{k+1}]$, under the event $\mathcal{G}_{\tau_k}\cap \mathcal{\overline H}\cap \mathcal{\widetilde H}$, we obtain
\begin{align*}
&\left\| K_1^{N}(\widetilde X_t)-  K_1^{N}(\overline X_t)\right\|_\infty\\
\leq& \left\| K_1^{N}(\widetilde X_t)- K_1^{N}(\widetilde X_{\tau_k})\right\|_\infty+\left\| K_1^{N}(\widetilde X_{\tau_k})-  K_1^{N}(\overline X_{\tau_k})\right\|_\infty+\left\| K_1^{N}(\overline X_{\tau_k})-  K_1^{N}(\overline X_t)\right\|_\infty\\
\leq& \norm{\nabla K_1^{N}}_\infty\left(\norm{\widetilde X_t-\widetilde X_{\tau_k}}_\infty+\norm{\overline X_t-\overline X_{\tau_k}}_\infty\right)+\left\| K_1^{N}(\widetilde X_{\tau_k})-  K_1^{N}(\overline X_{\tau_k})\right\|_\infty\\
\leq& CN^{3\delta-\frac{\gamma-1}{3}}+C_{5,\alpha}N^{2\delta-1} \log(N)+C_{5,\alpha}\log^2(N)N^{3\lambda_1-\lambda_2}\|k_1^N\|_1\\
\leq &C(\alpha,T,C_{f_0})N^{2\delta-1} \log(N)+C(\alpha,T,C_{f_0})\log^2(N)N^{3\lambda_1-\lambda_2}\|k_1^N\|_1
\end{align*}
when $\gamma>4$ is sufficiently large. This yields that under the event $\bigcap_{k=0}^M\mathcal{G}_{\tau_k}\cap\mathcal{\overline H}\cap \mathcal{\widetilde H}$ it holds that
	\begin{align*}
&\max\limits_{t\in[t_n,t_{n+1}]} \left\| K_1^{N}(\widetilde X_t)-  K_1^{N}(\overline X_t)\right\|_\infty\\
\leq &C(\alpha,T,C_{f_0})N^{2\delta-1}\log(N)+C(\alpha,T,C_{f_0})\log^2(N) N^{3\lambda_1-\lambda_2}\|k_1^N\|_1.
\end{align*}
Therefore it follows from \eqref{CHevent1} and \eqref{Cg} that
	\begin{align}\label{gnesti}
\PP\bigg(&\max\limits_{t\in[t_n,t_{n+1}]} \left\| K_1^{N}(\widetilde X_t)-  K_1^{N}(\overline X_t)\right\|_\infty\geq C(\alpha,T,C_{f_0})N^{2\delta-1}\log(N) \notag \\
&+C(\alpha,T,C_{f_0}) \log^2(N) N^{3\lambda_1-\lambda_2}\|k_1^N\|_1\bigg)\leq \sum_{k=0}^{M}\PP(\mathcal{G}_{\tau_k}^c)+\PP(\mathcal{\overline H}^c)+\PP(\mathcal{\widetilde H}^c)\notag\\
\leq& TN^{-\frac{3\alpha-\gamma}{3}}+2C_BN^{\frac{\gamma-1}{3}}\exp(-C_BN^{\frac{2}{3}})\leq N^{-\alpha'}.
\end{align}
Denote $C_{4,\alpha'}$ to be the constant $C(\alpha, T, C_{f_0})$ in \eqref{gnesti}. Since $\alpha>0$ is arbitrary and so is $\alpha'$, \eqref{tildaXeq} holds true. This completes the proof of Lemma \ref{tildaX}.
\end{proof}
\subsection{Convergence and the proof of Theorem \ref{mainthm}}
In this section, we achieve the convergence by using the consistency from Proposition \ref{propconsis} and the stability from Proposition \ref{propstab}.
To do this, we first prove the following Gronwall-type inequality.
\begin{lem}\label{lmprior}
	For any $T>0$, let $e(t)$ be a non-negative continuous function on $[0,T]$ with the initial data $e(0)=0$ and $\lambda_2,~\lambda_3$ be two universal constants satisfying $0<\lambda_2<\lambda_3$. Assume that for any $0< T_1\leq T$ the function $e(t)$ satisfies the following differential inequality that holds with $C>0$ independent of $N>0$
	\begin{equation}\label{priorineq}
		\frac{de(t)}{dt}\leq C\sqrt{\log(N)}e(t)+C\log^2(N)N^{-\lambda_3},\quad 0<t\leq T_1,
	\end{equation}
	provided that 
	\begin{equation}\label{priorcon}
	\max\limits_{t\in[0,T_1]}e(t)\leq N^{-\lambda_2},
	\end{equation}
	holds.  Then $e(t)$ is uniformly bounded on $[0,T]$. Furthermore there is a $N_0\in\mathbb{N}$ depending only on $C$ and $T$  such that for all $N\geq N_0$
	\begin{equation}\label{priorres}
	\max\limits_{t\in[0,T]}e(t)\leq  N^{-\lambda_2}.
	\end{equation}

\end{lem}
\begin{rmk}\label{ATexplain}
	The lemma is in fact a generalization of Gronwall's Lemma. In Gronwall's Lemma it is assumed that  \eqref{priorineq}  holds and \eqref{priorres} is the consequence. Here we have a weaker condition:  namely we  assume that \eqref{priorineq} holds under the additional assumption \eqref{priorcon}.
	But as long as \eqref{priorineq} holds we can control the growth of $e(t)$ via Gronwall's inequality to make sure that \eqref{priorres} remains  valid on an even larger interval.  	 
\end{rmk}

\begin{proof}
 We prove the lemma by contradiction: we assume that there is a $t\in [0,T]$ with $e(t)\geq N^{-\lambda_2}$ and show that for $N\geq N_0$ with some $N_0\in\mathbb{N}$ specified below, we get a contradiction. 

It follows that the infimum over all times $t$ where $e(t)$ is larger than or equal to  $N^{-\lambda_2}$ exists and we define  $$T_*=\inf_{}\{0\leq t \leq T:e(t)\geq N^{-\lambda_2}\}.$$
We get by  continuity of $e(t)$ together with $e(0)=0$ that $T^*>0$,
\begin{equation}\label{contra}e(T_*)=N^{-\lambda_2}\text{ and }\max_{0\leq t \leq T_*}e(t)=N^{-\lambda_2}\;.
\end{equation}
 
Since  \eqref{priorcon} implies \eqref{priorres}, we get for $T_1=T_*$ that

$$\frac{de(t)}{dt}\leq C\sqrt{\log(N)}e(t)+C\log^2(N)N^{-\lambda_3},\quad 0<t\leq T_*.$$
Gronwall's Lemma gives that
$$e(t)\leq e^{C\sqrt{\log(N)}t}\log^2(N)N^{-\lambda_3},$$
in particular $$e(T_*)\leq e^{C\sqrt{\log(N)}T_*}\log^2(N)N^{-\lambda_3}.$$

Since $ e^{C\sqrt{\log(N)}T_*}$ and $\log^2(N)$ are  asymptotically bounded by any positive power of $N$, we can find a $N_0\in\mathbb{N}$  depending only on $C$ and $T_*$ such that for any $N\geq N_0$
$$ e^{C\sqrt{\log(N)}T_*}\log^2(N)< N^{\lambda_3-\lambda_2},\quad \mbox{ for }0<\lambda_2<\lambda_3,$$
 and hence
$$e(T_*)< N^{-\lambda_2}\text{ for any }N\geq N_0\:.$$

Thus  we get a contradiction to \eqref{contra} for all $N\geq N_0$ and the lemma is proven.

\end{proof}

We now return to the proof of Theorem \ref{mainthm}.
Denote the event
\begin{equation}
\mC_{T}:=\left\{\max\limits_{t\in[0,T]}\nel K^N(\overline X_t)-\overline K^N(\overline X_t)\ner_\infty\leq C_{2,\alpha} N^{2\delta-1}\log (N)\right\},
\end{equation}
and consider the quantity $e(t)$ defined as
\begin{equation}
e(t):=\norm{\Phi_t-\Psi_t}_\infty =\sqrt{\log(N)}\nel X_t-\overline X_t\ner_\infty+\nel V_t-\overline V_t\ner_\infty.
\end{equation}
Recall that $$\A_{T}=\left\{\max\limits_{t\in[0,T]} e(t)\leq N^{-\lambda_2}\right\}.$$

To prove the theorem we will show that under the assumptions $\mathcal{C}_{T}$ and $\mathcal{A}_{T}^c\cup\mathcal{S}_{T}(C_{3,\alpha})$ it follows that 
\begin{equation}\label{statement}\max\limits_{t\in[0,T]}e(t)\leq N^{-\lambda_2}.\end{equation}
Let us explain why proving \eqref{statement} under the assumptions $\mathcal{C}_{T}$ and $\mathcal{A}_{T}^c\cup\mathcal{S}_{T}(C_{3,\alpha})$ proves the theorem:
Since
 $\mathcal{C}_{T}$ is the consistency in Proposition \ref{propconsis}, i.e. $\mathbb{P}\left(\mathcal{C}_{T}^c\right)\leq N^{-\alpha}$, and by Proposition \ref{propstab} one has $\mathbb{P}\left(\mathcal{A}_{T}\cap\mathcal{S}_{T}^c(C_{3,\alpha})\right)\leq N^{-\alpha}$,
this implies that 
$$\PP\left(\max\limits_{t\in[0,T]}\norm{\Phi_t-\Psi_t}_\infty \geq N^{-\lambda_2}\right)\leq \PP(\mathcal{C}_{T}^c) +\PP(\A_{T}\cap\mathcal{S}_{T}^c(C_{3,\alpha}) )\leq 2N^{-\alpha}\;.$$
It follows that for any $\alpha>0$, there exists some  $N_0\in \mathbb{N}$ such that 
$$\PP\left(\max\limits_{t\in[0,T]}\norm{\Phi_t-\Psi_t}_\infty \leq N^{-\lambda_2}\right)\geq 1- N^{-\alpha}$$
for all $N\geq N_0$, which proves Theorem \ref{mainthm}.

To prove the statement \eqref{statement} we use Lemma \ref{lmprior}.  We will show that for any $0<T_1\leq T$, under the additional assumption $\A_{T_1}$, the following differential inequality holds
\begin{equation}\label{gwin}
\frac{d e(t)}{dt} 
\leq C \sqrt{\log(N)} e(t)  +C \log^2(N)N^{-\lambda_3},\mbox{ for all } t\in(0,T_1]\,,
\end{equation}
for some $\lambda_3>\lambda_2$.
Then Lemma  \ref{lmprior} states that in fact \eqref{statement} holds which, as explained above, proves Theorem \ref{mainthm}.  Note that since $e(0)=0$, according to the general Gronwall's inequality in Lemma  \ref{lmprior}, the assumption $\A_{T_1}$ can be removed.

Since $\A_{T}\subseteq \A_{T_1}$, we have to prove \eqref{gwin} under the assumption that $\A_{T}\cap\mathcal{C}_{T}\cap(\mathcal{A}_{T}^c\cup\mathcal{S}_{T}(C_{3,\alpha}))=\A_{T}\cap\mathcal{C}_{T}\cap\mathcal{S}_{T}(C_{3,\alpha})$ holds.
Let us recall the assumptions $\mathcal{C}_{T}$, $\mathcal{S}_{T}(C_{3,\alpha})$ and $\A_{T}$  for easier reference. They hold if
\begin{align}
(i)\;\;&\max\limits_{t\in[0,T]}\nel K^N(\overline X_t)-\overline K^N(\overline X_t)\ner_\infty\leq C_{2,\alpha} N^{2\delta-1}\log (N)\label{assum1},\\
(ii)\;\;&\| K^N(X_t)-K^N(\overline X_t) \|_\infty \leq C_{3,\alpha}\log(N)\nel X_t-\overline X_t\ner_\infty\notag\\
&+C_{3,\alpha}\log^2(N)(N^{6\delta-1-\lambda_1-4\lambda_2}+N^{3\lambda_1-2\lambda_2}+N^{2\delta-1}),~\forall~t\in[0,T]\label{assum2} ,\\
(iii)\;\;&\max\limits_{0\leq t\leq T}e(t)\leq N^{-\lambda_2}.
\end{align}
Notice that for any $0<T_1\leq T$
\begin{equation}
\A_{T}\subseteq \A_{T_1},\quad \mC_{T}\subseteq \mC_{T_1}, \quad \mathcal{S}_{T}(C_{3,\alpha})\subseteq \mathcal{S}_{T_1}(C_{3,\alpha}).
\end{equation}
Using the fact that $\frac{d\|x\|_{\infty}}{dt}\leq \|\frac{dx}{dt}\|_{\infty}$, one has for all $t\in(0,T_1]$
\begin{align}
\frac{d e(t)}{dt}&\leq \sqrt{\log(N)}\nel V_t-\overline V_t\ner_\infty+\nel K^N(X_t)-\overline K^N(\overline X_t)\ner_\infty \notag \\
&\leq \sqrt{\log(N)}\nel V_t-\overline V_t\ner_\infty+\nel K^N(X_t)- K^N(\overline X_t)\ner_\infty +\nel K^N(\overline X_t)-\overline K^N(\overline X_t)\ner_\infty.
\end{align}
It follows that
\begin{align}\label{et}
\frac{d e(t)}{dt}
&\leq \sqrt{\log(N)}\nel V_t-\overline V_t\ner_\infty \notag \\
&~~+C_{3,\alpha}\log(N)\nel X_t-\overline X_t\ner_\infty+C_{3,\alpha}\log^2(N)(N^{6\delta-1-\lambda_1-4\lambda_2}+N^{3\lambda_1-2\lambda_2}+N^{2\delta-1})\notag \\
&~~+C_{2,\alpha} N^{2\delta-1}\log (N)\notag\\
&\leq C(\alpha,T,C_{f_0})\sqrt{\log(N)} e(t) \notag \\
&\quad+C(\alpha,T,C_{f_0})\log^2(N)(N^{6\delta-1-\lambda_1-4\lambda_2}+N^{3\lambda_1-2\lambda_2}+N^{2\delta-1})\notag\\
&\leq  C(\alpha,T,C_{f_0})\sqrt{\log(N)} e(t) +C(\alpha,T,C_{f_0})\log^2(N) N^{-\lambda_3},
\end{align}
where in the first inequality we used  assumptions \eqref{assum1} and \eqref{assum2} and  in the second inequality we used the fact that 
\begin{equation}
\sqrt{\log(N)}\nel V_t-\overline V_t\ner_\infty+\log(N)\nel X_t-\overline X_t\ner_\infty=\sqrt{\log(N)} e(t).
\end{equation}
Here we denote
\begin{equation}
-\lambda_3:=\max\left\{6\delta-1-\lambda_1-4\lambda_2,3\lambda_1-2\lambda_2,2\delta-1\right\}.
\end{equation}

Notice that for
\begin{equation}
0<\lambda_2<1/3;~0<\lambda_1<\frac{\lambda_2}{3};~\frac{1}{3}\leq\delta<\min\left\{\frac{\lambda_1+3\lambda_2+1}{6},\frac{1-\lambda_2}{2}\right\},
\end{equation}
one has $-\lambda_3<-\lambda_2$. In other words, we obtain that for $\lambda_2<\lambda_3$
\begin{align}
\frac{d e(t)}{dt}\leq C(\alpha,T,C_{f_0})\sqrt{\log(N)} e(t) +C(\alpha,T,C_{f_0})\log^2(N) N^{-\lambda_3},\mbox{ for all } t\in(0,T_1]\,,
\end{align}
which verifies \eqref{gwin} and the theorem is proven.

\section{Proof of Theorem \ref{cor}}
In order to prove the error estimate between $f_t$ and $\mu_\Phi(t)$, let us split the error into three parts
\begin{align}
W_p(f_t,\mu_\Phi(t))&\leq W_p(f_t,f_t^N)+W_p(f_t^N,\mu_\Psi(t))+W_p(\mu_\Psi(t),\mu_\Phi(t)).
\end{align}
The   Theorem \ref{cor} is proven once we   obtain the respective error estimates of those three parts.
\begin{proof}[Proof of Theorem \ref{cor}]
	
	$\bullet$\textit{The first term $W_p(f_t,f_t^N)$}.  The convergence of this term is a deterministic result: solutions of the regularized VPFP equations \eqref{vlasov} approximate solutions of the original VPFP equations \eqref{vlasovoriginal} as the width of the cut-off goes to zero.
	It follows from \cite[Lemma 3.2]{carrillo2018propagation} that
	\begin{equation}\label{term1}
	\max\limits_{t\in[0,T]}W_p(f_t,f_t^N)\leq N^{-\delta}e^{C_1\sqrt{\log(N)}},
	\end{equation}
	where $p\in[1,\infty)$, $N>3$ and $C_1$ depends only on $T$ and $C_{f_0}$. The proof is inspired by the method of Leoper \cite{loeper2006uniqueness}. Note that here we can't follow the method in \cite{lazarovici2015mean} directly since the support of $f^N$ and $f$ are not compact in our present case.
	
	$\bullet$\textit{The second term $W_p(f_t^N,\mu_\Psi(t))$}. This term concerns the sampling of the mean-field dynamics by discrete particle trajectories. The convergence rate has been proved in \cite[Corollary 9.4]{lazarovici2015mean} by using the concentration estimate of Fournier and Guillin \cite{fournier2015rate}. We summarize the result as follows: let $p\in[1,\infty)$, $\kappa<\min\{\delta,\frac{1}{6},\frac{1}{2p}\}$ and $N>3$. Assume that there exists $m>2p$ such that $$\iint_{\RR^6}(|x|^m+|v|^m)f_0(x,v)dxdv<+\infty.$$ Then there exist  constants $C_2$ and $C_3$ such that it holds
	\begin{align}\label{term2}
	\PP\bigg(\max\limits_{t\in[0,T]}W_p(f_t^N,\mu_\Psi(t))&\leq \sqrt{\log (N)}N^{-\kappa}e^{C_2\sqrt{\log(N)}}\bigg)\notag \\
	&\geq 1-C_3\left(e^{-C_4N^{1-\max\{6,2p\}\kappa}}+N^{1-\frac{m}{2p}}\right).
	\end{align}
	
	$\bullet$\textit{The third term $W_p(\mu_\Psi(t),\mu_\Phi(t))$}. The convergence of this term is a direct result of Theorem \ref{mainthm}. Indeed, it follows from \cite[Lemma 5.2]{lazarovici2015mean} that for all $p\in[0,\infty]$
	\begin{equation}
	\max\limits_{t\in[0,T]}W_p(\mu_\Psi(t),\mu_\Phi(t))\leq \max\limits_{t\in[0,T]}\norm{\Psi(t)-\Phi(t)}_\infty.
	\end{equation}
	Then we choose $\alpha=\frac{m}{2p}-1$ in Theorem \ref{mainthm} so that
	\begin{equation}\label{term3}
	\mathbb{P}\left(\max\limits_{t\in[0,T]}W_p(\mu_\Psi(t),\mu_\Phi(t)) \leq N^{-\lambda_2} \right)\geq 1-N^{1-\frac{m}{2p}}.
	\end{equation}
	
	$\bullet$\textit{Convergence of $W_p(f_t,\mu_\Phi(t))$}. Collecting estimates \eqref{term1}, \eqref{term2} and \eqref{term3} and choosing $\kappa <\min\{\delta,\frac{1}{6},\frac{1}{2p}\}$, it follows  that
	\begin{align}
	\PP\bigg(\max\limits_{t\in[0,T]}W_p(f_t,\mu_\Phi(t))&\leq (1+\sqrt{\log (N)})N^{-\kappa}e^{C_5\sqrt{\log(N)}}+N^{-\lambda_2}\bigg)\notag\\
	&\geq 1-C_6\left(e^{-C_7N^{1-\max\{6,2p\}\kappa}}+N^{1-\frac{m}{2p}}\right),
	\end{align}
	where $C_5$ depends only on $T$ and $C_{f_0}$, and $C_6$, $C_7$ depend only on $m$, $p$, $\kappa$. We can simplify this result by demanding $N\geq e^{\left(\frac{2C_5}{1-3\lambda_2}\right)^2}$, which yields $N^{1-3\lambda_2}\geq (1+\sqrt{\log (N)})e^{C_5\sqrt{\log(N)}}$. Hence we conclude that
	\begin{align}
	\PP\bigg(\max\limits_{t\in[0,T]}W_p(f_t,\mu_\Phi(t))&\leq N^{-\kappa+1-3\lambda_2}+N^{-\lambda_2}\bigg)\notag\\
	&\geq 1-C_6\left(e^{-C_7N^{1-\max\{6,2p\}\kappa}}+N^{1-\frac{m}{2p}}\right).
	\end{align}
\end{proof}

\section{The proof of Lemma \ref{tildaXficedtime}}\label{prolem}
In this section, we present the proof of Lemma \ref{tildaXficedtime}, which provides the distance between $K_1^{N}(\widetilde X_t)$ and $K_1^{N}(\overline X_t)$ ($t\in[t_n,t_{n+1}]$), where  $(\widetilde X_t,\widetilde V_t)$, $(\overline X_t,\overline V_t)$ satisfying \eqref{tilde}-\eqref{tilde1} and \eqref{eq:mean} respectively with two different initial phases $(X_{t_{n-1}},V_{t_{n-1}})$ and $(\overline X_{t_{n-1}},\overline X_{t_{n-1}})$ at time $t=t_{n-1}$ when $1\leq n\leq M'$, or $(X_{0},V_{0})$ and $(\overline X_{0},\overline V_{0})$ at time $t=0$ when $n=0$. 
To do this, we introduce the following stochastic process: For time $0\leq s \leq t$ and $a:=(a_x,a_v)\in \RR^{6N}$, let $Z^{a,N}_{t,s}:=(Z^{a,N}_{x,t,s},Z^{a,N}_{v,t,s})$ be the process starting at time $s$ at the position $(a_x,a_v)$ and evolving from time $s$ up to time $t$ according to the mean-field force $\overline K^N$:
\begin{align}\label{lastprocess}
\begin{cases}
dZ^{a,i,N}_{x,t,s}&=Z^{a,i,N}_{v,t,s} dt,\quad  t>s,\\
dZ^{a,i,N}_{v,t,s}&=\int_{\RR^3} k^N(Z^{a,i,N}_{x,t,s}-x)\rho^N(x,t)dx+\sqrt{2\sigma}dB_i^t,\quad i=1,\cdots,N,
\end{cases}
\end{align}
and 
\begin{equation}
(Z^{a,i,N}_{x,s,s},Z^{a,i,N}_{v,s,s})=(a_x^i,a_v^i),\quad \mbox{at }t=s.
\end{equation}
Note that here $(Z^{a,i,N}_{x,t,s},Z^{a,i,N}_{v,t,s})$, $i=1,\cdots,N$ are independent.
Furthermore $(Z^{a,N}_{x,t,s},Z^{a,N}_{v,t,s})$ has the strong Feller property  (see \cite{Girsanov59} Definition (A)),  implying in particular that it has a transition probability density $u^{a,N}_{t,s}$ which is given by
the product $u^{a,N}_{t,s}:=\prod_{i=1}^N u^{a,i,N}_{t,s}$. Hence each term $u^{a,i,N}_{t,s}$ is the transition probability density of $(Z^{a,i,N}_{x,t,s},Z^{a,i,N}_{v,t,s})$ 
and is also the solution to the linearized equation for $t>$:
\begin{equation}
\partial _t u^{a,i,N}_{t,s}+v\cdot\nabla_x u^{a,i,N}_{t,s}+k^N\ast \rho^N\cdot \nabla_v u^{a,i,N}_{t,s}=\Delta_vu^{a,i,N}_{t,s},\quad u^{a,i,N}_{s,s}=\delta_{a_i},
\end{equation}
where $\rho^N=\int_{\RR^3}f^N(t,x,v)dv$, and $f^N$ solves the regularized VPFP  equations \eqref{vlasov} with initial condition $f_0$.

Consider now the process $Z^{a,N}_{t,s}$ and $Z^{b,N}_{t,s}$ for two different starting points $a,b\in \RR^{6N}$. It is intuitively clear that the probability density $u^{a,i,N}_{t,s}$ and $u^{b,i,N}_{t,s}$ are just a shift of each other. The next lemma gives an estimate for the distance between any two densities
in terms of the distance between the starting points $a$ and $b$ and the elapsed time $t-s$. The proof is carried out in Appendix A.
\begin{lem}\label{transition}
	\label{lemma}There exists a positive constant $C$ depending only on $C_{f_0}$ and
	$T$ such that for each $N \in \mathbb{N}$, any starting points $a,b \in
	\mathbb{R}^{6 N}$ and any time $0 < t \leqslant T$, the following estimates
	for the transition probability densities $u_{t, s}^{a,i,N}$ resp. $u_{t,
		s}^{b,i,N}$ of the processes $Z_{t, s}^{a,i,N}$ resp. $Z_{t, s}^{b,i,N}$ given
	by (\ref{lastprocess}) hold for $t-s<\min\{1,T-s\}$:	
		\begin{enumerate}[(i)]
					\item $ \| u_{t, s}^{a,i,N} \|_{\infty,1}
				\leqslant C \left((t - s)^{-\frac{9}{2}}+1\right)$,
		\item$\| u_{t, s}^{a,i,N} -
		u_{t, s}^{b,i,N} \|_{\infty,1} \leqslant C  | a - b|\left((t - s)^{- 6}+1\right).$
			\end{enumerate} 
		The norm $\|\cdot\|_{p,q}$ denotes the $p$-norm in the $x$ and $q$-norm in the $v$-variable, i.e. for any $f:\mathbb{R}^3\times\mathbb{R}^3\to\mathbb{R} $
		\begin{equation}
		\|f\|_{p,q}:=\left(\int_{\RR^3} \left(\int_{\RR^3} |f(x,v)|^q dv\right)^{p/q}dx\right)^{1/p}.
		\end{equation}
\end{lem}

To this end one assumes $\Delta t=t_{n+1}-t_n=N^{-\lambda_1}$. Next we define for $t\in[t_n,t_{n+1}]$ the random sets 
	\begin{equation}\label{Mtn}
	M_{t_n}^t:=\left\{2\leq j\leq N: \left| x_1^{t_n}- x_j^{t_n}+(t-t_n)(v^{t_n}_1- v^{t_n}_j)\right|\leq N^{-\lambda_2}+ \log (N) \Delta t^{\frac{3}{2}} \right\} 
	\end{equation} 
	and 
	\begin{equation}\label{-Mtn}
	\overline M_{t_n}^t:=\left\{2\leq j\leq N: \left| \overline x_1^{t_n}- \overline x_j^{t_n}+(t-t_n)(\overline v^{t_n}_1- \overline v^{t_n}_j)\right|\leq 3N^{-\lambda_2}+ \log (N) \Delta t^{\frac{3}{2}}  \right\} .
	\end{equation}
Here $M_{t_n}^t$ is at time $t_n$ the set of indices of those  particles $x_j^{t_n}$ which are in the ball of  radius $N^{-\lambda_2}+ \log (N)  \Delta t^{\frac{3}{2}} $ around $x_1^{t_n}+(t-t_n)(v^{t_n}_1- v^{t_n}_j)$, and $\overline M_{t_n}^t$ is an  intermediate set introduced to help to control $M_{t_n}^t$.
 Note that under the event $\mathcal{A}_T$, we have $M_{t_n}^t\subseteq \overline M_{t_n}^t$.

We also define random sets for $t\in[t_n,t_{n+1}]$
\begin{equation}\label{Stn}
\mathcal{S}_{t_n}^t=\left\{\mbox{card }(M_{t_n}^t)< 2C_\ast  N\left(3N^{-\lambda_2}+ \log (N)  \Delta t^{\frac{3}{2}} \right)^2 \right\},
\end{equation}
and
\begin{equation}\label{-Stn}
\overline{\mathcal{S}}_{t_n}^t=\left\{\mbox{card }(\overline M_{t_n}^t) < 2C_\ast N\left(3N^{-\lambda_2}+ \log (N)  \Delta t^{\frac{3}{2}} \right)^2 \right\}\;,
\end{equation}
where $C_\ast$ will be defined later. Here $\mathcal{S}_{t_n}^t$ indicates the event where the number of particles inside the set $M_{t_n}^t$ is smaller than $2C_\ast  N\left(3N^{-\lambda_2}+ \log (N)  \Delta t^{\frac{3}{2}} \right)^2  $, and the  event $\overline{\mathcal{S}}_{t_n}^t$ is introduced to help estimate $\PP(\mathcal{S}_{t_n}^t)$.
 
Our next lemma provides the probability estimate of the event where particle $\widetilde x_{j}^t$ (or $\overline x_{j}^t$) is close to $\widetilde x_{1}^t$ (or $\overline x_{1}^t$) (distance smaller than $N^{-\lambda_2}$) during a short time interval $t-t_n$, which contributes to the interaction of $k_1^N$ defined in  \eqref{k1}, since the support of $k_1^N$ has radius $N^{-\lambda_2}$. 
\begin{lem}\label{unlikely}
Let $(\widetilde x_{j}^t,\widetilde v_{j}^t)$  satisfy \eqref{tilde} and  \eqref{tilde1} on $t\in[t_n,t_{n+1}]$ and the random   set $M_{t_n}^t$ satisfy \eqref{Mtn}, then for any $\alpha>0$, there exists some constant $N_0>0$ depending only on $\alpha$, $T$ and $C_{f_0}$ such that for all $N\geq N_0$ it holds
\begin{align*}\PP\left(\min_{t\in[t_n,t_{n+1}]}\max_{j\in( M^ t_{t_n})^c}\left\{\left|\widetilde x_{1}^t-\widetilde x_{j}^t\right|\right\}<N^{-\lambda_2}\;\;\right)&\leq N^ {-\alpha},
\\\PP\left(\min_{t\in[t_n,t_{n+1}]}\max_{j\in(\overline M^ t_{t_n})^c}\left\{\left|\overline x_{1}^t-\overline x_{j}^t\right|\right\}<N^{-\lambda_2}\;\;\right)&\leq  N^ {-\alpha}.
\end{align*}
This means that for some particle index $j$ outside $M_{t_n}^t$, $\widetilde x_{j}^t$ for some $t\in[t_n,t_{n+1}]$ such that $\left|\widetilde x_{1}^t-\widetilde x_{j}^t\right|<N^{-\lambda_2}$ (i.e. $\widetilde x_{j}^t$  contributes to the interaction of $k_1^N$) with probability less than $N^ {-\alpha}$. Here $\PP$ is understood to be taken on the initial condition $\widetilde x_{j}^{t_n}$.
\end{lem}

\begin{proof}

Let $(1,j)$ be fixed and $a_1^t:=(a^t_{1,x},a^t_{1,v}),b_j^t:=(b^t_{j,x},b^t_{j,v})\in\mathbb{R}^6$ satisfy the stochastic differential equations
$$da^t_{1,x}=a^ t_{1,v} dt,~da^ t_{1,v}= \sqrt{2\sigma}dB_1^ t;\quad db^t_{j,x}=b^t_{j,v} dt,~db^t_{j,v}=\sqrt{2\sigma}dB_j^ t,\quad t_n<t\leq t_{n+1},$$
with the initial data $a_1^{t_n}=0$ and  $b_j^{t_n}=0$. Here $B_j^t$ is the same as in \eqref{tilde}. It follows from the evolution equation \eqref{tilde} that
$$d \left(\widetilde x_1^t - a^t_{1,x}\right)= \left(\widetilde v^{t}_1 - a^t_{1,v}\right)dt \text{ and }d \left(\widetilde v^{t}_1 - a^t_{1,v}\right)= \overline k^N(\widetilde x^ t_1) dt,\quad t_n<t\leq t_{n+1},$$
and
$$d \left(\widetilde x^{t}_j - b^t_{j,x}\right)= \left(\widetilde v^{t}_j - b^t_{j,v}\right)dt \text{ and }d \left(\widetilde v^{t}_j - b^t_{j,v}\right)= \overline k^N(\widetilde x^ t_j) dt,\quad t_n<t\leq t_{n+1}, $$
where $$\overline k^N:=k^N\ast\rho^N,$$ which is bounded by  $\norm{ \overline k^N}_\infty\leq C(\norm{\rho^N}_1+\norm{\rho^N}_\infty)$ according to Lemma \ref{converse}.
Integrating twice we get for any $s\geq t_n$
$$ \left(\widetilde v^{s}_j - b^s_{j,v}\right)= \widetilde v^{t_n}_j +\int_{t_n}^s \overline k^N(\widetilde x^\tau_j) d\tau$$ and
$$ \left(\widetilde x^{t}_j - b^t_{j,x}\right)=\widetilde x_j^{t_n}+\int_{t_n}^t \left(\widetilde v^{t_n}_j +\int_{t_n}^s \overline k^N(\widetilde x^\tau_j) d\tau\right) ds\;.$$
And by the same argument one has 
\begin{align*}
\widetilde x^{t}_1-\widetilde x^{t}_j-(a^t_{1,x}-b^t_{j,x})=\widetilde x_1^{t_n}-\widetilde x_j^{t_n}+ 
\int_{t_n}^t \left(\widetilde v^{t_n}_1-\widetilde v^{t_n}_j+\int_{t_n}^s \overline k^N(\widetilde x^\tau_1) d\tau -\int_{t_n}^s \overline k^N(\widetilde x^\tau_j) d\tau\right) ds\;.
\end{align*}
Since $\overline k^N(\widetilde x^\tau_j)$ is bounded by  $\norm{ \overline k^N}_\infty\leq C(\norm{\rho^N}_1+\norm{\rho^N}_\infty)$ according to Lemma \ref{converse}., it follows that there is a constant $0<C<\infty$  depending only on $\norm{\overline k^N}_\infty$ such that
\begin{align}\label{esta}
\left|\widetilde x^{t}_1-\widetilde x^{t}_j\right|\geq \left|\widetilde x_1^{t_n}-\widetilde x_j^{t_n}+(t-t_n)(\widetilde v^{t_n}_1-\widetilde v^{t_n}_j)\right|-\left|a^t_{1,x}-b_{j,x}^t\right|-C\Delta t^2, \mbox{ for all } t\in[t_n,t_{n+1}].
\end{align}

For $j\in (M_{t_n}^t)^c$ for some $t\in[t_n,t_{n+1}]$, i.e.
\begin{equation}\label{161}
\left|\widetilde x_1^{t_n}-\widetilde x_j^{t_n}+(t-t_n)(\widetilde v^{t_n}_1-\widetilde v^{t_n}_j)\right|\geq N^{-\lambda_2}+\log (N) \Delta t^{\frac{3}{2}}
\end{equation}
 together with $\min\limits_{t\in[t_n,t_{n+1}]}\max\limits_{j\in( M^ t_{t_n})^c}\left\{\left|\widetilde x_{1}^t-\widetilde x_{j}^t\right|\right\}<N^{-\lambda_2}$,
\eqref{esta} and \eqref{161} imply
$$\max_{t\in[t_n,t_{n+1}]}\min_{j\in( M^ t_{t_n})^c}\left\{\left|a^t_{1,x}-b_{j,x}^ t\right|\right\}>(\ln N-C)  \Delta t^{\frac{3}{2}}.$$
 Hence
\begin{align}\label{min1}\nonumber
&\PP\left(\min_{t\in[t_n,t_{n+1}]}\max_{j\in( M^ t_{t_n})^c}\left\{\left|\widetilde x_{1}^t-\widetilde x_{j}^t\right|\right\}<N^{-\lambda_2}\;\;\right) \notag\\
\leq&\PP \left(\max_{t\in[t_n,t_{n+1}]}\min_{j\in( M^ t_{t_n})^c}\left\{\left|a^t_{1,x}-b_{j,x}^ t\right|\right\}>(\ln N-C)  \Delta t^{\frac{3}{2}}\right) \notag
\\\leq&\PP \left(\max_{t\in[t_n,t_{n+1}]}\min_{j\in( M^ t_{t_n})^c}\left\{\left|a^t_{1,v}-b_{j,v}^ t\right|\right\}>(\ln N-C)  \Delta t^{\frac{1}{2}}\right),
\end{align}
where we used $a_x^ t=\int_{t_n}^ta_v^sds$ and $b_x^ t=\int_{t_n}^tb_v^sds$ in the second inequality. In the same way we can argue that
\begin{align}\label{min2}
&\PP\left(\min_{t\in[t_n,t_{n+1}]}\max_{j\in(\overline M^ t_{t_n})^c}\left\{\left|\overline x_{1}^t-\overline x_{j}^t\right|\right\}<N^{-\lambda_2}\;\;\right)\notag\\
\leq&\PP \left(\max_{t\in[t_n,t_{n+1}]}\min_{j\in(\overline M^ t_{t_n})^c}\left\{\left|a^t_{1,v}-b_{j,v}^ t\right|\right\}>(\ln N-C) \Delta t^{\frac{1}{2}}\right)\;.
\end{align}

Due to independence the difference $c_{j,v}^t=(c_{j,1}^t,c_{j,2}^t,c_{j,3}^t)=a_{1,v}^t-b_{j,v}^t$ is itself a Wiener process \cite{wiener1923} since
\begin{equation}
dc_{j,v}^t=d (a_{1,v}^t-b_{j,v}^t)= d(B_1^t-B_j^t).
\end{equation}
 Splitting up this Wiener process into its three spacial components we get
\begin{align}\label{estc}
&\PP \left(\max_{t\in[t_n,t_{n+1}]}\min_{j\in( M^ t_{t_n})^c}\left\{\left|a^t_{1,v}-b_{j,v}^ t\right|\right\}>(\ln N-C) \Delta t^{\frac{1}{2}}\right) \notag\\
\leq& 3\PP \left(\max_{t\in[t_n,t_{n+1}]}\min_{j\in( M^ t_{t_n})^c}\left\{\left|c^ t_{j,1}\right|\right\}>(\ln N-C) \Delta t^{\frac{1}{2}}\right)\notag\\
\leq &6\PP \left(\max_{t\in[t_n,t_{n+1}]} \min_{j\in( M^ t_{t_n})^c}\left\{c^ t_{j,1}\right\} >(\ln N-C) \Delta t^{\frac{1}{2}}\right)\notag\\
=&12 \PP \left( \min_{j\in( M^ t_{t_n})^c} \left\{c^ {t_{n+1}}_{j,1}\right\}>(\ln N-C) \Delta t^{\frac{1}{2}}\right)\;.
\end{align}
 where in the last equality we used the reflection principle based on the Markov property \cite{Levy40}.

Recall that the time evolution of $a_{1,v}^ t$ and $b_{j,v}^t$ are standard Brownian motions, i.e. the density is a Gaussian with standard deviation  $\sigma_t=\sigma(t-t_n)^{\frac{1}{2}}$. Due to the independence of $a^t_{1,v}$ and $b^t_{j,v}$,  $c_{j,1}^t$ is also normal distributed with the standard deviation of order $(t-t_n)^{\frac{1}{2}}$. Hence for $N$ sufficiently large, following from \eqref{estc}, it holds that
$$\PP \left(\max_{t\in[t_n,t_{n+1}]}\min_{j\in( M^ t_{t_n})^c}\left\{\left|a^t_{1,v}-b_{j,v}^ t\right|\right\}>(\ln N-C) \Delta t^{\frac{1}{2}}\right)\leq N^{-\alpha}\,$$
and 
$$\PP \left(\max_{t\in[t_n,t_{n+1}]}\min_{j\in(\overline M^ t_{t_n})^c}\left\{\left|a^t_{1,v}-b_{j,v}^ t\right|\right\}>(\ln N-C) \Delta t^{\frac{1}{2}}\right)\leq N^{-\alpha}\,$$
With \eqref{min1} and \eqref{min2} the lemma follows.

\end{proof}

Now we have all the estimates needed for the proof of Lemma  \ref{tildaXficedtime}.
\begin{proof}[Proof of Lemma \ref{tildaXficedtime}]
 We show that under the event $\mathcal{A}_T$ defined in \eqref{eventA}, for any $\alpha>0$ there exists a $C_\alpha$ depending only on $\alpha$, $T$ and $C_{f_0}$ such that at any fixed time $t\in[t_n,t_{n+1}]$

\begin{align}\label{goal}
	\PP\bigg(&\left|  \frac{1}{N-1}\sum_{j\neq 1}^N\left(k_1^N(\widetilde x_1^t-\widetilde x_j^t)-  k_1^N(\overline x_1^t-\overline x_j^t)\right)\right|\notag\\
	&\geq C_{\alpha}N^{2\delta-1}\log(N) +C_\alpha\log^2(N)N^{3\lambda_1-\lambda_2}\|k_1^N\|_1\bigg)\leq N^{-\alpha}.
	\end{align}

This is done under the event $\mathcal{A}_T$ in three steps:
\begin{itemize}
\item[(1)] We prove that for any $t\in [t_n,t_{n+1}]$ the number of particles inside $M_{t_n}^t$ is larger than 
\begin{equation}\label{defM}
M_\ast:= 2C_\ast  N\left(3N^{-\lambda_2}+ \log (N) \Delta t^{3/2}\right)^2 
\end{equation}
with probability less than $N^{-\alpha}$. Note that $M_\ast$ is used  as a bound in the definition of  \eqref{Stn} and \eqref{-Stn}. For any $t\in [t_n,t_{n+1}]$ we prove that
\begin{equation}\label{a}
\PP\left(\mbox{card }(M_{t_n}^t)> M_\ast\right)=\PP((\mathcal{S}_{t_n}^t)^c)\leq\PP((\overline{\mathcal{S}}_{t_n}^t)^c)\leq N^{-\alpha}.
\end{equation}

\item[(2)] We prove that at any fixed time $t\in [t_n,t_{n+1}]$, particles outside $M_{t_n}^t$ contribute to the interaction of $k_1^N$ with probability  less than $N^{-\alpha}$, namely
 \begin{align}\label{b}
 \PP\bigg(&\left|  \frac{1}{N-1}\sum_{j\in (M_{t_n}^t)^c}\left(k_1^N(\widetilde x_1^t-\widetilde x_j^t)-  k_1^N(\overline x_1^t-\overline x_j^t)\right)\right|> 0\bigg)\leq N^{-\alpha}.
 \end{align}
\item[(3)]  According to step $(2)$ above, at any fixed time $t\in [t_n,t_{n+1}]$, particles outside $M_{t_n}^t$ do not contribute to the interaction of $k_1^N$ with high probability, so we only consider particles that are inside $M_{t_n}^t$. And we know already from step $(1)$ above that the number of particles inside  $M_{t_n}^t$ is larger than $M_\ast$, with low probability. To prove \eqref{goal}, we  only need to  prove
\begin{align}\label{c}
\PP\bigg(	\mathcal{X}(M_{t_n}^t)\cap\left\{\mbox{ card }(M_{t_n}^t)\leq M_\ast\right\}\bigg)\leq N^{-\alpha}\,
	\end{align}
at any fixed time $t\in [t_n,t_{n+1}]$,	where the event $\mathcal{X}(M_{t_n}^t)$ is defined by
	\begin{align}\label{eventX}
	\mathcal{X}(M_{t_n}^t):=\bigg\{&\bigg|\frac{1}{N-1}\sum_{j\in M_{t_n}^t}\left(k_1^N(\widetilde x_1^t-\widetilde x_j^t)-  k_1^N(\overline x_1^t-\overline x_j^t)\right)\bigg|\notag \\
	&\geq C_{\alpha}N^{2\delta-1}\log(N) +C_\alpha\log^2(N)N^{3\lambda_1-\lambda_2} \|k_1^N\|_1\bigg\}.
	\end{align}
\end{itemize}

\begin{figure}[pb]\label{figure}
	\centering
	\includegraphics[width=0.99\textwidth]{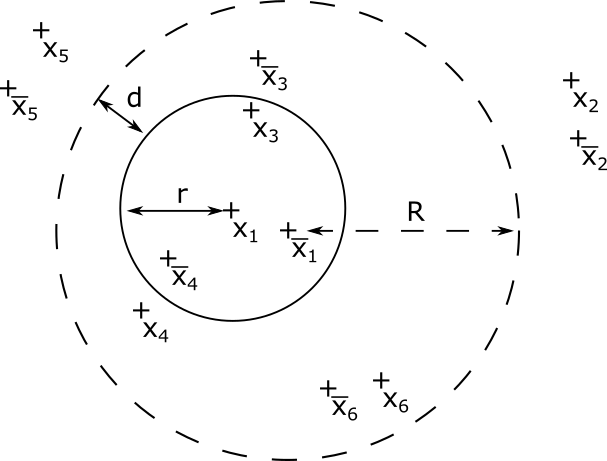}
	\caption{
		 Illustration of the sets $M_{t_n}^t$ and $\overline{M}_{t_n}$ under the assumption that $\A_T$ holds: the set  
		$M_{t_n}^t$ contains all indices of particles with respect to $X$ which  
		are in the ball of radius $r=N^{-\lambda_2}+ \log (N) (\Delta t)^{3/2}$  
		around $x_1$. In the figure this is the ball with solid lines and  
		$M_{t_n}^t=\{1,3\}$.
		The set $\overline M_{t_n}^t$ contains all indices of particles  with  
		respect to $\overline X$ which are in the ball of radius  
		$R=3N^{-\lambda_2}+ \log (N) (\Delta t)^{3/2}$ around $\overline x_1$. In  
		the figure this is the ball with  dashed lines and $\overline  
		M_{t_n}^t=\{1,3,4,6\}$.
		Since on the set $\mathcal{A}_T$  the distance $d$ of the particles  
		$x_1$ and $\overline x_1$ cannot be larger than $N^ {-\lambda_2}$, it  
		follows that, given that the event $\mathcal{A}_T$ holds, a particle $\overline x_j$ is in the solid ball only if  
		the particle $ x_j$ is in the ball with dashed lines, i.e. with radius  
		$R=3N^{-\lambda_2}+ \log (N) (\Delta t)^{3/2}$ around $x_1$ (see for  
		example particles $x_3$ and $\overline x_3$). Thus $M_{t_n}^t\subseteq  
		\overline{M}_{t_n}$.
		Controlling $M_{t_n}^t$ by $\overline{M}_{t_n}$ will be helpful to  
		estimate the number of particles inside these sets. The $\overline  
		x_j$ are distributed independently, and the probability of finding any of  
		these $\overline x_j$ inside the solid ball is small due to the small  
		volume of the ball. This helps to estimate the number of particles in  
		the set $\overline{M}_{t_n}$ (see \textit{Step} 1).
		Particles outside the ball, i.e. indices not in $\overline{M}_{t_n}$  
		do  not contribute to the interaction $k_1$. This comes  
		from the fact that in order to get a sufficiently small distance for 
		$x_1$ to interact, they have to travel a long distance during the  
		short time interval $(t-t_n)$: the distance $\log (N) (\Delta t)^{3/2}$  
		(recall that the support of $k_1$ has radius $N^{-\lambda_2}$). Due to  
		the Brownian motion, this is possible, of course, but the probability  
		to travel that far will be smaller than any polynomial in $N$. This  
		argument is worked out in \textit{Step} 2.
		The main contribution thus comes from \textit{Step} 3. Knowing that the number  
		of particles in $M_{t_n}$ is quite small helps to estimate this term. 
	}

\end{figure}

$\bullet \textit{Step 1:}$
To prove the first part of  \eqref{a}, note that on the event $\mathcal A_T$  defined in \eqref{eventA} and assuming that $t\in[t_n,t_{n+1}]$

$$\left| x_1^{t_n}- x_j^{t_n}+(t-t_n)(v_1^{t_n}- v_j^{t_n})\right|\leq N^{-\lambda_2}+ \log (N) \Delta t^{\frac{3}{2}}\, $$
implies
$$\left| \overline x_1^{t_n}- \overline x_j^{t_n}+(t-t_n)(\overline v_1^{t_n}- \overline v_j^{t_n})\right|\leq 3N^{-\lambda_2}+ \log (N) \Delta t^{\frac{3}{2}}\;.$$ 

Hence $M_{t_n}^t\subseteq \overline M_{t_n}^t$ and thus for any $R>0$, $\mbox{card }(\overline M_{t_n}^t)<R$ implies that
$\mbox{card } (M_{t_n }^t)\leq \mbox{card } (\overline M_{t_n }^t)<R$, consequently
 $\mathcal{S}_{t_n}^t\supseteq\overline{\mathcal{S}}_{t_n}^t$, i.e. $(\mathcal{S}^ t_{t_n})^c\subseteq(\overline{\mathcal{S}}^t_{t_n})^c$ .

The second part of \eqref{a} is trivial. For the third part we use the independence of the $\overline x$-particles. Note that the law of $(\overline x^j_{t_n}, \overline v^j_{t_n})$ has a density $f^N(x,v,t_n)$.
For any $j\in\left\{2,\ldots,N\right\}$ the probability to find $j\in \overline M_{t_n}^t$  for any $t\in[t_n,t_{n+1}]$ is given by
\begin{align}
\PP\left(j\in \overline M_{t_n}^t\right)&=\int_{\RR^3} \int_{B_{R}(\Xi^t)} f^N(x,v,t_n)dx dv,
\end{align}
where the center $\Xi^t$ of the ball is given by $\Xi^t=\overline x_1^{t_n}+(t-t_n)(\overline v_1^{t_n}- v)$, and the radius of the ball is given by $R=3N^{-\lambda_2}+ \log (N) \Delta t^{3/2} $.

	Define 
\begin{equation}\label{defgN}
g^N(x,v,s):=f^N(x-vs,v,t_n)
\end{equation}
which	 then satisfies the following transport equation
	\begin{equation}
	\begin{cases}
	\partial_sg^N(x,v,s)+v\cdot \nabla_x g^N(x,v,s)=0,~ 0<s\leq \Delta t,\\
	g^N(x,v,0)=f^N(x,v,t_n).
	\end{cases}
	\end{equation}

Then one has
\begin{equation}
\int_{B_{R}(\Xi^t)} f^N(x,v,t_n)dx=\int_{B_{R}( \Xi_0^t)} g^N(x,v,t-t_n)dx,
\end{equation}
where the center $\Xi_0^t$ of the ball is given by $\Xi_0^t=\overline x_1^{t_n}+(t-t_n)\overline v_1^{t_n}$, in particular the integration area is independent of $v$. It follows that the probability of finding $j\in\overline M_{t_n}^t$ for any $t\in[t_n,t_{n+1}]$ is equivalent to
	\begin{equation}\label{gfunc}
\PP\left(j\in  \overline M_{t_n}^t\right)=\int_{\RR^3} \int_{B_{R}( \Xi_0^t)} g^N(x,v,t-t_n)dx dv.
	\end{equation}
	
	Next, we compute for $0<s\leq \Delta t$
	\begin{align}
	\bar \rho^N (x,s)&:=\int_{\RR^3}g^N(x,v,s)dv=\int_{|v|\leq r(s)}g^N(x,v,s)dv+\int_{|v|> r(s)}g^N(x,v,s)dv \notag\\
	&\leq C_1\norm{g^N(\cdot,\cdot,s)}_\infty r(s)^3+\frac{1}{r(s)^6} \int_{|v|> r(s)}|v|^6g^N(x,v,s)dv\notag\\
	&= 2C_1^{\frac{2}{3}}\norm{g^N(\cdot,\cdot,s)}_\infty^{\frac{2}{3}} \left(\int_{|v|> r(s)}|v|^6g^N(x,v,s)dv\right)^{\frac{1}{3}},
	\end{align}
	where we have chosen
	\begin{equation}
	r(s)=\left(\frac{\int_{|v|> r(s)}|v|^6g^N(x,v,s)dv}{C_1\norm{g^N(\cdot,\cdot,s)}_\infty}\right)^{\frac{1}{9}}.
	\end{equation}
	 It follows that
	\begin{align}
	\int_{\RR^3}|\bar \rho ^N(x,s)|^3dx& \leq 8C_1^2\norm{g^N(\cdot,\cdot,s)}_\infty^2\iint_{\RR^6}|v|^6g^N(x,v,s)dxdv\notag\\
	&=8C_1^2\norm{f^N(x-vs,v,t_n)}_\infty^2\iint_{\RR^6}|v|^6f^N(x-vs,v,t_n)dxdv\notag\\
	&\leq C\left(\|f^N(\cdot,\cdot,t_n)\|_{L^{\infty}(\RR^6)},~
\||v|^{6}f^N(\cdot,\cdot,s)\|_{L^1(\RR^6)}\right),
	\end{align}
which leads to
	\begin{equation}
	\max\limits_{s\in[0,\Delta t]}\norm{\bar \rho^N(\cdot,s)}_{3}\leq C_2,
	\end{equation}
	because of \eqref{regularity}, where $C_2$ depends only on $T$, and $C_{f_0}$.  It follows from \eqref{gfunc} that
\begin{align}
\PP\left(j\in \overline M_{t_n}^t\right)&=\int_{B_{R}(\Xi_0^t)} \bar \rho^N(x,t-t_n)dx\leq  \norm{\bar \rho^N}_{3} |B_{R}(\Xi_0^t)| ^{\frac{2}{3}}\notag\\
&\leq C_2(\frac{4}{3}\pi) ^{\frac{2}{3}}\left(3N^{-\lambda_2}+ \log (N) \Delta t^{\frac{3}{2}}\right)^2\notag \\
&=C_\ast\left(3N^{-\lambda_2}+ \log (N) \Delta t^{3/2}\right)^2 =:p\;,
\end{align}
where we define $C_\ast:=C_2(\frac{4}{3}\pi) ^{\frac{2}{3}}$, which depends only on $T$ and $C_{f_0}$.

The probability of finding $k$ particles inside the set $\overline M_{t_n}^t$ is thus bounded from above by the binomial probability mass function with parameter $p$ at position $k$, i.e. for any natural number $0\leq A\leq N$ and any $t\in [t_n,t_{n+1}]$
	$$\PP\left(\mbox{card }(\overline M_{t_n}^t)\geq A\right)\leq\sum_{j=A}^ N \begin{pmatrix}
	N  \\
	j 
	\end{pmatrix}
	p^j (1-p)^{N-j}.$$
	Binomially distributed random variables have mean $Np$ and standard deviation $\sqrt{Np(1-p)}<\sqrt{Np}$, and the probability to find more than $Np+ a \sqrt{Np}$ particles in the set $\overline M_{t_n}^t$ is exponentially small in $a$, i.e. there is a sufficiently large $N$ for any $\alpha>0$ and any $t\in [t_n,t_{n+1}]$ such that
	$$\PP\left(\mbox{card }(\overline M_{t_n}^t)\geq Np+a\sqrt{Np}\right)\leq a^{-\alpha}\;.$$
	 This is because of  the central limit theory and so the binomial distribution can be seen as a normal distribution when $N$ is sufficiently large. 
	Since $p\geq C N^{-3\lambda_2}$, we get that $\sqrt{Np}>C N^ {\frac{1}{2}(1-3\lambda_2)}$ $(\lambda_2<1/3)$.  Hence the probability of finding more than $2Np=Np+\sqrt{Np}\sqrt{Np}$ (i.e. $a=\sqrt{Np}>C N^ {\frac{1}{2}(1-3\lambda_2)}$) particles is the set $\overline M_{t_n}^t$  is smaller than any polynomial in $N$, i.e. there is a $C_\alpha$ for any $\alpha>0$ and any $t\in [t_n,t_{n+1}]$ such that
	$$\PP((\overline{\mathcal{S}}_{t_n}^t)^c)=\PP\left(\mbox{card }(\overline M_{t_n}^t)\geq 2Np\right)\leq N^{-\alpha}.$$

$\bullet \textit{Step 2:}$
For \eqref{b} it is sufficient to show that for any $\alpha>0$ there is a sufficiently large $N$  such that for some $j\in (M_{t_n}^t)^c$
$$\PP\bigg(\max\limits_{t\in[t_n,t_{n+1}]}\left| \frac{1}{N-1} \left(k_1^N(\widetilde x_1^t-\widetilde x_j^t)-  k_1^N(\overline x_1^t-\overline x_j^t)\right)\right|> 0\bigg)\leq N^{-\alpha}.$$
The total probability we have to control in \eqref{b} is at maximum the $N$-fold value of this. The key to prove that is Lemma \ref{unlikely}. To have an interaction $k_1^N(\widetilde x_1^t-\widetilde x_j^t)\neq 0$ for all $t\in[t_n,t_{n+1}]$ the distance between particle $1$ and particle $j$ has to be reduced to a value smaller than $N^ {-\lambda_2}$. Due to the Brownian motion, this is possible, but suppressed. Due to the fast decay of the Gaussian it is very unlikely that $k_1^N(\widetilde x_1^t-\widetilde x_j^t)\neq 0$. The probability is smaller than any polynomial in $N$ (see Lemma \ref{unlikely}).The same holds true for $k_1^N(\overline x_1^t-\overline x_j^t)$.

In more detail: due to the cut-off $N^{-\lambda_2}$ we introduced for $k_1^N$ 

\begin{align*}
&\PP\bigg( \max\limits_{t\in[t_n,t_{n+1}]}\left|  \frac{1}{N-1}\sum_{j\in (M_{t_n}^t)^c}\left(k_1^N(\widetilde x_1^t-\widetilde x_j^t)-  k_1^N(\overline x_1^t-\overline x_j^t)\right)\right|> 0\bigg)
\\\leq&
\PP\bigg(\max\limits_{t\in[t_n,t_{n+1}]}\left|  \frac{1}{N-1}\sum_{j\in (M_{t_n}^t)^c}k_1^N(\widetilde x_1^t-\widetilde x_j^t) \right|> 0\bigg)
\notag\\
&+
\PP\bigg(\max\limits_{t\in[t_n,t_{n+1}]}\left|  \frac{1}{N-1}\sum_{j\in (M_{t_n}^t)^c}   k_1^N(\overline x_1^t-\overline x_j^t)\right|> 0\bigg)
\\\leq& N\PP\left(\min_{t\in[t_n,t_{n+1}]}\max_{j\in( M^ t_{t_n})^c}\left\{\left|\widetilde x_{1}^t-\widetilde x_{j}^t\right|\right\}<N^{-\lambda_2}\;\;\right)\\
&+N\PP\left(\min_{t\in[t_n,t_{n+1}]}\max_{j\in(\overline M^ t_{t_n})^c}\left\{\left|\widetilde x_{1}^t-\widetilde x_{j}^t\right|\right\}<N^{-\lambda_2}\;\;\right),
\end{align*}
where we used the fact that $(\overline M^ t_{t_n})^c\subseteq ( M^ t_{t_n})^c$ in the last inequality.
With Lemma \ref{unlikely} we get the bound for \eqref{b}.

$\bullet \textit{Step 3:}$
To get \eqref{c} we prove that  for any natural number
\begin{align*}
0\leq M&\leq M_\ast=2C_\ast  N\left(3N^{-\lambda_2}+ \log (N) \Delta t^{3/2}\right)^2 
\end{align*}
one has
\begin{align}\PP\big(\mathcal{X}(M_{t_n}^t)\cap\{\mbox{ card }(M_{t_n}^t)=M\}\big)\leq N^{-\alpha},
\end{align}	
where the event $\mathcal{X}(M_{t_n}^t)$ is defined in \eqref{eventX}.
This can be recast without relabeling $j$ as
\begin{align}\label{toshow}
\PP\bigg(&\left|\frac{1}{N-1}\sum_{j=1}^{M}\left(k_1^N(\widetilde x_1^t-\widetilde x_j^t)-  k_1^N(\overline x_1^t-\overline x_j^t)\right)\right|\notag \\
&\geq C_{\alpha}N^{2\delta-1}\log(N)
+C_\alpha\log^2(N)N^{3\lambda_1-\lambda_2} \|k_1^N\|_1\bigg)\leq N^{-\alpha}.
\end{align}

\begin{lem}\label{central2} 
	Let $Z_1,\cdots,Z_M$ be independent random variables with $\mathbb{E}[|Z_i|]\leq CM^{-2}$
	and $|Z_i|\leq C$ for any $i\in\{1,\cdots,M\}$. Then for any $\alpha>0$, it holds that
	\begin{equation}
	\PP\left(\sum_{i=1}^{M}|Z_i|\geq C_\alpha \ln(M)\right)\leq M^{-\alpha},
	\end{equation}
	where $C_\alpha$ depends only on $C$ and $\alpha$.
\end{lem}

\begin{proof}
	We first split the random variables $Z_i=Z_i^a+Z_i^b$ such that
	 $Z_i^a$ and $Z_i^b$ are sequences of independent random variables with
	$$\mathbb{P}(|Z_i^a|>0)= M^{-1}\: \text{and } |Z_i^b|\leq CM^{-1}.$$
	This can be achieved by defining 
	$$Z_i^a(\omega)=\begin{cases} Z_i(\omega) &\mbox{if } Z_i(\omega)>\gamma,\\ 
	0 & \mbox{else }. \end{cases}$$
	and $Z_i^b=Z_i-Z_i^a$. Here we choose $\gamma$ such that $\mathbb{P}(|Z_i^a|>0)= M^{-1}$. Applying Markov's inequality, one computes
   \begin{align}
   M^{-1}=\mathbb{P}(|Z_i^a|>0)=\PP (Z_i>\gamma)\leq\PP\left(|Z_i|>\frac{\gamma}{\mathbb{E}[|Z_i|]}\mathbb{E}[|Z_i|]\right)\leq \frac{\mathbb{E}[|Z_i|]}{\gamma}\leq C\frac{M^{-2}}{\gamma}.
   \end{align}
This implies that $\gamma\leq CM^{-1}$.

	For the sum of $Z_i^b$ we get the trivial bound
	$$\sum_{j=1}^M |Z_i^b|\leq CM^{-1} M=C\;.$$
	Thus the lemma follows if we can show that 
	\begin{equation}\label{Za}
	\PP\left(\sum_{i=1}^M|Z_i^a|\geq C_\alpha \ln(M)\right)\leq N^{-\alpha},
	\end{equation}
	where $C_\alpha$ has been changed.
	
	Let 	$$X_i(\omega)=\begin{cases} 0 &\mbox{if } Z_i^a(\omega)=0, \\ 
	1 & \mbox{else }. \end{cases}$$
	Since $|Z_i|\leq C$, one has
	\begin{equation}
	\sum_{i=1}^M|Z_i^a|=	\sum_{i=1}^MX_i|Z_i|\leq C\sum_{i=1}^MX_i.
	\end{equation}
	Then it follows that  $$\sum_{i=1}^M|Z_i^a|\geq C_\alpha \ln (M)\Longrightarrow\sum_{j=1}^M X_i\geq \frac{C_\alpha}{C}\ln (M).$$
	Noticing that  $X_i$ are i.i.d.  Bernoulli random variables with $\PP(X_i=1)=\mathbb{P}(|Z_i^a|>0)= M^{-1}$, we get
	\begin{align*}
		&\PP\left(\sum_{i=1}^M|Z_i^a|\geq C_\alpha \ln (M)\right)\leq \PP\left(\sum_{j=1}^M X_i\geq \frac{C_\alpha}{C}\ln (M)\right)
		\\
		= &\sum_{j=a}^M \frac{M!}{j!(M-j)!} M^{-j}(1-M^{-1})^{M-j}
		\leq \sum_{j=a}^M \frac{M^j}{j!} M^{-j}
		\leq \frac{2}{a!}\,,
	\end{align*}
	where $a=\frac{C_\alpha}{C}\ln(M)$.  Notice the decay property of the factorial  
		\begin{align}
	(\ln(M))!\geq \left(\frac{\ln(M)}{2}\right)^{\frac{\ln(M)}{2}}=\exp\left(\frac{\ln(M)}{2}\ln\left(\frac{\ln(M)}{2}\right)\right)=M^{\frac{1}{2}\ln\left(\frac{\ln(M)}{2}\right)}.
	\end{align}
	Thus one chooses $M$ large enough and concludes \eqref{Za}, which proves the lemma.

\end{proof}

Using the lemma above, now we proceed to prove \eqref{toshow}. Define $$Z_j:=N^{-2\delta }\left(k_1^N(\widetilde x_1^t-\widetilde x_j^t)-  k_1^N(\overline x_1^t-\overline x_j^t)\right).$$ 
 It follows that $|Z_j|$ is bounded and 
\begin{align*}
\mathbb{E}(|Z_j|)\leq CN^{-2\delta }\norm{k_1^N}_1\begin{cases}
\|u_{t,t_{n-1}}^{a,N}\|_{\infty,1},~&\mbox{ for }1\leq n\leq M'\,\\
\|f_t^N\|_{\infty,1},~&\mbox{ for } n=1\,
\end{cases}\leq C N^{\frac{9}{2} \lambda_1-2\delta-\lambda_2}\,,
\end{align*}
where we use the fact $\|u_{t,t_{n-1}}^{a,N}\|_{\infty,1}\leq CN^{\frac{9}{2}\lambda_1}$ $(\Delta t\leq t-t_{n-1}\leq 2\Delta t)$ from $(i)$ in Lemma \ref{transition},  $\|f_t^N\|_{\infty,1}\leq C_{f_0}$ and $\norm{k_1^N}_1\leq N^{-\lambda_2}$ from \eqref{k1est}. 
Using Lemma \ref{central2}  with $M=N^{\delta+\frac{\lambda_2}{2}-\frac{9}{4}\lambda_1}$ one obtains
\begin{align}
\PP\left(\frac{1}{N-1}\sum_{j=1}^{N^{\delta+\frac{\lambda_2}{2}-\frac{9}{4}\lambda_1}}\left|\left(k_1^N(\widetilde x_1^t-\widetilde x_j^t)-  k_1^N(\overline x_1^t-\overline x_j^t)\right)\right|\geq C_\alpha N^{2\delta-1}\ln(N)\right)\leq N^{-\alpha},
\end{align}
which leads to \eqref{toshow} for $M=N^{\delta+\frac{\lambda_2}{2}-\frac{9}{4}\lambda_1}$.  It is obvious that
\begin{equation}
\sum_{i=1}^{M}|Z_i|\leq \sum_{i=1}^{N^{\delta+\frac{\lambda_2}{2}-\frac{9}{4}\lambda_1}}|Z_i|,
\end{equation}
for any $M\leq N^{\delta+\frac{\lambda_2}{2}-\frac{9}{4}\lambda_1}$.
Thus one concludes
\eqref{toshow} holds for the case  $M\leq N^{\delta+\frac{\lambda_2}{2}-\frac{9}{4}\lambda_1}$.

For the remaining $M$ we note that \begin{equation}
2C_\ast  N\left(3N^{-\lambda_2}+ \log (N)N^{-\frac{3}{2}\lambda_1}\right)^2\leq 4C_\ast  N \log^2 (N)N^{-3\lambda_1},
\end{equation}
due to the fact that $0<\lambda_1<\frac{2}{3}\lambda_2$. Thus we are left to prove \eqref{toshow}  
 for the case
  \begin{equation}\label{remainingM}N^{\delta+\frac{\lambda_2}{2}-\frac{9}{4}\lambda_1}<M\leq 4C_\ast  N \log^2 (N)N^{-3\lambda_1}\;.\end{equation}
This can be done by Lemma \ref{central}, which we repeat below for easier reference:
\paragraph{Lemma \ref{central}}
	Let $Z_1,\cdots,Z_{M}$ be $i.i.d.$ random variables with $\mathbb{E}[Z_i]=0,$ $\mathbb{E}[Z_i^2]\leq g(M)$
	and $|Z_i|\leq C\sqrt{M g(M)}$. Then for any $\alpha>0$, the sample mean $\bar{Z}=\frac{1}{M}\sum_{i=1}^{M}Z_i$ satisfies
	\begin{equation}
	\PP\left(|\bar{Z}|\geq\frac{C_\alpha \sqrt{g(M)}\log(M)}{\sqrt{M}}\right)\leq M^{-\alpha},
	\end{equation}
	where $C_\alpha$ depends only on $C$ and $\alpha$.

For any fixed $t\in[t_n,t_{n+1}]$ we choose $Z_j^t:=\frac{M}{N-1} k_1^N(\widetilde x_1^t-\widetilde x_j^t)-\frac{M}{N-1} \EE[k_1^N(\widetilde x_1^t-\widetilde x_j^t)]$ and $g(M):=CM N^{4\delta-2}$, where $N^{\delta+\frac{\lambda_2}{2}-\frac{9}{4}\lambda_1}<M\leq 4C_\ast  N \log^2 (N)N^{-3\lambda_1}$. Then following the same argument as in \eqref{Esqure}, the condition
 \begin{align*}
\mathbb{E}[(Z_j^t)^2]\leq&C\frac{M^2}{(N-1)^2} N^{\delta}\begin{cases}
\|u_{t,t_{n-1}}^{a,N}\|_{\infty,1},~&\mbox{ for }1\leq n\leq M'\,\\
\|f_t^N\|_{\infty,1},~&\mbox{ for } n=1\,
\end{cases}\leq C M^2 N^{\delta-2}N^{\frac{9}{2}\lambda_1}\leq g(M),
\end{align*}
is satisfied. We can also deduce that
\begin{align*}
|Z_j^t|\leq C\frac{M}{N-1}N^ {2\delta}\leq\sqrt{M (CM N^{4\delta-2})}=\sqrt{M g(M)}.
\end{align*}
Applying Lemma \ref{central} we obtain at any fixed time $t\in[t_n,t_{n+1}]$
\begin{equation}\label{part1}
\PP\left(\left|\frac{1}{N-1}\sum_{j=1}^{M}\left( k_1^N(\widetilde x_1^t-\widetilde x_j^t)-\mathbb{E}[k_1^N(\widetilde x_1^t-\widetilde x_j^t)]\right)\right|\geq C_\alpha N^ {2\delta-1}\log (N)\right)\leq  N^{-\alpha},
\end{equation}
and similarly
\begin{equation}\label{part2}
\PP\left(\left|\frac{1}{N-1}\sum_{j=1}^{M}\left( k_1^N(\overline x_1^t-\overline x_j^t)-\mathbb{E}[k_1^N(\overline x_1^t-\overline x_j^t)]\right)\right|\geq C_\alpha N^ {2\delta-1}\log (N)\right)\leq N^{-\alpha}\;.
\end{equation}
 

It is left to control the difference 
$$\left|\frac{1}{N-1}\sum_{j=1}^{M}\left(\EE[k_1^N(\widetilde x_1^t-\widetilde x_j^t)]-\EE[ k_1^N(\overline x_1^t-\overline x_j^t)]\right)\right|,$$ 
where $M$ satisfies \eqref{remainingM}. This can be done by using Lemma 
	\ref{transition}.
	For  any $t\in[t_n,t_{n+1}]$, when $1\leq n\leq M'$ we write $a=(\widetilde X_{t_{n-1}},\widetilde V_{t_{n-1}})=(X_{t_{n-1}},V_{t_{n-1}})$ and $b=(\overline X_{t_{n-1}},\overline V_{t_{n-1}})$. Then it follows that
	\begin{align}\label{transresult}
	&\left|\frac{1}{N-1}\sum_{j=1}^{M}\left(\EE[k_1^N(\widetilde x_1^t-\widetilde x_j^t)]-\EE[ k_1^N(\overline x_1^t-\overline x_j^t)]\right)\right|\notag\\
	=&\frac{1}{N-1}\bigg|\sum\limits_{j=1}^{M}\int k_1^N(x_1-x_j)\big(u^{a,1,N}_{t,t_{n-1}}(x_1,v_1)u^{a,j,N}_{t,t_{n-1}}(x_j,v_j)\notag \\
	& \qquad\qquad-u^{b,1,N}_{t,t_{n-1}}(x_1,v_1)u^{b,j,N}_{t,t_{n-1}}(x_j,v_j)\big)dx_1dv_1dx_jdv_j\bigg| \notag\\
	\leq& \frac{1}{N-1}\sum\limits_{j=1}^{M}\left|\int k_1^N(x_1-x_j)u^{a,1,N}_{t,t_{n-1}}(x_1,v_1)\left(u^{a,j,N}_{t,t_{n-1}}(x_j,v_j)-u^{b,j,N}_{t,t_{n-1}}(x_j,v_j)\right)dx_1dv_1dx_jdv_j\right|\notag \\
	&+\frac{1}{N-1}\sum\limits_{j=1}^{M}\left|\int k_1^N(x_1-x_j)u^{b,j,N}_{t,t_{n-1}}(x_1,v_1)\left(u^{a,1,N}_{t,t_{n-1}}(x_j,v_j)-u^{b,1,N}_{t,t_{n-1}}(x_j,v_j)\right)dx_1dv_1dx_jdv_j\right|\notag \\
	\leq&  \frac{1}{N-1}\sum\limits_{j=1}^{M} \left(\norm{u^{a,j,N}_{t,t_{n-1}}-u^{b,j,N}_{t,t_{n-1}}}_{\infty,1}\norm{k_1^N\ast\rho_{t,t_{n-1}}^{a,1,N}}_1+\norm{u^{a,1,N}_{t,t_{n-1}}-u^{b,1,N}_{t,t_{n-1}}}_{\infty,1}\norm{k_1^N\ast\rho_{t,t_{n-1}}^{b,j,N}}_1\right)\notag \\
	\leq& \frac{1}{N-1}\sum\limits_{j=1}^{M} C(t-s)^{-6}|a_i-b_i|(\norm{k_1^N}_1\norm{\rho_{t,t_{n-1}}^{a,1,N}}_1+\norm{k_1^N}_1\norm{\rho_{t,t_{n-1}}^{b,j,N}}_1)\notag \\
	\leq& C\log^2(N)N^{3\lambda_1}|a_i-b_i|\norm{k_1^N}_1\leq C\log^2(N)N^{3\lambda_1-\lambda_2}\norm{k_1^N}_1,
	\end{align}
	where $\rho_{t,t_{n-1}}^{a,1,N}(x_1)=\int_{\RR^3} u_{t,t_{n-1}}^{a,1,N}(x_1,v_1)dv_1$. Here we have  used the fact that when $1\leq n\leq M'$
	$$\norm{u^{a,j,N}_{t,t_{n-1}}-u^{b,j,N}_{t,t_{n-1}}}_{\infty,1}\leq C|a_i-b_i|((t-t_{n-1})^{-6}+1)\leq C|a_i-b_i|N^{6\lambda_1}$$
	by Lemma \ref{transition} since $N^{-\lambda_1}\leq t-t_{n-1}\leq 2N^{-\lambda_1}$.  When $n=1$, since $a=(\widetilde X_{0},\widetilde V_{0})=(X_{0},V_{0})=(\overline X_{0},\overline V_{0})=b$, one has
	$$\left|\frac{1}{N-1}\sum_{j=1}^{M}\left(\EE[k_1^N(\widetilde x_1^t-\widetilde x_j^t)]-\EE[ k_1^N(\overline x_1^t-\overline x_j^t)]\right)\right|=0\,.$$
	Collecting \eqref{part1}, \eqref{part2} and \eqref{transresult} we get \eqref{toshow} for $M$ satisfying \eqref{remainingM}, which finishes the proof of \eqref{toshow} for any $M$. Hence we conclude \eqref{c}.

$\bullet \textit{Step 4:}$ Now we prove \eqref{goal}. To see this, we split the summation $\sum_{j\neq 1}^N$ into two parts: the part where $j\in M _{t_n}^t$ and the part where $j\in (M _{t_n}^t )^c$
\begin{align}
\PP\bigg(&\left|  \frac{1}{N-1}\sum_{j\neq 1}^N\left(k_1^N(\widetilde x_1^t-\widetilde x_j^t)-  k_1^N(\overline x_1^t-\overline x_j^t)\right)\right|\geq 2C_{\alpha}N^{2\delta-1}\log(N) \notag \\
&+2C_\alpha\log^2(N)N^{3\lambda_1-\lambda_2} \|k_1^N\|_1\bigg) \leq\PP\left(\mathcal{X}(M_{t_n}^t)\right)+\PP\left(\mathcal{X}((M_{t_n}^t)^c)\right)\,,
\end{align}
where $\mathcal{X}(M_{t_n}^t)$ is defined in \eqref{eventX} and
	\begin{align}\label{eventX1}
\mathcal{X}((M_{t_n}^t)^c):=\bigg\{&\big|\frac{1}{N-1}\sum_{j\in (M_{t_n}^t)^c}\left(k_1^N(\widetilde x_1^t-\widetilde x_j^t)-  k_1^N(\overline x_1^t-\overline x_j^t)\right)\big|\notag \\
&\geq C_{\alpha}N^{2\delta-1}\log(N) +C_\alpha\log^2(N)N^{3\lambda_1-\lambda_2} \|k_1^N\|_1\bigg\}.
\end{align}
For the part in the event $\mathcal{X}((M_{t_n}^t)^c)$ where $j\in (M _{t_n}^t )^c$, it follows from \eqref{b}  that
\begin{align}
 \PP\left(\mathcal{X}((M_{t_n}^t)^c)\right)\leq N^{-\alpha}.
\end{align}
Thus we have
\begin{align}\label{186}
\PP\bigg(&\left|  \frac{1}{N-1}\sum_{j\neq 1}^N\left(k_1^N(\widetilde x_1^t-\widetilde x_j^t)-  k_1^N(\overline x_1^t-\overline x_j^t)\right)\right|\geq 2C_{\alpha}N^{2\delta-1}\log(N) \notag \\
&+2C_\alpha\log^2(N)N^{3\lambda_1-\lambda_2} \|k_1^N\|_1\bigg)\leq \PP\left(\mathcal{X}(M_{t_n}^t)\right) +N^{-\alpha}.
\end{align}

Next we split the summation $\sum_{j\in M _{t_n}^t }$ in the event $\mathcal{X}(M_{t_n}^t)$ \eqref{eventX} into two cases: the case where $\mbox{ card }(M_{t_n}^t)\leq M_\ast$ and the case where $\mbox{ card }(M_{t_n}^t)> M_\ast$. Here $M_\ast$ is defined in \eqref{defM}.
\begin{align}\label{187}
\PP\left(\mathcal{X}(M_{t_n}^t)\right)& \leq \PP\big(\mathcal{X}(M_{t_n}^t)\cap \left\{\mbox{ card }(M_{t_n}^t)\leq M_\ast\right\}\big) +\PP\big(\mathcal{X}(M_{t_n}^t)\cap \left\{\mbox{ card }(M_{t_n}^t)> M_\ast\right\}\big) \notag \\
&\leq\PP\big(\mathcal{X}(M_{t_n}^t)\cap \left\{\mbox{ card }(M_{t_n}^t)> M_\ast\right\}\big)+N^{-\alpha},
\end{align} 
where in the last inequality we used \eqref{c}. 

According to \eqref{a}, for any $t\in[t_n,t_{n+1}]$ one has  
\begin{equation}
\PP\left(\mbox{card }(M_{t_n}^t)> M_\ast\right)\leq N^{-\alpha},
\end{equation}
which leads to
\begin{align}
\PP\big(\mathcal{X}(M_{t_n}^t)\cap \left\{\mbox{ card }(M_{t_n}^t)> M_\ast\right\}\big)\leq \PP\left(\mbox{card }(M_{t_n}^t)> M_\ast\right)\leq N^{-\alpha}.
\end{align}
Therefore it follows from \eqref{187} that
\begin{align}
\PP\left(\mathcal{X}(M_{t_n}^t)\right)\leq 2N^{-\alpha}.
\end{align}
Together with \eqref{186}, it implies
\begin{align}\label{goal1}
\PP\bigg(&\left|  \frac{1}{N-1}\sum_{j\neq 1}^N\left(k_1^N(\widetilde x_1^t-\widetilde x_j^t)-  k_1^N(\overline x_1^t-\overline x_j^t)\right)\right|\geq 2C_{\alpha}N^{2\delta-1}\log(N) \notag \\
&+2C_\alpha\log^2(N)N^{3\lambda_1-\lambda_2} \|k_1^N\|_1\bigg)\leq 3N^{-\alpha}.
\end{align}

 Finally, since the particles are exchangeable, the same result holds for changing $(\widetilde x_1^t,\overline x_1^t)$ in \eqref{goal1} into $(\widetilde x_i^t,\overline x_i^t)$, $i=2,\cdots, N$, which completes the proof of  Lemma \ref{tildaXficedtime}.
\end{proof}

{\bf Acknowledgments:}
H.H. is partially supported by NSFC (Grant No. 11771237). The research of J.-G. L. is partially supported by  KI-Net NSF RNMS (Grant No. 1107444) and NSF DMS
(Grant No. 1812573). 

\section*{Appendix}
\appendix
\renewcommand{\appendixname}{Appendix~\Alph{section}}
\section{Proof of Lemma \ref{transition}}

First, let us consider the fundamental solution $G(x,v, t)$ of the equation
\begin{equation}
\partial_t G  + v\cdot \nabla_xG=\Delta_v G,\;\;G\mid_{t=0}=\delta(x)\delta(v),
\end{equation}
which  can be calculated explicitly as
\begin{equation}\label{gexpl}
G(x,v, t)=C \frac1{t^6}\exp\left(-\frac{|v|^2}{4t}-\frac{3|x- tv/2|^2}{t^3}\right),
\end{equation}
where $C$ is a normalization constant. The following lemma states some estimates of the fundamental solution.
\begin{lem} Let  $G(x,v, t)$ be defined in \eqref{gexpl}
	and $p\in[1,\infty]$. There exists a $C_p$ such that for any $j\in\mathbb{N}_0$ the following holds
	\begin{align}\label{heat1}
	\left\||x|^ j\nabla_v G \right\|_{p,1}\leq C_p t^{\frac{-10p+3jp+9}{2p}}\;,\quad \left\||x|^ j\nabla_x G \right\|_{p,1}\leq C_p t^{\frac{-12p+3jp+9}{2p}}\; ,
	\end{align}
	and 
	\begin{align}\label{heat3}
	\left\|G\left(\cdot-\frac{1}{2}\left(a_i-b_i\right)\right)-G\left(\cdot-\frac{1}{2}\left(b_i-a_i\right)\right)\right\|_{p,1}\leq C_p|a_i-b_i|\left(t^{\frac{-12p+9}{2p}}+t^{\frac{-10p+9}{2p}}\right)\;,
	\end{align}
	as well as
	\begin{align}\label{heat2}
	\left\||\cdot|\left(G\left(\cdot-\frac{1}{2}\left(a_i-b_i\right)\right)-G\left(\cdot-\frac{1}{2}\left(b_i-a_i\right)\right)\right)\right\|_{p,1}\leq C_p|a_i-b_i|\left(t^{\frac{-7p+9}{2p}}+t^{\frac{-9p+9}{2p}}\right)\;.
	\end{align}
	The norm $\|\cdot\|_{p,q}$ denotes the $p$-norm in the $x$ and $q$-norm in the $v$-variable, i.e. for any $f:\mathbb{R}^3\times\mathbb{R}^3\to\mathbb{R} $
	\begin{equation}
	\|f\|_{p,q}:=\left(\int_{\RR^3} \left(\int_{\RR^3} |f(x,v)|^q dv\right)^{p/q}dx\right)^{1/p}.
	\end{equation}	
\end{lem}

\begin{proof}
	It is easy to compute that
	\begin{equation}
	G=C \frac1{t^6} \exp\left(-\frac{3|x|^2}{4t^3}\right) \exp\left(-\frac{|v - \frac{3x}{2t}|^2}{t}\right)
	\end{equation}
	and
	\begin{equation}
	\nabla_vG=C\frac1{t^6} \exp\left(-\frac{3|x|^2}{4t^3}\right) \exp\left(-\frac{|v - \frac{3x}{2t}|^2}{t}\right)\left(-\frac{2v}{t}+\frac{3x}{t^2}\right).
	\end{equation}
	
	Now we can do the calculation of $\int_{\RR^3}|G|dv$ and $\int_{\RR^3}|\nabla_vG|dv$:
	\begin{align}\label{G}
	\int_{\RR^3}|G|dv&=C\frac1{t^6} \exp\left(-\frac{3|x|^2}{4t^3}\right) \int_{\RR^3}\exp\left(-\frac{|v - \frac{3x}{2t}|^2}{t}\right)dv\notag\\
	&\leq C\frac1{t^{9/2}} \exp\left(-\frac{3|x|^2}{4t^3}\right),
	\end{align}	
and
	\begin{align}\label{vG}
	\int_{\RR^3}|\nabla_vG|dv&=C\frac1{t^6} \exp\left(-\frac{3|x|^2}{4t^3}\right) \int_{\RR^3}\exp\left(-\frac{|v - \frac{3x}{2t}|^2}{t}\right)\left(-\frac{2v}{t}+\frac{3x}{t^2}\right)dv\notag\\
	&\leq C\frac1{t^5} \exp\left(-\frac{3|x|^2}{4t^3}\right)  \int_{\RR^3}u\exp\left(-u^2\right)du\leq C\frac1{t^5} \exp\left(-\frac{3|x|^2}{4t^3}\right) ,
	\end{align}
		respectively.
	As a direct result from \eqref{G} and  \eqref{vG}, one has
	\begin{equation}
	\left\||\cdot|^jG \right\|_{\infty,1}\leq C\frac1{t^{(9-3j)/2}}\hspace{1cm}\left\||\cdot|^ j\nabla_v G \right\|_{\infty,1}\leq C\frac1{t^{5-3j/2}} .
	\end{equation}
	For $1\leq p<\infty$
	\begin{align}
	\left\||\cdot|^jG \right\|_{p,1}&\leq Ct^{-{9/2}}\left(\int_{\RR^3}|x|^{pj}\exp\left(-\frac{3px^2}{4t^3}\right)dx\right)^{\frac{1}{p}}\notag \\
	&\leq 	C_pt^{-\frac{9}{2}+\frac{9+3pj}{2p}}\left(\int_{\RR^3}|y|^ j\exp\left(-y^2\right)dy\right)^{\frac{1}{p}}\leq C_p t^{\frac{-9p+3jp+9}{2p}},\label{Gnorm}
	\end{align}
	and
	\begin{align}
	\left\||\cdot|^j\nabla_v G \right\|_{p,1}&\leq Ct^{-5}\left(\int_{\RR^3}|x|^ j\exp\left(-\frac{3px^2}{4t^3}\right)dx\right)^{\frac{1}{p}}\notag \\
	&\leq 	C_pt^{-5+\frac{9+3jp}{2p}}\left(\int_{\RR^3}|y|^ j\exp\left(-y^2\right)dy\right)^{\frac{1}{p}}\leq C_p t^{\frac{-10p+3jp+9}{2p}}.
	\end{align}
	We also have
	\begin{align}
	\nabla_xG= \frac1{t^6} \exp\left(-\frac{3|x|^2}{4t^3}\right) \exp\left(-\frac{|v - \frac{3x}{2t}|^2}{t}\right)
	\left(-\frac{6x}{t^2}+\frac{3v}{t}\right),
	\end{align}
	which leads to
	\begin{align}\label{xG}
	\int_{\RR^3}|\nabla_xG|dv&=C\frac1{t^6} \exp\left(-\frac{3|x|^2}{4t^3}\right) \int_{\RR^3}\exp\left(-\frac{|v - \frac{3x}{2t}|^2}{t}\right)\left(-\frac{6x}{t^3}+\frac{3v}{t^2}\right)dv\notag\\
	&\leq C\frac1{t^8} \exp\left(-\frac{3|x|^2}{4t^3}\right) \int_{\RR^3}\exp\left(-\frac{|v - \frac{3x}{2t}|^2}{t}\right)\left(-\frac{2x}{t}+v\right)dv\notag \\
	&\leq Ct^{-8+\frac{3}{2}} \exp\left(-\frac{3|x|^2}{4t^3}\right)  \int_{\RR^3}(\sqrt{t}u-\frac{x}{2t})\exp\left(-u^2\right)du \notag\\
	&\leq C\frac1{t^6} \exp\left(-\frac{3|x|^2}{4t^3}\right)+C\frac1{t^6} \frac{x}{t^{\frac{3}{2}}}\exp\left(-\frac{3|x|^2}{4t^3}\right).
	\end{align}
	It follows from the above that
	\begin{equation}
	\left\||\cdot|^j\nabla_x G \right\|_{\infty,1}\leq C\frac1{t^{6-3j/2}} .
	\end{equation}
	For $1\leq p<\infty$
	\begin{align}
	&\norm{|\cdot|^j\nabla_x G}_{p,1}\notag\\
	\leq& C\frac1{t^6}\left(\left(\int_{\RR^3}|x|^j\exp\left(-\frac{3px^2}{4t^3}\right)dx\right)^{\frac{1}{p}} +\left(\int_{\RR^3}|x|^ j\exp\left(-\frac{3p|x|^2}{4t^3}\right) \left (\frac{x}{t^{\frac{3}{2}}}\right)^pdx\right)^{\frac{1}{p}}\right)\notag \\
	\leq &C_pt^{\frac{-12p+3jp+9}{2p}}\left(\left(\int_{\RR^3}|y|^ j\exp\left(-y^2\right)dy\right)^{\frac{1}{p}}+\left(\int_{\RR^3}|y|^j\exp\left(-py^2\right)|y|^pdy\right)^{\frac{1}{p}}\right)\notag\\
	\leq &C_pt^{\frac{-12p+3jp+9}{2p}},
	\end{align}
	which concludes the proof of \eqref{heat1}.
	
	
	As a direct result of \eqref{heat1} we can prove \eqref{heat3}. Indeed,
	\begin{align}
	&\left|G\left(\cdot-\frac{1}{2}\left(a_i-b_i\right)\right)-G\left(\cdot-\frac{1}{2}\left(b_i-a_i\right)\right)\right|\notag\\
	\leq& |a_i-b_i|\int_0^1\left|\nabla G\left(\cdot-\frac{1}{2}(b_i-a_i)+s(b_i-a_i)\right)\right|ds \notag\\
	\leq & |a_i-b_i|\int_0^1\left|\nabla_v G\left(\cdot-\frac{1}{2}(b_i-a_i)+s(b_i-a_i)\right)\right|ds \notag\\
	&+|a_i-b_i|\int_0^1\left|\nabla_x G\left(\cdot-\frac{1}{2}(b_i-a_i)+s(b_i-a_i)\right)\right|ds,
	\end{align}
	which leads to
	\begin{align}
	&\left\|\left(G\left(\cdot-\frac{1}{2}\left(a_i-b_i\right)\right)-G\left(\cdot-\frac{1}{2}\left(b_i-a_i\right)\right)\right)\right\|_{p,1}\notag \\
	\leq &C|a_i-b_i|\left(\norm{\nabla_v G}_{p,1}+\norm{\nabla_x G}_{p,1}\right)
	\leq C_p|a_i-b_i|\left(t^{\frac{-12p+9}{2p}}+t^{\frac{-10p+9}{2p}}\right).
	\end{align}
	
	Next we prove \eqref{heat2}:
	\begin{align}
	&|\cdot| \left(G\left(\cdot-\frac{1}{2}\left(a_i-b_i\right)\right)-G\left(\cdot-\frac{1}{2}\left(b_i-a_i\right)\right)\right)\notag\\
	\leq& \left(|\cdot-\frac{1}{2}\left(a_i-b_i\right)| G\left(\cdot-\frac{1}{2}\left(a_i-b_i\right)\right)-|\cdot-\frac{1}{2}\left(b_i-a_i\right)| G\left(\cdot-\frac{1}{2}\left(b_i-a_i\right)\right)\right)
	\notag\\&+\frac{1}{2}|a_i-b_i|\left(G\left(\cdot-\frac{1}{2}\left(a_i-b_i\right)\right)+G\left(\cdot-\frac{1}{2}\left(b_i-a_i\right)\right)\right).
	\end{align}
	In view of \eqref{Gnorm}, the $(p,1)$-norm of the terms in the last line have the right bound. With the other term we proceed as above, using the function $H=|\cdot|G$:
	\begin{align}&
	\left(|\cdot-\frac{1}{2}\left(a_i-b_i\right)| G\left(\cdot-\frac{1}{2}\left(a_i-b_i\right)\right)-|\cdot-\frac{1}{2}\left(b_i-a_i\right)| G\left(\cdot-\frac{1}{2}\left(b_i-a_i\right)\right)\right)
	\notag\\\leq& |a_i-b_i|\int_0^1\left|\nabla H\left(\cdot-\frac{1}{2}(b_i-a_i)+s(b_i-a_i)\right)\right|ds \notag\\
	\leq & |a_i-b_i|\int_0^1\left|\nabla_v H\left(\cdot-\frac{1}{2}(b_i-a_i)+s(b_i-a_i)\right)\right|ds \notag\\
	&+|a_i-b_i|\int_0^1\left|\nabla_x H\left(\cdot-\frac{1}{2}(b_i-a_i)+s(b_i-a_i)\right)\right|ds.
	\end{align}
	It follows from our estimates in \eqref{heat1} that
	\begin{align}
	&\left\||\cdot-\frac{1}{2}\left(a_i-b_i\right)| G\left(\cdot-\frac{1}{2}\left(a_i-b_i\right)\right)-|\cdot-\frac{1}{2}\left(b_i-a_i\right)| G\left(\cdot-\frac{1}{2}\left(b_i-a_i\right)\right)\right\|_{p,1}
	\\\leq & C|a_i-b_i|\left(\left\||\cdot|\nabla_v G\right\|_{p,1}+\left\||\cdot|\nabla_x G\right\|_{p,1}+\left\| G\right\|_{p,1}\right)\notag\\
	\leq&C_p|a_i-b_i|\left(t^{\frac{-7p+9}{2p}}+t^{\frac{-9p+9}{2p}}\right),
	\end{align}
	which leads to \eqref{heat2}.
\end{proof}

\begin{proof}[Proof of Lemma \ref{transition}]
	
	The proof of the estimates follows the ideas of  \cite[Lemma 2]{garcia2017}. 
	However, the evolution equation for the present system is more difficult to handle, and in particular, the spacial overlap is suppressed for short periods of time since we have a noise term in the momentum variable only. Both estimates can be proved in the same way. We just give the proof for the more difficult
	part $(ii)$, which can be easily
	adapted for part $(i)$. Without loss of generality we set $s=0$ and $t<1$. What we need to show then  is 
	$$\| u_{t, s}^{a, i,N} -
	u_{t, s}^{b,i, N} \|_{\infty,1} \leqslant C  | a_i - b_i|\left((t - s)^{- 6}+1\right)$$
	holds for all $i=1,\cdots,N.$
	
	Note that the force $\overline k_t^N(x):=k^N\ast \rho^N_t$ we consider is globally Lipschitz and $L^\infty$ because of \eqref{con2}, thus there exists a $C>0$ independent of $N$ such that
	\begin{equation}\label{lipschf}
	\max_{0\leq t\leq T;x,y\in\mathbb{R}^3}\frac{|\overline k_t(x)-\overline k_t(y)|}{|x-y|}\leq C\;.
	\end{equation}

	Let $c_t$ be the trajectory on phase space following the Newtonian equations of motion with respect to the force $\overline k_t^N$, starting with $\frac{1}{2}(a_i+b_i)$ at time $0$, i.e.
	$$c_t=(x^c_t,v^c_t),\hspace{1cm}\frac{d}{dt}x^c_t=v^c_t,\hspace{1cm}  \frac{d}{dt}v^c_t=\overline k_t^N(x^c_t),\hspace{1cm} c_0=\frac{1}{2}(a_i+b_i)\;.  $$
	
	We use the trajectory $c$ to change the frame of inertia that we use to look at $u_{t, s}^{d,i,N}$ for $d\in\{a,b\}$, i.e. we define for any $t>0$ the density $w_{t, 0}^{a,i,N}$ on phase space by
	\begin{equation}
	w_{t, 0}^{a,i,N}((x,v)):=u_{t, 0}^{a,i,N}((x,v)+c_t)\,.
	\end{equation}

	From the evolution equation of $u_{t, s}^{d,i,N}$ for $d\in\{a,b\}$ and $c_t$ one gets directly
	\begin{equation}\label{defw}
	\frac{\partial}{\partial t} w_{t, 0}^{d,i,N}\left(x,v\right):=\Delta_v w_{t, 0}^{d,i,N}\left(x,v\right)-\nabla_x w_{t, 0}^{d,i,N}\cdot v-\nabla_v w_{t, 0}^{d,i,N}\cdot \left(\overline k_t^N\left(x+x^c_t\right)-\overline k_t^N\left(x^c_t\right)\right),
	\end{equation}  
	with $w_{0, 0}^{a,i,N}=\delta\left(\cdot-\left(\frac{1}{2}\left(a_i-b_i\right)\right)\right)$ and $w_{0, 0}^{b,i,N}=\delta\left(\cdot-\left(\frac{1}{2}\left(b_i-a_i\right)\right)\right)$.
	
	Since $w$ is built from $u$ by  translation we have for any $1\leq p \leq\infty$
	\begin{equation}\label{diffu}
	\| u_{t, 0}^{a,i,N} -
	u_{t, 0}^{b,i,N} \|_{p,1}=\| w_{t, 0}^{a,i,N} -
	w_{t, 0}^{b,i,N} \|_{p,1}\;.
	\end{equation}
	
	Before proceeding we would like to explain the advantage of looking at $w$ instead of $u$  first  on a heuristic level.   The difficulties arise when dealing with short periods of  time.  There the $u^{d}$, $d\in\{a,b\}$ are roughly given by a Gaussian around the center at $\frac{1}{2}(a+b)$, respectively the $w^ {d}$ are roughly given by a Gaussian  around the center at $0$.  Here the force term of $w$ -- which is zero at $x=0$ -- suppresses the last term of \eqref{defw}. Thus $w $ will be very close to the  heat-kernel $G_t$ of our time evolution.
	
	Using \eqref{defw} and the properties of the heat kernel we get 
	\begin{align}\nonumber
	w_{t, 0}^{a,i,N}=& G_t\ast \delta\left(\cdot-\left(\frac{1}{2}\left(a_i-b_i\right)\right)\right) - \int_{0}^t G_{t-s}\ast  \left(\nabla_v w_{s, 0}^{a,i,N}\cdot \left(\overline k_s\left(\cdot+x^c_s\right)-\overline k_s\left(x^c_s\right)\right)\right) ds
	\\=& G_t\left(\cdot-\frac{1}{2}\left(a_i-b_i\right)\right) - \int_{0}^t \nabla_vG_{t-s}\ast\left( w_{s, 0}^{a,i,N} \left(\overline k_s\left(\cdot+x^c_s\right)-\overline k_s\left(x^c_s\right)\right)\right)  ds ,\label{wformula}
	\end{align}
	and
	\begin{align*}
	w_{t, 0}^{b,i,N}= G_t\left(\cdot-\frac{1}{2}\left(b_i-a_i\right)\right) - \int_{0}^t \nabla_vG_{t-s}\ast\left( w_{s, 0}^{b,i,N} \left(\overline k_s\left(\cdot+x^c_s\right)-\overline k_s\left(x^c_s\right)\right)\right)  ds,
	\end{align*}
	thus
	\begin{align}\label{diffw}
	w_{t, 0}^{a,i,N}-w_{t, 0}^{b,i,N}=&  \left(G_t\left(\cdot-\frac{1}{2}\left(a_i-b_i\right)\right)-G_t\left(\cdot-\frac{1}{2}\left(b_i-a_i\right)\right)\right) \\&- \int_{0}^t \nabla_vG_{t-s}\ast\left( \left( w_{s, 0}^{a,i,N}-w_{s, 0}^{b,i,N}\right) \left(\overline k_s\left(\cdot+x^c_s\right)-\overline k_s\left(x^c_s\right)\right)\right)  ds.\nonumber
	\end{align}
	
	Defining $\eta_{t, 0}^{ N}:\mathbb{R}^6\to\mathbb{R}^+_0$ by
	$\eta_{t, 0}^{ N}(x,v):=|(x,v)|\left|w_{t, 0}^{a,i,N}-w_{t, 0}^{b,i,N}\right|$ and using   \eqref{lipschf}, we can find a constant $C$ such that
	\begin{align}\label{etaform}
	\eta_{t, 0}^{ N}\leq & |\cdot| \left|G_t\left(\cdot-\frac{1}{2}\left(a_i-b_i\right)\right)-G_t\left(\cdot-\frac{1}{2}\left(b_i-a_i\right)\right)\right| \\&+  C \left|\int_{0}^t \nabla_v G_{t-s}\ast\eta_{s, 0}^{ N} ds\right|.\nonumber
	\end{align}
	Using the properties of the heat kernel \eqref{heat1}, \eqref{heat2} and Young's inequality  in \eqref{etaform}, we get
	\begin{align}
	\left\|\eta_{t, 0}^{ N}\right\|_{1,1}\leq & C|a_i-b_i|+  C  \int_{0}^t (t-s)^{-\frac{1}{2}}\left\|\eta_{s, 0}^{ N} \right\|_{1,1}ds.
	\end{align}
	Applying a generalized Gronwall's inequality with weak singularities \cite[Lemma 7.1.1]{henry2006geometric} leads to
	\begin{equation}\label{l1}
	\|\eta_{t, 0}^{N}\|_{1,1}\leq C |a_i-b_i| \hspace{1cm}\text{ uniform in }t\in[0,T]\;.
	\end{equation} 
	Further \eqref{etaform} gives for any $1\leq p\leq\infty$ and $t\in[0,T]$
	\begin{align}\label{use}
	\left\|\eta_{t, 0}^{N}\right\|_{p,1}\leq & \left\||\cdot|G_t\left(\cdot-\frac{1}{2}\left(a_i-b_i\right)\right)-G_t\left(\cdot-\frac{1}{2}\left(b_i-a_i\right)\right)\right\|_{p,1} \\&+  C\int_{0}^{t/2} \left\|\nabla_v G_{t-s}\ast\eta_{s, 0}^{ N}\right\|_{p,1}ds
	+  C \int_{t/2}^t  \left\|\nabla_v G_{t-s}\ast\eta_{s, 0}^{ N}\right\|_{p,1} ds.\nonumber
	\end{align}
	Using Young's inequality we get for $1+p^{-1}=\frac{9}{10}+q^{-1}$ and $t\in[0,T]$,
	\begin{align*}
	\left\|\eta_{t, 0}^{ N}\right\|_{p,1}\leq & C|a_i-b_i|t^{\frac{-9p+9}{2p}} +  C \int_{0}^{t/2} \left\|\nabla_v G_{t-s}\right\|_{p,1} \left\|\eta_{s, 0}^{ N}\right\|_{1,1} ds
	\\&+  C \int_{t/2}^{t} \left\|\nabla_v G_{t-s}\right\|_{10/9,1} \left\|\eta_{s, 0}^{N}\right\|_{q,1} ds.
	\end{align*}
	Due to \eqref{heat1}, one has $\left\|\nabla_v G_{t-s}\right\|_{10/9,1}\leq C (t-s)^{-19/20}$. This and \eqref{l1} give
	\begin{align}\label{iteration}
	\left\|\eta_{t, 0}^{ N}\right\|_{p,1}\leq  C|a_i-b_i|t^{\frac{-9p+9}{2p}} +  C|a_i-b_i| \int_{0}^{t/2} \left\|\nabla_v G_{t-s}\right\|_{p,1} ds
	+  C \max_{t/2\leq s \leq t}\left\|\eta_{s, 0}^{N}\right\|_{q,1}.
	\end{align}

	We use this formula starting at $p_1=1$ and setting $p_{k+1} =  \frac{10p_k}{10-p_k}$. Therefore, starting with our estimate for $ \left\|\eta_{t, 0}^{a,i,N}\right\|_{1,1}$ (see \eqref{l1}) we can then iteratively estimate the $L^p$ norms of $\eta_{t, 0}^{ N}$ for
	higher exponents, i.e.
	\begin{align}\label{iteration1}
	\left\|\eta_{t, 0}^{ N}\right\|_{p_{k+1},1}\leq  C|a_i-b_i|t^{\frac{-9p_{k+1}+9}{2p_{k+1}}} +  C|a_i-b_i| \int_{0}^{t/2} \left\|\nabla_v G_{t-s}\right\|_{p_{k+1},1} ds
	+  C \max_{t/2\leq s \leq t}\left\|\eta_{s, 0}^{N}\right\|_{p_{k},1}.
	\end{align}
	The exponent $p_{k+1} = \infty$ is attained  after $k = 10$ steps.  
	It follows that
	\begin{align}\label{etaest}
	\left\|\eta_{t, 0}^{ N}\right\|_{\infty,1}\leq  C|a_i-b_i|(t^{\frac{-9}{2}}+1)\;.
	\end{align}

	Having good control of $\|\eta_{t, 0}^{ N}\|_{\infty,1}$ we can now estimate 
	$w_{t, 0}^{a,i,N}-w_{t, 0}^{b,i,N}$ using \eqref{diffw}:
	\begin{align}
	\left\|w_{t, 0}^{a,i,N}-w_{t, 0}^{b,i,N}\right\|_{\infty,1}\leq&  \left\|G_t\left(\cdot-\frac{1}{2}\left(a_i-b_i\right)\right)-G_t\left(\cdot-\frac{1}{2}\left(b_i-a_i\right)\right)\right\|_{\infty,1} \\&+ \int_{0}^t \left\| \nabla_vG_{t-s}\ast\left( \left( w_{s, 0}^{a,i,N}-w_{s, 0}^{b,i,N}\right) \left(\overline k_s\left(\cdot+x^c_s\right)-\overline k_s\left(x^c_s\right)\right)\right)\right\|_{\infty,1}  ds\nonumber
	\\\leq&  C|a_i-b_i|t^{-6} +C\int_{0}^{t/2} \left\| \nabla_vG_{t-s}\right\|_{\infty,1} \left\|\eta_{s, 0}^{N}\right\|_{1,1} ds\nonumber
	\\&+\int_{t/2}^t \left\| \nabla_vG_{t-s}\right\|_{1,1} \left\|\eta_{s, 0}^{N}\right\|_{\infty,1} ds\nonumber
	\\\leq&   C|a_i-b_i|t^{-6}+C|a_i-b_i|\int_{0}^{t/2}(t-s)^{-5}ds+C|a_i-b_i|\int_{t/2}^t(t-s)^{-\frac{1}{2}} (s^{-\frac{9}{2}}+1)ds \notag\\
	\leq & C|a_i-b_i|\left(t^{-6}+t^{-4}+t^{\frac{1}{2}}\right)\leq C|a_i-b_i|(t^{-6}+1).
	\end{align}
	With \eqref{diffu} statement $(ii)$ of the lemma follows.

\end{proof}

\newpage

\bibliography{meanfield}

\begin{thebibliography}{10}

\bibitem{aarseth2003gravitational}
S.~J. Aarseth.
\newblock {\em Gravitational N-body simulations: tools and algorithms}.
\newblock Cambridge University Press, 2003.

\bibitem{BM}
J.~T. Beale and A.~Majda.
\newblock Vortex methods. \uppercase{i}: Convergence in three dimensions.
\newblock {\em Mathematics of Computation}, 39(159):1--27, 1982.

\bibitem{BM1}
J.~T. Beale and A.~Majda.
\newblock Vortex methods. \uppercase{ii}: Higher order accuracy in two and and
  three dimensions.
\newblock {\em Mathematics of Computation}, 39(159):29--52, 1982.

\bibitem{boers2016mean}
N.~Boers and P.~Pickl.
\newblock On mean field limits for dynamical systems.
\newblock {\em Journal of Statistical Physics}, 164(1):1--16, 2016.

\bibitem{bolley2011stochastic}
F.~Bolley, J.~A. Canizo, and J.~A. Carrillo.
\newblock Stochastic mean-field limit: {non-Lipschitz} forces and swarming.
\newblock {\em Mathematical Models and Methods in Applied Sciences},
  21(11):2179--2210, 2011.

\bibitem{bossy2015clarification}
M.~Bossy, O.~Faugeras, and D.~Talay.
\newblock Clarification and complement to ¡°mean-field description and
  propagation of chaos in networks of {Hodgkin-Huxley and Fitzhugh-Nagumo
  neurons}¡±.
\newblock {\em The Journal of Mathematical Neuroscience (JMN)}, 5(1):19, 2015.

\bibitem{BO}
F.~Bouchut.
\newblock Existence and uniqueness of a global smooth solution for the
  {Vlasov-Poisson-Fokker-Planck} system in three dimensions.
\newblock {\em Journal of functional analysis}, 111(1):239--258, 1993.

\bibitem{BO1}
F.~Bouchut.
\newblock Smoothing effect for the non-linear {Vlasov-Poisson-Fokker-Planck}
  system.
\newblock {\em Journal of differential equations}, 122(2):225--238, 1995.

\bibitem{braun1977vlasov}
W.~Braun and K.~Hepp.
\newblock The {Vlasov} dynamics and its fluctuations in the $\frac{1}{N}$ limit
  of interacting classical particles.
\newblock {\em Communications in mathematical physics}, 56(2):101--113, 1977.

\bibitem{garcia2017}
A.~Ca{\~n}izares-Garc{\'\i}a and P.~Pickl.
\newblock Microscopic derivation of the {Keller-Segel} equation in the
  sub-critical regime.
\newblock {\em arXiv preprint arXiv:1703.04376}, 2017.

\bibitem{carpio1998long}
A.~Carpio.
\newblock Long-time behaviour for solutions of the
  {Vlasov-Poisson-Fokker-Planck} equation.
\newblock {\em Mathematical methods in the applied sciences}, 21(11):985--1014,
  1998.

\bibitem{carrillo2018propagation}
J.~A. Carrillo, Y.-P. Choi, and S.~Salem.
\newblock Propagation of chaos for the {VPFP} equation with a polynomial
  cut-off.
\newblock {\em arXiv preprint arXiv:1802.01929}, 2018.

\bibitem{carrillo2010particle}
J.~A. Carrillo, M.~Fornasier, G.~Toscani, and F.~Vecil.
\newblock Particle, kinetic, and hydrodynamic models of swarming.
\newblock {\em Mathematical modeling of collective behavior in socio-economic
  and life sciences}, pages 297--336, 2010.

\bibitem{chorin1973numerical}
A.~J. Chorin.
\newblock Numerical study of slightly viscous flow.
\newblock {\em Journal of fluid mechanics}, 57(04):785--796, 1973.

\bibitem{degond}
P.~Degond.
\newblock Global existence of smooth solutions for the {Vlasov-Fokker-Planck}
  equation in $1$ and $2$ space dimensions.
\newblock In {\em Annales scientifiques de l'{\'E}cole Normale Sup{\'e}rieure},
  volume~19, pages 519--542, 1986.

\bibitem{dobrushin1979vlasov}
R.~L. Dobrushin.
\newblock {Vlasov} equations.
\newblock {\em Functional Analysis and Its Applications}, 13(2):115--123, 1979.

\bibitem{duerinckx2016mean}
M.~Duerinckx.
\newblock Mean-field limits for some {Riesz} interaction gradient flows.
\newblock {\em SIAM Journal on Mathematical Analysis}, 48(3):2269--2300, 2016.

\bibitem{fetecau2018propagation}
R.~C. Fetecau, H.~Huang, and W.~Sun.
\newblock Propagation of chaos for the {Keller-Segel} equation over bounded
  domains.
\newblock {\em Journal of Differential Equations}, 266(4):2142--2174, 2019.

\bibitem{fournier2015rate}
N.~Fournier and A.~Guillin.
\newblock On the rate of convergence in wasserstein distance of the empirical
  measure.
\newblock {\em Probability Theory and Related Fields}, 162(3-4):707--738, 2015.

\bibitem{fournier2014propagation}
N.~Fournier, M.~Hauray, and S.~Mischler.
\newblock Propagation of chaos for the 2d viscous vortex model.
\newblock {\em Journal of the European Mathematical Society}, 16(7):1423--1466,
  2014.

\bibitem{freedman1983brownian}
D.~Freedman.
\newblock {\em Brownian motion and diffusion}.
\newblock Springer Science \& Business Media, 1983.

\bibitem{Girsanov59}
I.~V. Girsanov.
\newblock Strongly-feller processes i. general properties.
\newblock {\em Theory Probab. Appl.}, 5(1):5--24, 1959.

\bibitem{GJ}
J.~Goodman.
\newblock Convergence of the random vortex method.
\newblock {\em Communications on Pure and Applied Mathematics}, 40(2):189--220,
  1987.

\bibitem{grass2019microscopic}
P.~Gra{\ss}.
\newblock {\em Microscopic derivation of Vlasov equations with singular
  potentials}.
\newblock PhD thesis, lmu, 2019.

\bibitem{hald}
O.~Hald and V.~M. Del~Prete.
\newblock Convergence of vortex methods for {Euler's} equations.
\newblock {\em Mathematics of Computation}, 32(143):791--809, 1978.

\bibitem{HOH}
O.~H. Hald.
\newblock Convergence of vortex methods for {Euler's} equations.
  \uppercase{ii}.
\newblock {\em SIAM Journal on Numerical Analysis}, 16(5):726--755, 1979.

\bibitem{havlak1996numerical}
K.~J. Havlak and H.~D. Victory, Jr.
\newblock The numerical analysis of random particle methods applied to
  {Vlasov-Poisson-Fokker-Planck} kinetic equations.
\newblock {\em SIAM journal on numerical analysis}, 33(1):291--317, 1996.

\bibitem{henry2006geometric}
D.~Henry.
\newblock {\em Geometric theory of semilinear parabolic equations}, volume 840.
\newblock Springer, 2006.

\bibitem{HH2}
H.~Huang and J.-G. Liu.
\newblock Discrete-in-time random particle blob method for the {Keller-Segel}
  equation and convergence analysis.
\newblock {\em Communication in Mathematical Sciences}, 15(7):1821--1842, 2017.

\bibitem{HH1}
H.~Huang and J.-G. Liu.
\newblock Error estimate of a random particle blob method for the
  {Keller-Segel} equation.
\newblock {\em Mathematics of Computation}, 86:2719--2744, 2017.

\bibitem{huilearning}
H.~Huang, J.-G. Liu, and J.~Lu.
\newblock Learning interacting particle systems: diffusion parameter estimation
  for aggregation equations.
\newblock {\em Mathematical Models and Methods in Applied Sciences},
  29(01):1--29, 2019.

\bibitem{jabin2014review}
P.-E. Jabin.
\newblock A review of the mean field limits for {Vlasov} equations.
\newblock {\em Kinet. Relat. Models}, 7(4):661--711, 2014.

\bibitem{jabin2015particles}
P.-E. Jabin and M.~Hauray.
\newblock Particles approximations of {Vlasov} equations with singular forces:
  {Propagation} of chaos.
\newblock In {\em Annales Scientifiques de l'{\'E}cole Normale Sup{\'e}rieure},
  2015.

\bibitem{jabin2004identification}
P.-E. Jabin and F.~Otto.
\newblock Identification of the dilute regime in particle sedimentation.
\newblock {\em Communications in mathematical physics}, 250(2):415--432, 2004.

\bibitem{jabin2016mean}
P.-E. Jabin and Z.~Wang.
\newblock Mean field limit and propagation of chaos for {Vlasov} systems with
  bounded forces.
\newblock {\em Journal of Functional Analysis}, 271(12):3588--3627, 2016.

\bibitem{jabin2017mean}
P.-E. Jabin and Z.~Wang.
\newblock Mean field limit for stochastic particle systems.
\newblock In {\em Active Particles, Volume 1}, pages 379--402. Springer, 2017.

\bibitem{jabin2018quantitative}
P.-E. Jabin and Z.~Wang.
\newblock Quantitative estimates of propagation of chaos for stochastic systems
  with $w^{-1,\infty}$ kernels.
\newblock {\em Inventiones mathematicae}, 214(1):523--591, 2018.

\bibitem{jeans1915theory}
J.~Jeans.
\newblock On the theory of star-streaming and the structure of the universe.
\newblock {\em Monthly Notices of the Royal Astronomical Society}, 76:70--84,
  1915.

\bibitem{keller1970initiation}
E.~F. Keller and L.~A. Segel.
\newblock Initiation of slime mold aggregation viewed as an instability.
\newblock {\em Journal of Theoretical Biology}, 26(3):399--415, 1970.

\bibitem{lazarovici2015mean}
D.~Lazarovici and P.~Pickl.
\newblock A mean field limit for the {Vlasov-Poisson} system.
\newblock {\em Archive for Rational Mechanics and Analysis}, 225:1201--1231,
  2017.

\bibitem{Levy40}
P.~L\'evy.
\newblock Sur certains processus stochastiques homog\'enes.
\newblock {\em Compos. Math.}, 7:283--339, 1940.

\bibitem{lions1991propagation}
P.-L. Lions and B.~Perthame.
\newblock Propagation of moments and regularity for the $3$-dimensional
  {Vlasov-Poisson} system.
\newblock {\em Inventiones mathematicae}, 105(1):415--430, 1991.

\bibitem{liu2016propagation}
J.-G. Liu and R.~Yang.
\newblock Propagation of chaos for large {Brownian} particle system with
  {Coulomb} interaction.
\newblock {\em Research in the Mathematical Sciences}, 3(1):40, 2016.

\bibitem{YZ}
J.-G. Liu and Y.~Zhang.
\newblock Convergence of diffusion-drift many particle systems in probability
  under {Sobolev} norm.
\newblock In P.~Gon\c{c}alves and A.~J. Soares, editors, {\em From Particle
  Systems to Partial Differential Equations-\uppercase{iii}}. Springer, 2016.

\bibitem{loeper2006uniqueness}
G.~Loeper.
\newblock Uniqueness of the solution to the {Vlasov-Poisson} system with
  bounded density.
\newblock {\em Journal de math{\'e}matiques pures et appliqu{\'e}es},
  86(1):68--79, 2006.

\bibitem{LD}
D.-G. Long.
\newblock Convergence of the random vortex method in two dimensions.
\newblock {\em Journal of the American Mathematical Society}, 1(4):779--804,
  1988.

\bibitem{MCPM}
C.~Marchioro and M.~Pulvirenti.
\newblock Vortex methods in two-dimensional fluid dynamics.
\newblock {\em Lecture notes in physics}, 203:1--137, 1984.

\bibitem{motsch2011new}
S.~Motsch and E.~Tadmor.
\newblock A new model for self-organized dynamics and its flocking behavior.
\newblock {\em Journal of Statistical Physics}, 144(5):923, 2011.

\bibitem{olla1991scaling}
S.~Olla and S.~Varadhan.
\newblock Scaling limit for interacting {Ornstein-Uhlenbeck} processes.
\newblock {\em Communications in mathematical physics}, 135(2):355--378, 1991.

\bibitem{ono2000regular}
K.~Ono and W.~A. Strauss.
\newblock Regular solutions of the {Vlasov-Poisson-Fokker-Planck} system.
\newblock {\em Discrete and Continuous Dynamical Systems}, 6(4):751--772, 2000.

\bibitem{OH}
H.~Osada.
\newblock Propagation of chaos for the two dimensional {Navier-Stokes}
  equation.
\newblock {\em Proceedings of the Japan Academy. Series A Mathematical
  sciences}, 62(1):8--11, 1986.

\bibitem{PCS}
C.~S. Patlak.
\newblock Random walk with persistence and external bias.
\newblock {\em The bulletin of mathematical biophysics}, 15(3):311--338, 1953.

\bibitem{pulvirenti2000infty}
M.~Pulvirenti and C.~Simeoni.
\newblock $l^\infty$-estimates for the {Vlasov-Poisson-Fokker-Planck} equation.
\newblock {\em Mathematical methods in the applied sciences}, 23(10):923--935,
  2000.

\bibitem{serfaty2018mean}
S.~Serfaty.
\newblock Mean field limit for {Coulomb} flows.
\newblock {\em arXiv preprint arXiv:1803.08345}, 2018.

\bibitem{soler1997asymptotic}
J.~Soler, J.~A. Carrillo, and L.~L. Bonilla.
\newblock Asymptotic behavior of an initial-boundary value problem for the
  {Vlasov-Poisson-Fokker-Planck} system.
\newblock {\em SIAM Journal on Applied Mathematics}, 57(5):1343--1372, 1997.

\bibitem{spohn2004dynamics}
H.~Spohn.
\newblock {\em Dynamics of charged particles and their radiation field}.
\newblock Cambridge university press, 2004.

\bibitem{touboul2014propagation}
J.~Touboul et~al.
\newblock Propagation of chaos in neural fields.
\newblock {\em The Annals of Applied Probability}, 24(3):1298--1328, 2014.

\bibitem{tremoulet2002hydrodynamic}
C.~Tremoulet.
\newblock Hydrodynamic limit for interacting {Ornstein-Uhlenbeck} particles.
\newblock {\em Stochastic processes and their applications}, 102(1):139--158,
  2002.

\bibitem{V-O-B}
H.~D. Victory and B.~P. O'Dwyer.
\newblock On classical solutions of {Vlasov-Poisson Fokker-Planck} systems.
\newblock {\em Indiana University mathematics journal}, 39(1):105--155, 1990.

\bibitem{villani2008optimal}
C.~Villani.
\newblock {\em Optimal transport: old and new}, volume 338.
\newblock Springer Science \& Business Media, 2008.

\bibitem{vlasov1968vibrational}
A.~Vlasov.
\newblock The vibrational properties of an electron gas.

\bibitem{wiener1923}
N.~Wiener.
\newblock Differential space.
\newblock {\em Journal of Mathematical Physics}, 2:131--174, 1923.

\bibitem{RY}
R.~Yang and J.-G. Liu.
\newblock Propagation of chaos for the {Keller-Segel} equation with a
  logarithmic cut-off.
\newblock preprint.

\end{thebibliography}
\bibliographystyle{abbrv}
\end{document}